\documentclass[12pt]{amsart}

\usepackage{amsthm,amsfonts,amsmath,verbatim,mathrsfs,amssymb,cite,fouridx}
\usepackage[usenames]{color}
\usepackage{todonotes}
\usepackage{hyperref}
\usepackage{breakurl}
\usepackage{url}
\usepackage[enableskew]{youngtab}

\newtheorem{theorem}{Theorem}[section]
\newtheorem{lemma}[theorem]{Lemma}
\newtheorem{proposition}[theorem]{Proposition}
\newtheorem{corollary}[theorem]{Corollary}

\theoremstyle{definition}

\theoremstyle{remark}

\numberwithin{equation}{section}

\newcommand{\Shat}{\widehat{\mathcal{S}}}
\newcommand{\Int}{\int\limits}
\newcommand{\floor}[1]{\lfloor{#1}\rfloor}
\newcommand{\ceiling}[1]{\lceil{#1}\rceil}
\newcommand{\qhyp}[2]{\fourIdx{}{#1}{}{#2}\phi}
\newcommand{\hyp}[2]{\fourIdx{}{#1}{}{#2}F}
\newcommand{\la}{\lambda}
\newcommand{\si}{\sigma}
\newcommand{\La}{\Lambda}
\newcommand{\Symm}{\mathfrak{S}}
\DeclareMathOperator{\sgn}{sgn}
\DeclareMathOperator{\SL}{SL} 
\DeclareMathOperator{\BC}{BC}
\DeclareMathOperator{\Sp}{Sp}
\DeclareMathOperator{\sL}{sl} 
\DeclareMathOperator{\so}{so}
\DeclareMathOperator{\ortho}{o}
\DeclareMathOperator{\symp}{sp}
\DeclareMathOperator*{\Pf}{Pf}
\newcommand{\D}{\mathrm D}

\renewcommand{\leq}{\leqslant}
\renewcommand{\geq}{\geqslant}

\newcommand{\abs}[1]{\lvert#1\rvert}
\newcommand{\Abs}[1]{\big\lvert#1\big\rvert}
\newcommand{\BigAbs}[1]{\Big\lvert#1\Big\rvert}
\DeclareMathOperator{\eup}{e}
\DeclareMathOperator{\dup}{d\hspace{-1.5pt}}

\newcommand{\R}{{\mathbb R}}

\newcommand{\Z}{\mathbb Z}

\newcommand{\Complex}{\mathbb C}
\newcommand{\qbin}[2]{\genfrac{[}{]}{0pt}{}{#1}{#2}}

\newcommand{\be}{\bar{1}}
\newcommand{\bt}{\bar{2}}
\newcommand{\bd}{\bar{3}}
\newcommand{\bv}{\bar{4}}
\newcommand{\bw}{\bar{5}}

\newcommand{\bil}[2]{\langle #1,#2\rangle}
\DeclareMathOperator{\ch}{ch}

\begin{document}

\title{Discrete analogues of Macdonald--Mehta integrals}

\author[Brent]{Richard P. Brent}
\address{Australian National University, Canberra, ACT 2600, Australia}
\email{Richard.Brent@anu.edu.au}

\author[Krattenthaler]{Christian Krattenthaler}
\address{Fakult\"{a}t f\"{u}r Mathematik, Universit\"{a}t at Wien, 
Oskar-Morgenstern-Platz 1, A-1090 Vienna, Austria}
\email{Christian.Krattenthaler@univie.ac.at}

\author[Warnaar]{S. Ole Warnaar}
\address{School of Mathematics and Physics,
The University of Queensland, Brisbane, QLD 4072, Australia}
\email{O.Warnaar@maths.uq.edu.au}

\thanks{R.P.B.\ is supported by the Australian Research Council
Discovery Grant DP140101417. \newline
\indent C.K.\ is partially supported by the Austrian
Science Foundation FWF, grant S50-N15,
in the framework of the Special Research Program
``Algorithmic and Enumerative Combinatorics.'' \newline
\indent 
S.O.W. is supported by the Australian Research Council Discovery Grant
DP110101234.}


\begin{abstract}
We consider discretisations of the Macdonald--Mehta integrals
from the theory of finite reflection groups. 
For the classical groups, $\mathrm{A}_{r-1}$, $\mathrm{B}_r$ and 
$\mathrm{D}_r$, we provide closed-form evaluations in those cases for 
which the Weyl denominators featuring in the summands have exponents $1$ 
and $2$.
Our proofs for the exponent-$1$ cases rely on identities for 
classical group characters, while most of 
the formulas for the exponent-$2$ cases 
are derived from a transformation formula for elliptic hypergeometric series
for the root system $\BC_r$.
As a byproduct of our results, we obtain closed-form product formulas
for the (ordinary and signed) enumeration of orthogonal and
symplectic tableaux contained in a box.
\end{abstract}

\maketitle

\section{Introduction}

Motivated by work in \cite{BOS} concerning the Hadamard maximal determinant
problem \cite{Hadamard83}, the recent papers \cite{BO,BOOPAA} considered 
various binomial multi-sum identities of which the following two 
results (the latter being conjectural in \cite{BOOPAA}) are representative:
\begin{multline}\label{Eq_BOOPAA} 
\sum_{i,j,k=-n}^n
\Abs{(i^2-j^2)(i^2-k^2)(j^2-k^2)}
\binom{2n}{n+i} \binom{2n}{n+j} \binom{2n}{n+k} \\
=3\cdot 2^{2n-1}n^3(n-1) \binom{2n}{n}^2
\end{multline}
and 
\begin{equation}\label{Eq_BOOPAA-2} 
\sum_{i,j=-n}^n
\Abs{ij(i^2-j^2)}
\binom{2n}{n+i}\binom{2n}{n+j} =
\frac{2n^3(n-1)}{2n-1} \binom{2n}{n}^2.
\end{equation}

Starting point for the current paper is the observation that these 
kinds of identities are reminiscent of multiple integral evaluations 
due to Macdonald and Mehta. To make this more precise, and to allow us 
to embed \eqref{Eq_BOOPAA} and \eqref{Eq_BOOPAA-2} into larger families of 
\emph{discrete analogues} of Macdonald--Mehta integrals,
we first review the continuous case.

Let $G$ be a finite reflection group consisting of $m$ reflecting 
hyperplanes $H_1,\dots,H_m$ in $\R^r$, see, e.g., \cite{Humphreys90}. 
Let $a_i\in\R^r$ be the normal of $H_i$ normalised up to sign such that
$\|a_i\|^2:=a_i\cdot a_i=2$.
For $x\in\R^r$ define the polynomial
\begin{equation}\label{Eq_Ppoly}
P(x)=P_G(x)=\prod_{i=1}^m (a_i\cdot x).
\end{equation}
In 1982 Macdonald \cite{Macdonald82} conjectured that 
\begin{equation}\label{Eq_Macdonald}
\Int_{\R^r} \abs{P(x)}^{2\gamma} \dup \varphi(x)=\prod_{i=1}^r 
\frac{\Gamma(1+d_i\gamma)}{\Gamma(1+\gamma)},
\end{equation}
where $\varphi(x)$ is the $r$-dimensional Gau\ss ian measure
\[
\dup\varphi(x)=\frac{\eup^{-\|x\|^2/2}}{(2\pi)^{r/2}}\,
\dup x_1\cdots\dup x_r,
\]
$d_1,\dots,d_r$ are the degrees of the fundamental invariants of $G$,
and $\text{Re}(\gamma)>-\min\{1/d_i\}$.
For $G=\mathrm{A}_{r-1}$ the integral \eqref{Eq_Macdonald} had appeared
as an earlier conjecture in work of Mehta and Dyson \cite{Mehta74,MD63} 
and is commonly referred to as Mehta's integral. 
It was first proved by Bombieri, who obtained it as a limit of the
Selberg integral \cite{Selberg44}, see \cite{FW08} for details.
For the two other classical series, 
$\mathrm{B}_r$ and $\mathrm{D}_r$, the conjecture also follows from the
Selberg integral, as was already noted in Macdonald's original 
paper.\footnote{Macdonald attributes this to A. Regev, unpublished.}
Complete proofs of Macdonald's conjecture were subsequently given in
\cite{Etingof10,Garvan89,Opdam89,Opdam93}.

The above-mentioned three classical series are of particular interest
to us here. For these, we have
\begin{subequations}\label{Eq_Poly-ABD}
\begin{gather}
P_{\mathrm{A}_{r-1}}(x)=\prod_{1\leq i<j\leq r} (x_i-x_j)=:\Delta(x), \\
P_{\mathrm{B}_r}(x)=2^{r/2}\prod_{i=1}^r x_i 
\prod_{1\leq i<j\leq r} (x_i^2-x_j^2),\quad\text{and}\quad
P_{\mathrm{D}_r}(x)=\prod_{1\leq i<j\leq r} (x_i^2-x_j^2),
\end{gather}
\end{subequations}
so that we can identify these cases of \eqref{Eq_Macdonald}
as the $(\alpha,\delta)=(1,0),(2,2\gamma),(2,0)$ instances of the 
\emph{Macdonald--Mehta integral}
\begin{equation}\label{Eq_Meh-Mac}
\mathcal{S}_r(\alpha,\gamma,\delta):=\int_{\R^r} 
\abs{\Delta(x^{\alpha})}^{2\gamma} \prod_{i=1}^r \abs{x_i}^{\delta} \dup
\varphi(x).
\end{equation}
It may now be recognised that \eqref{Eq_BOOPAA} and \eqref{Eq_BOOPAA-2} are
discrete analogues of the $\mathrm{D}_3$ and $\mathrm{B}_2$ 
Macdonald--Mehta integral for $\gamma=1/2$.
This suggests that one should study the more general binomial sums
\begin{equation}\label{Eq_Srdef} 
\mathcal{S}_{r,n}(\alpha,\gamma,\delta):= 
\sum_{k_1,\dots,k_r=-n}^n \abs{\Delta(k^\alpha)}^{2\gamma}\,
\prod_{i=1}^r \abs{k_i}^{\delta}\, \binom{2n}{n+k_i},
\end{equation}
where $n$ is a non-negative integer.
It is easy to show that \eqref{Eq_Srdef} is indeed a (scaled) discrete
approximation to \eqref{Eq_Meh-Mac} in the sense that
\[
\lim_{n\to\infty} 
2^{-2rn}\,(\tfrac{1}{2}n)^{-\alpha\gamma \binom{r}{2}-\delta r/2}
\mathcal{S}_{r,n}(\alpha,\gamma,\delta)=
\mathcal{S}_r(\alpha,\gamma,\delta).
\]
Using elements from representation theory and from the theory of
elliptic hypergeometric series, respectively,
we evaluate the \emph{discrete Macdonald--Mehta integral} \eqref{Eq_Srdef} 
for $\gamma=1/2$ and $\gamma=1$ and
$\alpha,\delta$ corresponding to $\mathrm{A}_{r-1}$, $\mathrm{B}_r$ and 
$\mathrm{D}_r$.
By the same methods we can evaluate four additional cases that do not 
appear to be related to reflection groups (or root systems), and the
total of ten evaluations is summarised in the following table:

\begin{center}
\begin{table}[h]	
\begin{tabular}{cccccc}
$\alpha$ & $\gamma$ & $\delta$ & $G$ \\
\hline
$1$ & $1/2$ & $0$ & $\mathrm{A}_{r-1}$ \\
$1$ & $1$ & $0, 1$ & $\mathrm{A}_{r-1}$, -- \\
$2$ & $1/2$ & $0, 1, 2$ & $\mathrm{D}_r$, $\mathrm{B}_r$, -- \\
$2$ & $1$ & $0, 1, 2, 3$ & $\mathrm{D}_r$, --, $\mathrm{B}_r$, --\\
\hline \\
\end{tabular}
\caption{The ten closed-form evaluations}\label{Table_ten}
\end{table}
\end{center}
All of these correspond to discrete analogues of the integrals
\[
\mathcal{S}_r(1,\gamma,0)=
\Int_{\R^r} \prod_{1\leq i<j\leq r} \abs{x_i-x_j}^{2\gamma} \, 
\dup \varphi(x) = \prod_{i=1}^r \frac{\Gamma(1+i\gamma)}{\Gamma(1+\gamma)}
\]
for $\text{Re}(\gamma)>-1/r$,
\begin{align*}
\mathcal{S}_r(1,1,1)&=
\Int_{\R^r} \prod_{1\leq i<j\leq r} \abs{x_i-x_j}^2 
\prod_{i=1}^r \, \abs{x_i} \, \dup \varphi(x) \\ &=
2^{r^2/2} \, \frac{\Gamma(1+r)}{\Gamma(\frac{1}{2})}
\prod_{i=1}^{\floor{\frac{1}{2}r}}
\frac{\Gamma(i)\Gamma(1+i)}{\Gamma(\frac{1}{2})}
\prod_{i=1}^{\ceiling{\frac{1}{2}r}-1}
\frac{\Gamma^2(1+i)}{\Gamma(\frac{1}{2})},
\end{align*}
and
\begin{align*}
\mathcal{S}_r(2,\gamma,\delta)&=
\Int_{\R^r} \prod_{1\leq i<j\leq r} \Abs{x_i^2-x_j^2}^{2\gamma} 
\prod_{i=1}^r \, \abs{x_i}^{\delta} \, \dup \varphi(x) \\ 
& = 2^{2\gamma \binom{r}{2}+\delta r/2}
\prod_{i=1}^r \frac{\Gamma(1+i\gamma)}{\Gamma(1+\gamma)}\cdot
\frac{\Gamma(\frac{1}{2}+(i-1)\gamma+\frac{1}{2}\delta)}{\Gamma(\frac{1}{2})}
\end{align*}
for $\text{Re}(\gamma)>-1/r$ and $\text{Re}(\delta/2+(r-1)\gamma)>-1/2$.
The first of these is the actual Mehta integral.
Also the last integral (which was also considered by Macdonald in
\cite{Macdonald82}) can easily be obtained as a limit of the Selberg 
integral by a generalisation of Regev's limiting procedure.

As a representative example of our results we state the closed-form evaluation 
of $\mathcal{S}_{r,n}(2,\frac{1}{2},0)$. 

\begin{proposition}[Discrete Macdonald--Mehta integral for $\mathrm{D}_r$]
\label{Prop_Bn}
Let $r$ be a positive integer and $n$ a non-negative integer. Then 
\begin{align}\label{Eq_Dr-sum0}
\mathcal{S}_{r,n}(2,\tfrac{1}{2},0)&=\sum_{k_1,\dots,k_r=-n}^n\,
\prod_{1\leq i<j\leq r} \Abs{k_i^2-k_j^2} \,
\prod_{i=1}^r\binom{2n}{n+k_i} \\
&=2^{2rn-r(r-1)}\, 
\frac{\Gamma(1+\frac{1}{2}r)}{\Gamma(\frac{3}{2})}\cdot
\frac{\Gamma(n-\frac{1}{2}r+\frac{3}{2})}{\Gamma(n+1)} \notag \\
&\qquad\qquad\quad\times 
\prod_{i=1}^{r-1} \frac{\Gamma(i+1)}{\Gamma(\frac{3}{2})} \cdot
\frac{\Gamma(2n+1)\, \Gamma({n}-i+\frac{3}{2})}
{\Gamma(2n-i+1)\,\Gamma({n}-i+1)}. \notag 
\end{align}
\end{proposition}

For $r=2$ this is \cite[Theorem 1]{BO}, for $r=3$ it is 
\eqref{Eq_BOOPAA} (first proved in \cite[Theorem~4.1]{BOOPAA})
and for $r=4$ this proves Conjecture~4.1 of that same paper. 
We further remark that both sides of \eqref{Eq_Dr-sum0} trivially vanish
unless $n\geq r-1$. Indeed, all $k_i^2$ need to be distinct for the summand
to be nonzero, requiring $n\geq r-1$. On the right the factor
$1/\Gamma(n-i+1)|_{i=r-1}$ is identically zero for $0\leq n\leq r-2$,
and the poles of $\prod_i \Gamma(n-\frac{1}{2}r+\frac{3}{2})/\Gamma(2n-i+1)$
at $n=0,1,\dots,(r-3)/2$ (these only arise for odd values of $r$) have zero 
residue.

\medskip
In several instances we obtain $q$-analogues and/or extensions 
to half-integer values of $n$ (in which case the $k_i$ need to be summed
over half-integers so that $n+k_i\in\Z$). 
Furthermore, when $\gamma=1$ we prove more general summations
containing additional free parameters, see
Sections~\ref{sec:gamma=1-2} and \ref{sec:gamma=1-1}.

As a byproduct of our proofs, we obtain some new results on the enumeration 
of tableaux. A particularly elegant example concerns 
\label{page_Sundaram-tableaux} \emph{Sundaram tableaux} \cite{Sundaram90}. 
These are semi-standard Young tableaux on the alphabet
$1<\bar{1}<2<\bar{2}<\cdots<n<\bar{n}<\infty$ such that all entries
in row $k$ are at least $k$ and with the exceptional rule that $\infty$ 
may occur multiple times in each column but at most once in each row.
We denote the size (or number of squares) of $T$ by $\abs{T}$ and
the number of occurrences of the letter $k$ by $m_k(T)$. Obviously,
$\sum_k m_k(T)=\abs{T}$ with $k$ summed over all $2n+1$ letters.
For example, 
\[
\young(1\be25\infty,2\bt3\infty,3\bd\bv\infty,\bw\bw)
\]
is a Sundaram tableau of size $15$ for all $n\geq 5$.
 
\begin{theorem}\label{Thm_Sundaram}
The number of Sundaram tableaux of height at most $n$ and width at most
$r$ is given by
\[
\prod_{i=1}^n \frac{2i+r-1}{2i-1}
\prod_{i,j=1}^n \frac{i+j+r-1}{i+j-1}.
\]
Similarly, the number of Sundaram tableaux of height at most $n$ and 
width at most $r$ such that each tableaux is given a weight
$(-1)^{\abs{T}}$ 
\textup{(}resp. $(-1)^{m_{\infty}(T)}$\textup{)} is given by
\[
(-1)^{rn} \prod_{i,j=1}^n \frac{i+j+r-1}{i+j-1}
\qquad
\bigg(\text{resp.}\;\: \prod_{i,j=1}^n \frac{i+j+r-1}{i+j-1} \: \bigg).
\]
\end{theorem}

For example, when $r=n=2$, there are 
$(3\cdot 5\cdot 3\cdot 4\cdot 4\cdot 5)/
(1\cdot 3\cdot 1\cdot 2\cdot 2\cdot 3)=
100$ tableaux, with the following break-down according to shape

\medskip

\begin{center}
\begin{tikzpicture}[scale=0.3]
\draw[fill] (0,-0.5) circle (0.08cm);
\draw (1.5,0) rectangle (2.5,-1);
\draw (4,0) rectangle (5,-2);
\draw (4,-1) -- (5,-1);
\draw (6.5,0) rectangle (8.5,-1);
\draw (7.5,0) -- (7.5,-1);
\draw (10,0) -- (12,0) -- (12,-1) -- (11,-1) -- (11,-2) -- (10,-2) -- cycle;
\draw (10,-1) -- (11,-1) -- (11,0);
\draw (13.5,0) rectangle (15.5,-2);
\draw (14.5,0) -- (14.5,-2);
\draw (13.5,-1) -- (15.5,-1);
\draw (0,-3) node {$1$};
\draw (2,-3) node {$5$};
\draw (4.5,-3) node {$10$};
\draw (7.5,-3) node {$14$};
\draw (11,-3) node {$35$};
\draw (14.5,-3) node {$35$};
\end{tikzpicture}
\end{center}

\noindent
or according to the multiplicities $m_{\infty}(T)$: 
\[
\Abs{\big\{T:m_{\infty}(T)=0\big\}}=50,\;
\Abs{\big\{T:m_{\infty}(T)=1\big\}}=40,\;
\Abs{\big\{T:m_{\infty}(T)=2\big\}}=10.
\]
Moreover, $1-5+10+14-35+35=20=(3\cdot 4\cdot 4\cdot 5)/
(1\cdot 2\cdot 2\cdot 3)$, and also $50-40+10=20$.

\bigskip
Our paper is organised as follows. 
The next, short section summarises the ten key evaluations corresponding 
to the binomial sums of Macdonald--Mehta-type listed in Table~\ref{Table_ten}.
Section~\ref{Sec_Weyl} reviews some standard material concerning 
classical group characters needed in our subsequent computations.
Section~\ref{sec:Okada} deals with summation identities for orthogonal and 
symplectic characters. Although several such identities were derived 
previously by Okada \cite{Okada98}, his results are not sufficient for our 
purposes, and more refined identities as well as identities in which 
the summands have alternating signs are added to Okada's list.
In Section~\ref{sec:gamma=1/2} we then apply the results from
Section~\ref{sec:Okada} to evaluate the sums
$\mathcal{S}_{r,n}(\alpha,\frac{1}{2},\delta)$ claimed in 
Section~\ref{sec:summary}.
In most cases, we are able to also provide $q$-analogues. 
Our evaluations of $\mathcal{S}_{r,n}(\alpha,1,\delta)$ given in
Section~\ref{sec:summary} are dealt with in
Sections~\ref{sec:gamma=1-2} and \ref{sec:gamma=1-1}. 
All these evaluations result from a single identity, a transformation 
formula between multiple elliptic hypergeometric series originally 
conjectured by the third author \cite[Conj.~6.1]{WarnAG}, and proven 
independently by Rains~\cite[Theorem~4.9]{RainAA} and by Coskun and 
Gustafson~\cite{CoGuAA}. We do not present this formula in its full generality
here, but restrict ourselves to stating the relevant \mbox{($q$-)special} 
case in Theorem~\ref{thm:Rains} at the beginning of 
Section~\ref{sec:gamma=1-2}. The remainder of that section is devoted to
proving our evaluations of the sums $\mathcal{S}_{r,n}(2,1,\delta)$,
while Section~\ref{sec:gamma=1-1} is devoted to proving the evaluations of 
the sums $\mathcal{S}_{r,n}(1,1,\delta)$. In all cases but one, we provide
$q$-analogues which actually contain an additional parameter.
The only exception is the sum $\mathcal{S}_{r,n}(1,1,1)$, where we
are ``only'' able to establish a summation containing an additional
parameter (see Proposition~\ref{Prop_S121}), but for which we were
not able to find a $q$-analogue. Moreover, in this case we needed
to take recourse to an ad hoc approach, since we could not figure
out a way to use the aforementioned transformation formula.
The final section, Section~\ref{sec:disc}, discusses some further
aspects of the work presented in this article, open problems,
and (possible) further avenues.

\medskip
To conclude the introduction, we point out two further articles addressing
the multi-sums in \cite{BOOPAA}. First, in \cite{KrScAC} the double sums
considered in \cite{BOOPAA} are embedded into a three-parameter 
family of double sums, and it is shown that all of them can be explicitly
computed by using complex contour integrals or by the use of the 
computer algebra package \textsl{Sigma} \cite{SchnAA}, thus proving in 
particular all the respective conjectures in \cite{BOOPAA}, including
\eqref{Eq_BOOPAA-2}. 
Second, Bostan, Lairez and Salvy \cite{BoLSAA} recently presented an 
algorithmic approach to finding recurrences for multiple binomial sums of 
the type considered in this paper. Interestingly, complex contour integrals
are again instrumental in this approach. Among other things, 
it allowed them to prove automatically all the double-sum identities 
from \cite{BOOPAA}, again including all the conjectures from \cite{BOOPAA},
such as \eqref{Eq_BOOPAA-2}. 
Moreover, their algorithmic approach is --- in principle ---
capable of proving any of our $r$-fold sum identities for \emph{fixed}~$r$.
(As usual, ``in principle'' refers to the fact that today's computers
may not actually be able to finish the required computations.)
To come up with an automatic proof for any of our identities for 
\emph{generic}~$r$ seems however to be
currently out of reach.

\subsection*{Acknowledgements}
We thank Peter Forrester, Ron King, Soichi Okada and Helmut Prodinger for 
helpful discussions.
The second author also gratefully acknowledges the Galileo Galilei Institute 
of Theoretical Physics in Firenze, Italy, and to the National Institute of 
Mathematical Sciences, Daejeon, South Korea for the inspiring environments 
during his visits in June/July 2015, when most of this work was carried out. 

\section{Summary of the ten primary identities}\label{sec:summary}

Here we summarise as succinctly as possible the ten product formulas 
for the discrete Macdonald--Mehta integral 
$\mathcal{S}_{r,n}(\alpha,\gamma,\delta)$ (defined in \eqref{Eq_Srdef}), 
corresponding to the parameter choices listed in Table~\ref{Table_ten}.
Proofs and further generalisations are given in 
Sections~\ref{sec:Okada}--\ref{sec:gamma=1-1}.

For $\alpha=2$ there are a total of seven cases, given by
\begin{multline}\label{eq:2anyanygp4}
\mathcal{S}_{r,n}(2,\gamma,\delta) \\
=\prod_{i=1}^r
\frac{\Gamma(1+i\gamma)}{\Gamma(1+\gamma)}\cdot
\frac{\Gamma(2n+1)\,\Gamma(n-i-\gamma+\chi+2)\,
\Gamma((i-1)\gamma+\frac{\delta+1}{2})}
{\Gamma(n-i+\chi+1)\,\Gamma(n-i\gamma+\chi+1)\,
\Gamma(n-(i-1)\gamma-\frac{\delta-3}{2}-\chi)},
\end{multline}
where $\chi=1$ if $\delta=0$, and $\chi=0$ otherwise.
For $\alpha=1$ and $\delta=0$ there are two cases, given by
\begin{equation}\label{eq:1any0}
\mathcal{S}_{r,n}(1,\gamma,0)=
2^{2rn-\gamma r(r-1)}
\prod_{i=1}^r \frac{\Gamma(1+i\gamma)}{\Gamma(1+\gamma)}\cdot
\frac{\Gamma(2n+1)\,\Gamma(2n-i+\gamma+2)}
{\Gamma(2n-(i-2)\gamma+1)\,\Gamma(2n-i+2)}.
\end{equation}
(This formula remains valid if $\gamma=0$ or $n$ is a half-integer.)

The remaining case is
\begin{equation}\label{eq:121}
\mathcal{S}_{r,n}(1,1,1) = r!  \prod_{i=1}^{\ceiling{r/2}}
\frac{\Gamma^2(i)\,\Gamma(2n+1)}{\Gamma(n-i+1)\,\Gamma(n-i+2)}
\prod_{i=1}^{\floor{r/2}} \frac{\Gamma(i)\,\Gamma(i+1)\,\Gamma(2n+1)}
{\Gamma^2(n-i+1)}.
\end{equation}

\section{The Weyl character formula and Schur functions of type $G$}
\label{Sec_Weyl}

The purpose of this section is to collect standard material on
classical group characters that we use in
Sections~\ref{sec:Okada} and \ref{sec:gamma=1/2}.

\subsection{Some simple $q$-functions} \label{subsec:simple_q}

Assume that $0<q<1$ and $m,n$ are integers such that $0\leq m\leq n$.
Then the \emph{$q$-shifted factorial}, \emph{$q$-binomial coefficient},
\emph{$q$-gamma function} and \emph{$q$-factorial\/} are given by
\begin{gather*}
(a;q)_n=\prod_{k=1}^n (1-aq^{k-1}),\qquad
(a;q)_{\infty}=\prod_{k=1}^{\infty} (1-aq^{k-1}) \\
\qbin{n}{m}=\qbin{n}{m}_q=\frac{(q^{n-m+1};q)_m}{(q;q)_m} \\
\Gamma_q(x)=(1-q)^{1-x}\,\frac{(q;q)_{\infty}}{(q^x;q)_{\infty}} \\
[n]_q=\frac {1-q^n} {1-q},\qquad 
[n]_q!=\Gamma_q(n+1)=[n]_q\,[n-1]_q\cdots[1]_q.
\end{gather*}

We also need some generalisations of the $q$-shifted factorials to 
partitions. We use standard terminology for partitions, as for
example found in \cite[Chapter~1]{Macdonald95}, More precisely,
let $\la$ be a \emph{partition}, that is, $\la=(\la_1,\la_2,\dots)$ is a 
weakly decreasing sequence of non-negative integers with only finitely 
many non-zero $\la_i$. The positive $\la_i$  are called the \emph{parts} of
$\la$ and the number of parts is called the \emph{length} of the partition,
denoted by $l(\la)$. As usual we identify a partition with its
(Young) diagram,
and the \emph{conjugate partition} $\la'$ is the partition
obtained by reflecting the diagram in the main diagonal.
We shall frequently need partitions of rectangular shape.
By definition, this is a partition all of whose parts are the same.
In order to have a convenient notation, we write $(r^n)$ for the
partition $(r,r,\dots,r)$ with $n$ occurrences of $r$.
If $\la$ is a partition of length at most $n$ and largest part at most
$r$, we use the suggestive notation 
$\la\subseteq (r^n)$. Clearly this is equivalent to
$\la'\subseteq (n^r)$. 
We say that \emph{$(i,j)$ is a square (in the diagram) of $\la$}
and write $(i,j)\in\la$ if and only if $1\leq i\leq l(\la)$
and $1\leq j\leq \la_i$.
Following \cite{Rains05}, we now define
\begin{subequations}\label{Eq_C-def}
\begin{align}
C_{\la}^{-}(a;q)&=\prod_{(i,j)\in\la} (1-aq^{\la_i+\la'_j-i-j}) \\
C_{\la}^{+}(a;q)&=\prod_{(i,j)\in\la} (1-aq^{\la_i-\la'_j+j-i+1}) \\
C_{\la}^{0}(a;q)&=\prod_{(i,j)\in\la} (1-aq^{j-i}).
\end{align}
\end{subequations}
Expressed in terms of ordinary $q$-binomial coefficients we have
\begin{subequations}\label{Eq_C-explicit}
\begin{align}
C_{\la}^{-}(a;q)&=\prod_{i=1}^n (aq^{n-i};q)_{\la_i}
\prod_{1\leq i<j\leq n} \frac{1-aq^{j-i-1}}{1-aq^{\la_i-\la_j+j-i-1}} \\
C_{\la}^{+}(a;q)&=\prod_{i=1}^n \frac{(aq^{2-2i};q)_{2\la_i}}
{(aq^{2-i-n};q)_{\la_i}}
\prod_{1\leq i<j\leq n} \frac{1-aq^{2-i-j}}{1-aq^{\la_i+\la_j-i-j+2}} \\
C_{\la}^{0}(a;q)&=\prod_{i=1}^n (aq^{1-i};q)_{\la_i},
\end{align}
\end{subequations}
where $n$ is an arbitrary integer such that $l(\la)\leq n$.
Since conjugation simply interchanges rows and columns of a partition, it
follows readily from \eqref{Eq_C-def} that 
\begin{subequations}\label{Eq_C-conjugation}
\begin{align}
C_{\la'}^{-}(a;q)&=C_{\la}^{-}(a;q) \\
C_{\la'}^{+}(a;q)&=(-aq)^{\abs{\la}} q^{3n(\la)-3n(\la')}
C_{\la}^{+}\big(a^{-1}q^{-2};q\big) \\
C_{\la'}^{0}(a;q)&=(-a)^{\abs{\la}} q^{n(\la)-n(\la')}
C_{\la}^{0}\big(a^{-1};q\big),
\end{align}
\end{subequations}
where $\abs{\la}:=\la_1+\la_2+\cdots$ and 
$n(\la):=\sum_{i\geq 1} (i-1)\la_i=\sum_{i\geq 1} \binom{\la'_i}{2}$.

\subsection{The Weyl character and dimension formulas}

Let $\mathfrak{g}$ be a complex semi\-simple Lie algebra of rank $r$, 
$\mathfrak{h}$ and $\mathfrak{h}^{\ast}$ the Cartan subalgebra and its dual,
and $\Phi$ the root system spanning $\mathfrak{h}^{\ast}$ with basis of simple
roots $\{\alpha_1,\dots,\alpha_r\}$, see e.g., \cite{Bourbaki02,Humphreys78}. 
Let $\bil{\cdot}{\cdot}$ denote the usual symmetric bilinear
form on $\mathfrak{h}^{\ast}$, and assume the standard identification
of $\mathfrak{h}$ and $\mathfrak{h}^{\ast}$ through the Killing form so that
the coroots are given by
\[
\alpha^{\vee}=\frac{2\alpha}{\bil{\alpha}{\alpha}}=
\frac{2\alpha}{\|\alpha\|^2}.
\]
Let $\omega_1,\dots,\omega_r$ be the fundamental weights, i.e.,
$\bil{\omega_i}{\alpha^{\vee}_j}=\delta_{ij}$, and denote 
the root lattice 
$\Z\alpha_1\oplus\cdots\oplus\Z\alpha_r$
and weight lattice 
$\Z\omega_1\oplus\cdots\oplus\Z\omega_r$ by $Q$ and $P$, respectively.
Further, let $P_{+}$ be the set of dominant (integral) weights,
\[
P_{+}=\big\{\la\in P:~\bil{\la}{\alpha^{\vee}_i}\geq 0 
\text{ for $1\leq i\leq r$}\big\},
\]
and set
\[
Q_{+}=\big\{\alpha\in Q:~\bil{\alpha^{\vee}}{\omega_i}\geq 0 
\text{ for $1\leq i\leq r$}\big\}.
\]
We also denote the set of positive roots by $\Phi_{+}$, so that
$\Phi_{+}=Q_{+}\cap \Phi$.

The irreducible highest weight modules $V(\la)$ of
$\mathfrak{g}$ are indexed by dominant weights~$\lambda$. The
characters corresponding to these modules are defined as
\[
\ch V(\la):=\sum_{\mu\in\mathfrak{h}^{\ast}} \dim(V_{\mu}) \eup^{\mu},
\]
where the $V_{\mu}$ are the weight spaces in the weight-space
decomposition of $V(\la)$ and $\eup^{\la}$ for $\la\in P$ is a formal
exponential satisfying $\eup^{\la}\eup^{\mu}=\eup^{\la+\mu}$.
It is a well-known fact that 
$\dim(V_{\la})=1$ and $\dim(V_{\mu})=0$ if $\la-\mu\not\in Q_{+}$.
The characters can be computed explicitly using the Weyl character formula
\begin{equation}\label{Eq_Weyl-char}
\ch V(\la)=\frac{\sum_{w\in W}\sgn(w) \eup^{w(\la+\rho)-\rho}}
{\prod_{\alpha>0}(1-\eup^{-\alpha})}.
\end{equation}
Here, $W$ is the Weyl group of $\mathfrak{g}$, $\alpha>0$ is 
shorthand for $\alpha\in\Phi_{+}$, and 
$\rho=\frac{1}{2}\sum_{\alpha>0}\alpha=\sum_{i=1}^r \omega_i$
is the Weyl vector. For $\la=0$, Weyl's formula simplifies to the 
denominator identity
\begin{equation}\label{Eq_denom}
\sum_{w\in W}\sgn(w) \eup^{w(\rho)-\rho}=\prod_{\alpha>0}(1-\eup^{-\alpha}).
\end{equation}

The dimension of the highest weight module $V(\la)$ follows from the
Weyl character formula by applying the map $\eup^{\la}\mapsto 1$.
We will require two slightly more general specialisations resulting
in $q$-dimension formulas. Let $s$ be the squared length of the short 
roots in $\Phi$ and define $F$ and $F^{\vee}$ by
\begin{align*}
F:&~\Z[\eup^{-\alpha_0},\dots,\eup^{-\alpha_r}]\to \Z[q^s], &
F(\eup^{-\alpha_i})&=q^{\bil{\rho}{\alpha_i}} \\
F^{\vee}:&~\Z[\eup^{-\alpha_0},\dots,\eup^{-\alpha_r}]\to \Z[q], &
F^{\vee}(\eup^{-\alpha_i})&=q^{\bil{\rho}{\alpha_i^{\vee}}}=q
\end{align*}
for all $i$ with $1\leq i\leq r$.
By defining the $q$-dimensions by
\[
\dim_q V(\la):=F\big(\eup^{-\la}\ch V(\la)\big) \quad\text{and}\quad
\dim_q^{\vee}V(\la):=F^{\vee} \big(\eup^{-\la}\ch V(\la)\big),
\]
we have the following pair of dimension formulas.
\begin{lemma}\label{Lem_qdim}
We have
\begin{subequations}\label{Eq_qdim}
\begin{align}\label{Eq_qdim1}
\dim_q V(\la)&=\prod_{\alpha>0} 
\frac{1-q^{\bil{\la+\rho}{\alpha}}}
{1-q^{\bil{\rho}{\alpha}}} ,\\
\dim_q^{\vee} V(\la)&=\prod_{\alpha>0} 
\frac{1-q^{\bil{\la+\rho}{\alpha^{\vee}}}}
{1-q^{\bil{\rho}{\alpha^{\vee}}}}.
\label{Eq_qdim2}
\end{align}
\end{subequations}
\end{lemma}

In the $q\to 1$ limit, \eqref{Eq_qdim1} implies the Weyl dimension formula
\[
\dim V(\la)=
\prod_{\alpha>0} \frac{\bil{\la+\rho}{\alpha}}{\bil{\rho}{\alpha}}.
\]

\begin{proof}
Applying $F$ to 
$\eup^{-\la} \ch V(\la)\in \Z[\eup^{-\alpha_1},\dots,\eup^{-\alpha_r}]$
and using \eqref{Eq_Weyl-char}, we obtain 
\[
\dim_q V(\la)=
\frac{\sum_{w\in W}\sgn(w) q^{-\bil{\rho}{w(\la+\rho)-\la-\rho}}}
{\prod_{\alpha>0}(1-q^{\bil{\rho}{\alpha}})}.
\]
Since $\bil{\rho}{w(\la+\rho)}=\bil{w^{-1}(\rho)}{\la+\rho}$ and
$\sgn(w) =\sgn(w^{-1})$, a change of the summation index from 
$w$ to $w^{-1}$ results in
\[
\dim_q V(\la)=
\frac{\sum_{w\in W}\sgn(w) q^{-\bil{w(\rho)-\rho}{\la+\rho}}}
{\prod_{\alpha>0}(1-q^{\bil{\rho}{\alpha}})}.
\]
The first claim now follows from the denominator formula 
\eqref{Eq_denom} with $\eup^{-u}\mapsto q^{-\bil{u}{\la+\rho}}$.

The proof of \eqref{Eq_qdim2} is nearly identical and is left to the reader.
\end{proof}

In the next four subsections we restrict the Weyl character and dimension
formulas to the four classical types and give ``dual'' forms
for the $q$-dimension formulas needed in our proofs of the discrete
Macdonald--Mehta integrals.

\subsection{The Schur functions}

For $x=(x_1,\dots,x_n)$ and $\la$ a partition of length at most $n$,
the \emph{Schur function} $s_{\la}(x)$ is defined by
\begin{equation}\label{Eq_Schur}
s_{\la}(x):=
\frac{\det_{1\leq i,j\leq n}(x_i^{\la_j+n-j})}
{\det_{1\leq i,j\leq n}(x_i^{n-j})}.
\end{equation}
If $\La_n=\Z[x_1,\dots,x_n]^{\Symm_n}$ denotes the ring of symmetric 
functions in $n$ variables, then the Schur functions indexed by partitions
of length at most $n$ form a basis of $\La_n$.
The Schur functions have a simple interpretation in terms of 
the representation theory of the symmetric group $\Symm_n$ and the general
linear group $\mathrm{GL}_n(\Complex)$. More precisely, 
they are exactly the characters of the irreducible (polynomial) 
representations of $\mathrm{GL}_n(\Complex)$.
The representation theory of $\mathrm{SL}_n(\Complex)$ is almost identical
to that of $\mathrm{GL}_n(\Complex)$, the only notable difference being 
that in the former irreducible representations are indexed by partitions 
of length at most $n-1$, and to interpret such $s_{\la}(x)$ as a character
we should impose the restriction $x_1\cdots x_n=1$.
Since the Schur function $s_{\la}(x)$ is homogeneous of degree $\la$ and
satisfies
\[
s_{\la}(x)=(x_1\cdots x_n)^{\la_n}
s_{(\la_1-\la_n,\dots,\la_{n-1}-\la_n,0)}(x),
\]
these differences do not affect any of the underlying combinatorics.
In particular, if $\mathfrak{g}$ is the Lie algebra $\sL_n(\Complex)$ 
and $\phi$ the ring isomorphism
\begin{subequations}
\begin{gather}
\phi:\Z\big[\eup^{\la}:\la\in P\big]^W\to 
\Z\big[x_1,\dots,x_{n-1},x_1^{-1}\cdots x_{n-1}^{-1}\big]^{\Symm_n}=\La'_n \\
\label{Eq_omega-x}
\phi(\eup^{\omega_i})
=x_1\cdots x_i \quad\text{for $1\leq i\leq n-1$},
\end{gather}
\end{subequations}
then
\begin{equation}\label{Eq_phiVs}
\phi\big(\ch V(\la)\big)=s_{\la}(x)|_{x_n=x_1^{-1}\cdots x_{n-1}^{-1}},
\end{equation}
where on the left $\la$ is a dominant weight parametrised as
\begin{equation}\label{Eq_la-omega}
\la=(\la_1-\la_2)\omega_1+\cdots+(\la_{n-2}-\la_{n-1})\omega_{n-2}
+\la_{n-1}\omega_{n-1}
\end{equation}
and on the right $\la$ is the partition $(\la_1,\dots,\la_{n-1},0)$.

Instead of using the ratio of determinants given in \eqref{Eq_Schur},
we can compute the Schur function in a more combinatorial fashion
using semi-standard Young tab\-leaux. Namely,
\begin{equation}\label{Eq_Schur-tab}
s_{\la}(x)=\sum_T x^T,
\end{equation}
where the sum is over all semi-standard Young
tableaux $T$ of shape $\la$ on the alphabet
$1<2<\cdots<n$ and $x^T:=x_1^{m_1(T)}\cdots x_n^{m_n(T)}$.

From Lemma~\ref{Lem_qdim} and equation \eqref{Eq_phiVs},
it follows that for $l(\la)\leq n$ we have the principal specialisation
formula
\begin{equation}\label{Eq_Schur-PS}
s_{\la}(1,q,\dots,q^{n-1})=
q^{n(\la)} \prod_{1\leq i<j\leq n} 
\frac{1-q^{\la_i-\la_j+j-i}}{1-q^{j-i}}.
\end{equation}
Indeed, since the above only depends on differences between the parts of
$\la$, we may assume without loss of generality that $\la_n=0$.
Since the set of positive roots is given by 
\[
\{\alpha_i+\cdots+\alpha_j:~1\leq i\leq j\leq n-1\},
\]
it follows that for $\la\in P_{+}$ parametrised by \eqref{Eq_la-omega}
we have
\begin{equation}\label{Eq_dimqsl}
\dim_q V(\la)=\dim^{\vee}_q V(\la)=
\prod_{1\leq i\leq j\leq n-1}\frac{1-q^{\la_i-\la_{j+1}+j-i+1}}{1-q^{j-i+1}}.
\end{equation}
Since $F(\eup^{-\omega_i})=q^{i(n-i)/2}$, it follows from
\eqref{Eq_omega-x} that under the induced action of $F$ on $\La_n'$
we have
\[
F(x_i)=q^{i-(n+1)/2} \quad\text{for $1\leq i\leq n-1$}.
\]
We also have $F(\eup^{-\la})=q^{(n-1)\abs{\la}/2-n(\la)}$, where on the right 
$\la$ is the partition corresponding to $\la\in P_{+}$ on the left.
Hence,
\begin{align*}
s_{\la}(1,q,\dots,q^{n-1})&=
q^{(n-1)\abs{\la}/2} 
s_{\la}\big(q^{-(n-1)/2},q^{-(n-3)/2},\dots,q^{(n-1)/2}\big) \\
&=q^{(n-1)\abs{\la}/2} F\big(s_{\la}(x)\big) \\
&=q^{n(\la)} F(\eup^{-\la} \ch V(\la))
=q^{n(\la)} \dim_q V(\la),
\end{align*}
which by \eqref{Eq_dimqsl} implies \eqref{Eq_Schur-PS}.
All of the above is well-known, although rarely made explicit.
Since later we want to refer to analogous results for other groups without
spelling out the (less well-known) details, we have included
the full details of the Schur function case.
We also note that each of the principal specialisation formulas for 
the classical groups has a dual form obtained by using conjugate partitions. 
These dual forms will be crucial later.

\begin{lemma}[Principal specialisation --- dual form]\label{Lem_Schur_PSprime}
For $\la\subseteq (r^n)$, we have
\[
s_{\la}(1,q,\dots,q^{n-1})=q^{n(\la)}
\prod_{i=1}^r \qbin{n+r-1}{\la'_i+r-i}
\qbin{n+r-1}{r-i}^{-1} \prod_{1\leq i<j\leq r}
\frac{1-q^{\la_i'-\la_j'+j-i}}{1-q^{j-i}}.
\]
\end{lemma}

\begin{proof}
Perhaps the most elegant proof is to use the \emph{dual Jacobi--Trudi identity} 
\cite[p.~41]{Macdonald95} and the principal specialisation formula for 
the \emph{elementary symmetric functions} \cite[p.~27]{Macdonald95}, combined 
with the the determinant evaluation \cite[Theorem~26]{Krattenthaler99}.

In view of the other types yet to be discussed, we will proceed in a slightly
different manner.
By \eqref{Eq_C-explicit}, we can write \eqref{Eq_Schur-PS} as
\[
s_{\la}(1,q,\dots,q^{n-1})=
q^{n(\la)} \frac{C_{\la}^{0}(q^n;q)}{C_{\la}^{-}(q;q)}.
\]
According to \eqref{Eq_C-conjugation}, the right-hand side also equals
\[
(-q^n)^{\abs{\la}}q^{n(\la')}\frac{C_{\la'}^{0}(q^{-n};q)}{C_{\la'}^{-}(q;q)},
\]
which, by \eqref{Eq_C-explicit} with $n\mapsto r$, is
\[
(-q^n)^{\abs{\la}} q^{n(\la')} 
\prod_{i=1}^r \frac{(q^{1-i-n};q)_{\la_i'}}{(q^{r-i+1};q)_{\la_i'}}
\prod_{1\leq i<j\leq r} \frac{1-q^{\la'_i-\la'_j+j-i}}{1-q^{j-i}}.
\]
By 
\begin{equation}\label{Eq_simp}
\prod_{i=1}^r \frac{(q^{1-i-n};q)_{\la'_i}}{(q^{r-i+1};q)_{\la'_i}}
=(-q^n)^{-\abs{\la}}
q^{n(\la)-n(\la')} \prod_{i=1}^r \frac{(q;q)_{n+i-1}(q;q)_{r-i}}
{(q;q)_{\la'_i+r-i} (q;q)_{n+i-\la'_i-1}},
\end{equation}
the lemma follows.
\end{proof}

\subsection{The odd-orthogonal Schur functions}

A sequence $(\la_1,\dots,\la_n)$ is called a half-partition if 
$\la_1\geq\la_2\geq\cdots\geq\la_n>0$ and $\la_i\in\Z+1/2$.

For $x=(x_1,\dots,x_n)$ and $\la=(\la_1,\dots,\la_n)$ a partition or 
half-partition, the \emph{odd-orthogonal Schur functions} are defined as
(cf.\ \cite{FH91,Littlewood50})
\begin{equation}\label{Eq_so2n1-def}
\so_{2n+1,\la}(x):=
\frac{\det_{1\leq i,j\leq n}
\big(x_i^{\la_j+n-j+1/2}-x_i^{-(\la_j+n-j+1/2)}\big)}
{\det_{1\leq i,j\leq n}\big(x_i^{n-j+1/2}-x_i^{-(n-j+1/2)}\big)}.
\end{equation}
The $\so_{2n+1,\la}(x)$ again arise from \eqref{Eq_Weyl-char}, this time 
for $\mathfrak{g}=\mathrm{so}_{2n+1}(\Complex)$. Defining $\phi$ by
\begin{gather*}
\phi:\Z\big[\eup^{\la}:\la\in P\big]^W\to 
\Z\big[x_1^{\pm 1/2},\dots,x_n^{\pm 1/2}\big]^{B_{n}} \\
\phi(\eup^{-\omega_i})
=\begin{cases}
x_1\cdots x_i, & \text{for $1\leq i\leq n-1$}, \\
(x_1\cdots x_n)^{1/2}, & \text{for $i=n$},
\end{cases}
\end{gather*}
where $\mathrm{B}_n$ is the \emph{hyperoctahedral group} acting on the $x_i$
by permuting them and by sending $x_i$ to $x_i^{-1}$ for some~$i$,
we have
\begin{equation}\label{Eq_phiVs-so-odd}
\phi\big(\ch V(\la)\big)=\so_{2n+1,\la}(x),
\end{equation}
where on the left $\la$ is a dominant weight parametrised as
\[
\la=(\la_1-\la_2)\omega_1+\cdots+(\la_{n-1}-\la_n)\omega_{n-1}
+2\la_n\omega_n,
\]
and on the right $\la$ is the partition or half-partition
$(\la_1,\dots,\la_n)$.

For later use, we will also define the companion
\begin{equation}\label{Eq_soplus-def}
\so^{+}_{2n+1,\la}(x):=
\frac{\det_{1\leq i,j\leq n}
\big(x_i^{\la_j+n-j+1/2}+x_i^{-(\la_j+n-j+1/2)}\big)}
{\det_{1\leq i,j\leq n}\big(x_i^{n-j+1/2}+x_i^{-(n-j+1/2)}\big)}.
\end{equation}
If $\la$ is a partition, it readily follows that
\begin{equation}\label{Eq_soplusso}
\so^{+}_{2n+1,\la}(x)=(-1)^{\abs{\la}} \so_{2n+1,\la}(-x).
\end{equation}
For half-partitions, however, 
$\so^{+}_{2n+1,\la}(x)$ is a rational function such that
\[
\so^{+}_{2n+1,\la}(x) D(x)\in \Z[x^{\pm}]^{B_{n}},\qquad
D(x):=\prod_{i=1}^n (x_i^{1/2}+x_i^{-1/2}).
\]
Since for half-partitions $\so_{2n+1,\la}(x)D(x)\in \Z[x^{\pm}]^{B_{n}}$, 
it follows that, regardless of the type of $\la$, we have
\[
\so_{2n+1,\la}(x)\so^{+}_{2n+1,\la}(x)\in \Z[x^{\pm}]^{B_{n}}.
\]

\medskip

In terms of the Sundaram tableaux introduced on 
page~\pageref{page_Sundaram-tableaux}, for $\la$ a partition we have
\[
\so_{2n+1,\la}(x)=\sum_T x^T,
\]
where the sum is over all Sundaram tableaux of shape $\la$ and
\begin{equation}\label{Eq_xT}
x^T:=\prod_{k=1}^n x_k^{m_k(T)-m_{\bar{k}}(T)}.
\end{equation}

\begin{lemma}[Principal specialisation --- dual form]\label{Lem_Bn-PS}
For $\la\subseteq (r^n)$ a partition, we have
\begin{subequations}
\begin{multline}\label{Eq_so-odd-PS-1}
\so_{2n+1,\la}(q,q^2,\dots,q^n)=
q^{n(\la)-n\abs{\la}}
\prod_{i=1}^r \qbin{2n+2r-1}{\la'_i+r-i}
\qbin{2n+2r-1}{r-i}^{-1} \\ \times
\prod_{1\leq i<j\leq r} \frac{1-q^{\la'_i-\la'_j+j-i}}{1-q^{j-i}}\cdot
\frac{1-q^{2n-\la'_i-\la'_j+i+j-1}}{1-q^{2n+i+j-1}}
\end{multline}
and
\begin{multline}\label{Eq_so-odd-PS-2}
\so_{2n+1,\la}(q^{1/2},q^{3/2},\dots,q^{n-1/2})=
q^{n(\la)-(n-1/2)\abs{\la}} \\ \times
\prod_{i=1}^r \frac{1+q^{n-\la'_i+i-1/2}}{1+q^{n+i-1/2}} \, 
\qbin{2n+2r-1}{\la'_i+r-i}
\qbin{2n+2r-1}{r-i}^{-1} \\ \times
\prod_{1\leq i<j\leq r} \frac{1-q^{\la'_i-\la'_j+j-i}}{1-q^{j-i}}\cdot
\frac{1-q^{2n-\la'_i-\la'_j+i+j-1}}{1-q^{2n+i+j-1}}.
\end{multline}
\end{subequations}
\end{lemma}

\begin{proof}
Let $\epsilon_1,\dots,\epsilon_n$ be the standard unit vectors in $\R^n$.
Assuming the realisation 
$\{\alpha_1,\dots,\alpha_n\}=
\{\epsilon_1-\epsilon_2,\dots,\epsilon_{n-1}-\epsilon_n,\epsilon_n\}$
for the simple roots of $\mathrm{so}_{2n+1}(\Complex)$
(see \cite{Humphreys78}), 
the fundamental weights and positive roots are given by
\begin{align*}
\{\omega_1,\dots,\omega_n\}&=\{\epsilon_1,\epsilon_1+\epsilon_2,\dots,
\epsilon_1+\cdots+\epsilon_{n-1},\tfrac{1}{2}(\epsilon_1+\cdots+\epsilon_n)\} ,\\
\{\alpha\in\Phi: \alpha>0\}&=
\{\epsilon_i: 1\leq i\leq n\}\cup\{\epsilon_i\pm \epsilon_j: 1\leq i<j\leq n\}.
\end{align*}
Hence, by \eqref{Eq_qdim2}, \eqref{Eq_phiVs-so-odd} and
$F(x_i)=q^{n-i+1}$, we have
\begin{align}\label{Eq_so-odd-PS-check}
\so_{2n+1,\la}(q,q^2,\dots,q^n)
&=q^{n(\la)-n\abs{\la}} \dim_q^{\vee} V(\Lambda) \\
&=q^{n(\la)-n\abs{\la}} \prod_{i=1}^n 
\frac{1-q^{2\la_i+2n-2i+1}}{1-q^{2n-2i+1}} \notag \\ 
& \qquad \times
\prod_{1\leq i<j\leq n}\frac{1-q^{\la_i-\la_j+j-i}}{1-q^{j-i}}\cdot
\frac{1-q^{\la_i+\la_j+2n-i-j+1}}{1-q^{2n-i-j+1}}. \notag
\end{align}
It follows from \eqref{Eq_C-explicit} that the right-hand side can be 
expressed in terms of the generalised $q$-shifted factorials as
\[
q^{n(\la)-n\abs{\la}} \,
\frac{C^0_{\la}(q^n,-q^n,q^{n+1/2},-q^{n+1/2};q)}
{C^{-}_{\la}(q;q)C^{+}_{\la}(q^{2n-1};q)},
\]
where $C^0_{\la}(a_1,\dots,a_k;q)=C^0_{\la}(a_1;q)\cdots C^0_{\la}(a_k;q)$.
By \eqref{Eq_C-conjugation}, this is also
\[
(-q^{n+1})^{\abs{\la}} q^{n(\la')} \,
\frac{C^0_{\la'}(q^{-n},-q^{-n},q^{-n-1/2},-q^{-n-1/2};q)}
{C^{-}_{\la'}(q;q)C^{+}_{\la'}(q^{-2n-1};q)}.
\]
Again using \eqref{Eq_C-explicit}, but now with $n$ replaced by $r$,
this is
\begin{equation}
(-q^{n+r})^{\abs{\la}} q^{n(\la')} \prod_{i=1}^r 
\frac{(q^{1-i-2n-r};q)_{\la'_i}}{(q^{r-i+1};q)_{\la'_i}} 
\prod_{1\leq i<j\leq r} \frac{1-q^{\la'_i-\la'_j+j-i}}{1-q^{j-i}} \cdot
\frac{1-q^{2n-\la_i-\la_j+i+j-1}}{1-q^{2n+i+j-1}}.
\end{equation}
By \eqref{Eq_simp} with $n\mapsto 2n+r$, the first claim follows.

The second specialisation \eqref{Eq_so-odd-PS-2} follows in much 
the same way by applying \eqref{Eq_C-explicit} and 
\eqref{Eq_C-conjugation} to
\begin{align}\label{Eq_so-odd-PS}
\so_{2n+1,\la}&(q^{1/2},q^{3/2},\dots,q^{n-1/2}) \\
&=q^{n(\la)-(n-1/2)\abs{\la}} \dim_q V(\Lambda) \notag \\
&=q^{n(\la)-(n-1/2)\abs{\la}} 
\prod_{i=1}^n \frac{1-q^{\la_i+n-i+1/2}}{1-q^{n-i+1/2}} \notag \\ 
& \qquad \times
\prod_{1\leq i<j\leq n}\frac{1-q^{\la_i-\la_j+j-i}}{1-q^{j-i}}\cdot
\frac{1-q^{\la_i+\la_j+2n-i-j+1}}{1-q^{2n-i-j+1}}. \notag \qedhere
\end{align}
\end{proof}

For later reference we also state the principal specialisation
of $\so_{2n+1,\la}^{+}(x)$.

\begin{lemma}
For $\la=(\la_1,\dots,\la_n)$ a partition or half-partition, we have
\begin{multline}\label{Eq_so-odd-min-PS}
\so_{2n+1,\la}^{+}(q^{1/2},\dots,q^{n-1/2})
=q^{n(\la)-(n-1/2)\abs{\la}} 
\prod_{i=1}^n \frac{1+q^{\la_i+n-i+1/2}}{1+q^{n-i+1/2}} \\ 
\times
\prod_{1\leq i<j\leq n}\frac{1-q^{\la_i-\la_j+j-i}}{1-q^{j-i}}\cdot
\frac{1-q^{\la_i+\la_j+2n-i-j+1}}{1-q^{2n-i-j+1}}.
\end{multline}
\end{lemma}

\begin{proof}
According to \eqref{Eq_denom}, the denominator identity for $\mathrm{B}_n$ 
(or $\so_{2n+1,\la}(\Complex)$) is given by 
(see also \cite[Equation~(2.4)]{Krattenthaler99})
\begin{multline}\label{Eq_Bn-denom}
\det_{1\leq i,j\leq n} \Big(x_i^{n-j+1/2}-x_i^{-(n-j+1/2)}\Big) \\
=(-1)^{\binom{n+1}{2}} \prod_{i=1}^n x_i^{1/2-n}(1-x_i)
\prod_{1\leq i<j\leq n} (x_i-x_j)(1-x_ix_j).
\end{multline}
Replacing $x_i$ by $-x_i$ (readers worried about a choice of 
branch-cut should first multiply both sides by $\prod_i x_i^{-1/2}$ and
later divide by this factor) and taking the transpose of the determinant,
we obtain (see also \cite[Equation~(2.6)]{Krattenthaler99})
\begin{equation}\label{Eq_Bn-denom-2}
\det_{1\leq i,j\leq n} \Big(x_j^{n-i+1/2}+x_j^{-(n-i+1/2)}\Big) 
=\prod_{i=1}^n x_i^{1/2-n}(1+x_i)
\prod_{1\leq i<j\leq n} (x_i-x_j)(1-x_ix_j).
\end{equation}

If we specialise $x_i=q^{n-i+1/2}$ ($1\leq i\leq n$) in
\eqref{Eq_soplus-def}, then we get
\begin{equation}
\so^{+}_{2n+1,\la}(q^{1/2},\dots,q^{n-1/2}) 
=\frac{\det_{1\leq i,j\leq n}
\big(q^{(\la_j+n-j+1/2)(n-i+1/2)}+q^{-(\la_j+n-j+1/2)(n-i+1/2)}\big)}
{\det_{1\leq i,j\leq n}
\big(q^{(n-j+1/2)(n-i+/2)}+q^{-(n-j+1/2)(n-i+1/2)}\big)}.
\end{equation}
By \eqref{Eq_Bn-denom-2} with $x_j=q^{\la_j+n-j+1/2}$ or $x_j=q^{n-j+1/2}$,
both determinants on the right-hand side 
can be expressed in product form, resulting
in \eqref{Eq_so-odd-min-PS}.
\end{proof}

\subsection{The symplectic Schur functions}

For $x=(x_1,\dots,x_n)$ and $\la$ a partition of length at most $n$,
the \emph{symplectic Schur functions} are defined as
\begin{equation}\label{Eq_symp}
\symp_{2n,\la}(x):=
\frac{\det_{1\leq i,j\leq n}
\big(x_i^{\la_j+n-j+1}-x_i^{-(\la_j+n-j+1)}\big)}
{\det_{1\leq i,j\leq n}\big(x_i^{n-j+1}-x_i^{-(n-j+1)}\big)}.
\end{equation}
If $\mathfrak{g}=\symp_{2n}(\Complex)$, then
\[
\phi\big(\ch V(\la)\big)=\symp_{2n,\la}(x),
\]
where $\phi(\eup^{-\omega_i})=x_1\cdots x_i$ ($1\leq i\leq n$) and
\[
P_{+}\ni
\la=(\la_1-\la_2)\omega_1+\cdots+(\la_{n-1}-\la_n)\omega_{n-1}+\la_n\omega_n.
\]

To express this combinatorially, we need the symplectic tableaux of King 
and El-Sharkaway \cite{King76,KES83}. These are semi-standard Young tableaux 
on $1<\bar{1}<2<\bar{2}<\cdots<n<\bar{n}$ such that all entries in row $k$ 
are at least $k$. For example, 
\[
\young(1\be23\bw,2\bt34,\bv\bv5)
\]
is a symplectic tableau for $n\geq 5$. The symplectic analogue of
\eqref{Eq_Schur-tab} then is
\[
\symp_{2n,\la}(x)=\sum_T x^T,
\]
where the sum is over all symplectic tableaux of shape~$\lambda$ 
and $x^T$ is again given by
\eqref{Eq_xT}.

\begin{lemma}[Principal specialisation --- dual form]\label{Lemma_Cn}
For $\la\in (r^n)$, we have
\begin{subequations}
\begin{multline}\label{Eq_SpecCnb}
\symp_{2n,\la}(q,q^2,\dots,q^n)=q^{n(\la)-n\abs{\la}}
\prod_{i=1}^r \frac{1-q^{n-\la_i'+i}}{1-q^{n+i}} 
\qbin{2n+2r}{\la'_i+r-i} \qbin{2n+2r}{r-i}^{-1} \\
\times \prod_{1\leq i<j\leq r}
\frac{1-q^{\la_i'-\la_j'+j-i}}{1-q^{j-i}}\cdot
\frac{1-q^{2n-\la_i'-\la_j'+i+j}}{1-q^{2n+i+j}}
\end{multline}
and
\begin{multline}
\symp_{2n,\la}(q^{1/2},q^{3/2},\dots,q^{n-1/2}) \\
=q^{n(\la)-(n-1/2)\abs{\la}}
\prod_{i=1}^r \frac{1-q^{2(n-\la_i'+i)}}{1-q^{2(n+i)}} 
\qbin{2n+2r}{\la'_i+r-i} \qbin{2n+2r}{r-i}^{-1} \\
\times \prod_{1\leq i<j\leq r}
\frac{1-q^{\la_i'-\la_j'+j-i}}{1-q^{j-i}}\cdot
\frac{1-q^{2n-\la_i'-\la_j'+i+j}}{1-q^{2n+i+j}}.
\end{multline}
\end{subequations}
\end{lemma}

\begin{proof}
If we take the simple roots to be 
$\{\alpha_1,\dots,\alpha_n\}=
\{\epsilon_1-\epsilon_2,\dots,\epsilon_{n-1}-\epsilon_n,2\epsilon_n\}$
(see \cite{Humphreys78}), then
\begin{align*}
\{\omega_1,\dots,\omega_n\}&=\{\epsilon_1,\epsilon_1+\epsilon_2,\dots,
\epsilon_1+\cdots+\epsilon_n\}, \\
\{\alpha\in\Phi: \alpha>0\}&=
\{2\epsilon_i: 1\leq i\leq n\}\cup\{\epsilon_i\pm \epsilon_j: 
1\leq i<j\leq n\}.
\end{align*}
From Lemma~\ref{Lem_qdim}, it then follows that
\begin{subequations}\label{Eq_sp-PS}
\begin{align}
\symp_{2n,\la}(q,q^2,\dots,q^n)&=q^{n(\la)-n\abs{\la}} \dim_q V(\la) \\
&=q^{n(\la)-n\abs{\la}} 
\prod_{i=1}^n \frac{1-q^{2(\la_i+n-i+1)}}{1-q^{2(n-i+1)}} \notag  \\
& \quad 
\kern1cm
\times
\prod_{1\leq i<j\leq n}\frac{1-q^{\la_i-\la_j+j-i}}{1-q^{j-i}}\cdot
\frac{1-q^{\la_i+\la_j+2n-i-j+2}}{1-q^{2n-i-j+2}}  \notag
\end{align}
and
\begin{align}
\symp_{2n,\la}(q^{1/2},q^{3/2},\dots,q^{n-1/2})
&=q^{n(\la)-(n-1/2)\abs{\la}} \dim_q^{\vee} V(\Lambda) \\
&=q^{n(\la)-(n-1/2)\abs{\la}} 
\prod_{i=1}^n \frac{1-q^{\la_i+n-i+1}}{1-q^{n-i+1}} \notag \\
& \quad 
\kern1cm
\times
\prod_{1\leq i<j\leq n}\frac{1-q^{\la_i-\la_j+j-i}}{1-q^{j-i}}\cdot
\frac{1-q^{\la_i+\la_j+2n-i-j+2}}{1-q^{2n-i-j+2}}. \notag
\end{align}
\end{subequations}
The rest of the proof is analogous to that of Lemma~\ref{Lem_Bn-PS}; 
we omit the details.
\end{proof}

\subsection{The even-orthogonal Schur functions}

Let a $\mathrm{D}_n$ partition be a weakly decreasing sequence
$(\la_1,\dots,\la_n)$ 
such that each $\la_i\in\Z$ or each $\la_i\in\Z+1/2$, and such that 
$\la_{n-1}\geq \abs{\la_n}$.
If $\la$ is a $\D_n$ partition then so is 
$\bar{\la}:=(\la_1,\dots,\la_{n-1},-\la_n)$.

For $x=(x_1,\dots,x_n)$ and $\la$ a $\mathrm{D}_n$ partition, the
\emph{even-orthogonal Schur functions} are defined by
\begin{equation}\label{Eq_so2n}
\so_{2n,\la}(x):=\sum_{\sigma\in\{\pm 1\}} 
\frac{\det_{1\leq i,j\leq n}\big(\sigma x_i^{\la_j+n-j}
+x_i^{-(\la_j+n-j)}\big)}
{\det_{1\leq i,j\leq n}\big(x_i^{n-j}+x_i^{-(n-j)}\big)}.
\end{equation}
We note that $\so_{2n,\bar{\la}}(x)=\so_{2n,\la}(\bar{x})$,
where $\bar{x}:=(x_1,\dots,x_{n-1},x_n^{-1})$.
Assuming $\mathfrak{g}=\so_{2n}(\Complex)$, we have
\[
\phi\big(\ch V(\la)\big)=\so_{2n,\la}(x),
\]
where 
\[
\phi(\eup^{-\omega_i})=
\begin{cases} x_1\cdots x_i, & \text{for $1\leq i\leq n-2$}, \\
(x_1\cdots x_{n-1}x_n^{-1})^{1/2}, & \text{for $i=n-1$}, \\[1pt]
(x_1\cdots x_n)^{1/2}, & \text{for $i=n$},
\end{cases}
\]
and
\[
P_{+}\ni
\la=(\la_1-\la_2)\omega_1+\cdots+(\la_{n-1}-\la_n)\omega_{n-1}+
(\la_{n-1}+\la_n)\omega_n.
\]

For our purposes it is not enough to consider $\so_{2n,\la}(x)$;
we also need the closely related \emph{even-orthogonal characters}
(cf.\ \cite{KT87})
\begin{equation}\label{Eq_o2n}
\ortho_{2n,\la}(x)=u_{\la}\,
\frac{\det_{1\leq i,j\leq n}\big(x_i^{\la_j+n-j}
+x_i^{-(\la_j+n-j)}\big)}
{\det_{1\leq i,j\leq n}\big(x_i^{n-j}+x_i^{-(n-j)}\big)},
\end{equation}
where $\la$ is a partition or half-partition and 
$u_{\la}=1$ if $l(\la)<n$ and $u_{\la}=2$ if $l(\la)=n$.
Note that
\begin{equation}\label{Eq_o-so}
\ortho_{2n,\la}(x)=\begin{cases}
\so_{2n,\la}(x), & \text{if $l(\la)<n$}, \\
\so_{2n,\la}(x)+\so_{2n,\bar{\la}}(x), & \text{if $l(\la)=n$}.
\end{cases}
\end{equation}

Also the even-orthogonal characters can be expressed in terms of a tableau
sum, see, e.g., \cite{ProcAK,FK97}.
We will however not define these tableaux here and instead restrict our 
attention to the simpler ``even Sundaram tableaux'' of \cite{FK97}.
An \emph{even Sundaram tableau} is a semi-standard Young tableau on the 
alphabet $1<\bar{1}<2<\bar{2}<\cdots<n<\bar{n}<\infty$ such that all entries
in row $k$ are at least $\bar{k}$, with the exception that $\infty$ may 
occur multiple times in each column but at most once in each row.
Note that the only difference with the earlier definition of Sundaram 
tableaux is that entries in row $k$ have to be at least $\bar{k}$ instead
of $k$. 
This implies that $1$ cannot actually occur in an even Sundaram tableaux.
Due to the absence of the letter $1$, it is not known how to assign monomials
to even Sundaram tableaux so that they generate $\ortho_{2n,\la}(x)$.
It is however shown in \cite{FK97} that $\ortho_{2n,\la}(1^n)$ correctly
counts the number of even Sundaram tableaux of shape $\la$.

\begin{lemma}
For $\la$ a partition contained in $(r^n)$, we have
\begin{multline}\label{Eq_ortho-PS-dual}
\ortho_{2n,\la}(q^{1/2},q^{3/2},\dots,q^{n-1/2}) \\
=q^{n(\la)-(n-1/2)\abs{\la}} 
\prod_{i=1}^r \qbin{2n+2r-2}{\la'_i+r-i} \qbin{2n+2r-2}{r-i}^{-1} \\ 
\times
\prod_{1\leq i<j\leq r} \frac{1-q^{\la_i'-\la_j'+j-i}}{1-q^{j-i}}\cdot
\frac{1-q^{2n-\la_i'-\la_j'+i+j-2}}{1-q^{2n+i+j-2}}.
\end{multline}
\end{lemma}

There is a similar result for $\ortho_{2n,\la}(1,q,\dots,q^{n-1})$,
but this is not needed.

\begin{proof}
If we specialise $x_i=q^{n-i+1/2}$ in \eqref{Eq_o2n}, with $1\leq i\leq n$,
and then use the determinant evaluation \eqref{Eq_Bn-denom-2} with
$x_j=q^{\la_j+n-j}$ or $x_j=q^{n-j}$, we obtain
\begin{multline}\label{Eq_ortho-PS}
\ortho_{2n,\la}(q^{1/2},q^{3/2},\dots,q^{n-1/2}) 
=u_{\la} \, q^{n(\la)-(n-1/2)\abs{\la}} 
\prod_{i=1}^n \frac{1+q^{\la_i+n-i}}{1+q^{n-i}} \\ \times
\prod_{1\leq i<j\leq n}\frac{1-q^{\la_i-\la_j+j-i}}{1-q^{j-i}}\cdot
\frac{1-q^{\la_i+\la_j+2n-i-j}}{1-q^{2n-i-j}}. 
\end{multline}
The rest of the proof follows that of Lemma~\ref{Lem_Bn-PS}.
\end{proof}

For later reference we note that it follows in much the same way from
\eqref{Eq_Bn-denom} and \eqref{Eq_Bn-denom-2} that
\begin{multline}\label{Eq_so-even-PS-wrong}
\so_{2n,\la}(q^{1/2},q^{3/2},\dots,q^{n-1/2}) \\
=q^{n(\la)-(n-1/2)\abs{\la}} 
\bigg( \prod_{i=1}^n \frac{1+q^{\la_i+n-i}}{1+q^{n-i}} +
\prod_{i=1}^n \frac{1-q^{\la_i+n-i}}{1+q^{n-i}} \bigg) \\
\times
\prod_{1\leq i<j\leq n}\frac{1-q^{\la_i-\la_j+j-i}}{1-q^{j-i}}\cdot
\frac{1-q^{\la_i+\la_j+2n-i-j}}{1-q^{2n-i-j}}. \qedhere
\end{multline}

\section{Okada-type formulas}\label{sec:Okada}

With the exception of type $\mathrm{A}_{r-1}$, our proofs of the discrete 
analogues of Macdonald--Mehta integrals for $\gamma=1/2$ given in
the next section rely on
formulas for the multiplication of Schur functions of 
type~$\mathfrak{g}$ indexed by partitions of rectangular shape.
Such formulas have been given by Okada in \cite{Okada98}.
We use several of his formulas, but we also require additional
ones. In the subsection below, we list all these results,
and we present the (principal) specialisations of these
formulas that we actually need. Subsection~\ref{sec:OkadProof}
provides the proofs of the new results not contained in
\cite{Okada98}. These proofs heavily rely on ``preparatory results"
from \cite{Okada98}.

\subsection{Main results}
Our first result applies to $\mathfrak{g}=\so_{2n+1}(\Complex)$.
Let $\so_{2n+1,\la}^{-}(x):=\so_{2n+1,\la}(x)$.

\begin{theorem}\label{Thm_sososo}
Let $r$ be a non-negative integer, $\varepsilon\in\{-1,1\}$ and
$s:=\frac{1}{2}r$. Then
\begin{equation}\label{Eq_sososo}
\sum_{\la\subseteq (r^n)}  \varepsilon^{\abs{\la}} 
\so_{2n+1,\la}(\varepsilon x) 
=\so_{2n+1,(s^n)}(x)\so^{\sigma}_{2n+1,(s^n)}(x),
\end{equation}
where the sum on the left is over partitions, and $\sigma=-$ if 
$\varepsilon=1$ and $\sigma=+$ if $\varepsilon=-1$.
\end{theorem}

For $\varepsilon=1$ this is (a special case of) Okada's
\cite[Theorem~2.5(1)]{Okada98}.

Later we require \eqref{Eq_sososo} in principally specialised form
as follows from \eqref{Eq_so-odd-PS-check}, \eqref{Eq_so-odd-PS} and 
\eqref{Eq_so-odd-min-PS} for $\la=(s^n)$. 

\begin{corollary}\label{Cor_Bn-box}
For $r$ a non-negative integer and $\varepsilon \in\{-1,1\}$, we have
\begin{subequations}
\begin{equation}\label{Eq_B-sum-a}
\sum_{\la\subseteq (r^n)} \so_{2n+1,\la}(q,q^2,\dots,q^n)=
q^{-r\binom{n+1}{2}} \frac{(q^{r+1};q^2)_n}{(q;q^2)_n}
\prod_{i,j=1}^n \frac{1-q^{i+j+r-1}}{1-q^{i+j-1}}
\end{equation}
and
\begin{multline}\label{Eq_B-sum-b}
\sum_{\la\subseteq (r^n)} \varepsilon^{\abs{\la}}
\so_{2n+1,\la}(\varepsilon q^{1/2},\varepsilon
q^{3/2},\dots,\varepsilon q^{n-1/2}) \\
=q^{-rn^2/2}\, \frac{(q^{(r+1)/2};q)_n (\varepsilon
q^{(r+1)/2};q)_n}{(q^{1/2};q)_n(\varepsilon q^{1/2};q)_n}
\prod_{\substack{i,j=1 \\ i\neq j}}^n
\frac{1-q^{i+j+r-1}}{1-q^{i+j-1}},
\end{multline}
\end{subequations}
where $\la$ is summed over partitions.
\end{corollary}

Letting $q$ tend to $1$ in \eqref{Eq_B-sum-a} (or the $\varepsilon=1$
case of \eqref{Eq_B-sum-b}) yields the unweighted enumeration of Sundaram
tableaux given in Theorem~\ref{Thm_Sundaram}.
Taking $\varepsilon=-1$ in \eqref{Eq_B-sum-b}, then using
\[
\frac{(q^{(r+1)/2};q)_n
(-q^{(r+1)/2};q)_n}{(q^{1/2};q)_n(-q^{1/2};q)_n}=
\frac{(q^{r+1};q^2)_n}{(q;q^2)_n},
\]
and finally letting $q^{1/2}$ tend to $\pm 1$ gives
\[
\sum_{T \subseteq (r^n)} (-1)^{\abs{T}} 
(\mp 1)^{\sum_{k=1}^n (m_k(T)+m_{\bar{k}}(T))}
=(\pm 1)^{rn} \prod_{i,j=1}^n \frac{i+j+r-1}{i+j-1}.
\]
Since 
\[
\abs{T}=m_{\infty}(T)+\sum_{k=1}^n \big(m_k(T)+m_{\bar{k}}(T)\big),
\]
this results in the two weighted enumerations of that theorem.

\medskip
Next we consider $\mathfrak{g}=\symp_{2n}(\Complex)$.

\begin{theorem}\label{Thm_spspso}
Let $r$ be a non-negative integer and 
$s:=\floor{\frac{1}{2}r}$, $t:=\ceiling{\frac{1}{2}r}$.
Then
\begin{equation}\label{Eq_spspso}
\sum_{\la\subseteq (r^n)}  \symp_{2n,\la}(x) 
=\symp_{2n,(s^n)}(x) \so_{2n+1,(t^n)}(x).
\end{equation}
\end{theorem}

This identity follows from \cite[Theorem~2.5(1)]{Okada98} by observing that
(see e.g.\ \cite[Proposition~A2.1(c)]{ProcAK})
$$
     \so_{2n+1,\lambda+1/2}(x)
     =
     \symp_{2n,\lambda}(x)\,
      \prod_{i=1}^n (x_i^{1/2} + x_i^{-1/2}),
$$
where $\lambda+1/2$ stands for $(\lambda_1+1/2, \dots, \lambda_n+1/2)$.
It is interesting to note that Proctor 
\cite[Lemma~4, equation for $A_{2n}(m\omega_r)$, case~$r=n$]{ProcAE} 
obtained this same sum from a specialised Schur function. 
(In representation-theoretic terms:
the restriction of an $\SL_{2n+1}(\Complex)$-character indexed by a 
rectangular shape to $\Sp_{2n}(\Complex)$ 
decomposes into the sum of symplectic characters indexed by all
shapes contained in that rectangle; see also
\cite[Equation~(3.4)]{Krattenthaler98}.) He used his result to prove the
(at the time conjectured) formula for the number of symmetric
self-complementary plane partitions contained in a given box.

Once again, use of \eqref{Eq_so-odd-PS-check} as well as \eqref{Eq_sp-PS}
yields our second corollary.

\begin{corollary}\label{Cor_Cn-box}
For $r$ a non-negative integer, we have
\begin{subequations}
\begin{equation}\label{Eq_Sp-heel}
\sum_{\la\subseteq (r^n)} \symp_{2n,\la}(q,q^2,\dots,q^n)=
q^{-r\binom{n+1}{2}} 
\prod_{i=1}^{n+1} \prod_{j=1}^n \frac{1-q^{i+j+r-1}}{1-q^{i+j-1}}
\end{equation}
and
\begin{equation}\label{Eq_Sp-half}
\sum_{\la\subseteq (r^n)}
\symp_{2n,\la}(q^{1/2},q^{3/2},\dots,q^{n-1/2}) 
=q^{-rn^2/2} \prod_{i=1}^{2n} \frac{1-q^{(i+r)/2}}{1-q^{i/2}}
\prod_{i=1}^n \prod_{j=1}^{n-1} \frac{1-q^{i+j+r}}{1-q^{i+j}}.
\end{equation}
\end{subequations}
\end{corollary}

Letting $q^{1/2}$ tend to $\pm 1$ in \eqref{Eq_Sp-half} implies
two counting formulas for symplectic tableaux.

\begin{theorem}
The number of symplectic tableaux of height at most $n$ and width at most
$r$ is given by
\begin{equation}\label{Eq_Cn-in-box}
\prod_{i=1}^{n+1} \prod_{j=1}^n \frac{i+j+r-1}{i+j-1}
\end{equation}
and the number of such tableaux weighted by $(-1)^{\abs{T}}$ is
\[
(-1)^{rn} \prod_{i=1}^n \frac{i+\floor{r/2}}{i}
\prod_{i=1}^n \prod_{j=1}^{n-1} \frac{i+j+r}{i+j}.
\]
\end{theorem} 

For example, when $r=n=2$ there are
$(3\cdot 4^2\cdot5^2\cdot 6)/(1\cdot 2^2 \cdot 3^2\cdot 4)= 50$
symplectic tableaux, with the following break-down according to shape

\medskip

\begin{center}
\begin{tikzpicture}[scale=0.3]
\draw[fill] (0,-0.5) circle (0.08cm);
\draw (1.5,0) rectangle (2.5,-1);
\draw (4,0) rectangle (5,-2);
\draw (4,-1) -- (5,-1);
\draw (6.5,0) rectangle (8.5,-1);
\draw (7.5,0) -- (7.5,-1);
\draw (10,0) -- (12,0) -- (12,-1) -- (11,-1) -- (11,-2) -- (10,-2) -- cycle;
\draw (10,-1) -- (11,-1) -- (11,0);
\draw (13.5,0) rectangle (15.5,-2);
\draw (14.5,0) -- (14.5,-2);
\draw (13.5,-1) -- (15.5,-1);
\draw (0,-3) node {$1$};
\draw (2,-3) node {$4$};
\draw (4.5,-3) node {$5$};
\draw (7.5,-3) node {$10$};
\draw (11,-3) node {$16$};
\draw (14.5,-3) node {$14$};
\end{tikzpicture}
\end{center}

\noindent
so that the signed enumeration is 
$1-4+5+10-16+14=10=(2\cdot 3\cdot 4\cdot 5)/
(1\cdot 2^2\cdot 3)$.

We remark that \eqref{Eq_Cn-in-box} is not actually new,
and it is implicit in \cite{ProcAE} that
the number of symplectic tableaux contained in $(r^m)$ 
($0\leq m\leq n$) is given by
\[
\prod_{i=1}^{2n-m+1} \prod_{j=1}^m \frac{i+j+r-1}{i+j-1}.
\]
See also \cite[Theorem~7]{KrGVAA} for an equivalent statement in terms of
vicious walkers (non-intersecting lattice paths).

\medskip

Our final Okada-type formula involves the even-orthogonal as well as
orthogonal characters.

\begin{theorem}\label{Thm_sososo-2}
Let $r$ be a positive integer. Then
\begin{subequations}
\begin{equation}\label{Eq_sososo-2}
\sum_{\la\subseteq (r^n)}  \so_{2n,\la}(x) 
=\so_{2n,(s^n)}(x) \so_{2n+1,(s^n)}(x),
\end{equation}
where $s:=\frac{1}{2}r$, and
\begin{equation}\label{Eq_ooso}
\sum_{\substack{\la\subseteq (r^n) \\[1pt] l(\la)=n}}  
\ortho_{2n,\la}(x) 
=\ortho_{2n,(s^n)}(x) \so_{2n+1,(t^n)}(x),
\end{equation}
\end{subequations}
where $s:=\frac{1}{2}(r+1)$ and $t:=\frac{1}{2}(r-1)$.
\end{theorem}

We remark that \eqref{Eq_ooso} also holds when the orthogonal characters
are replaced by even-orthogonal Schur functions, but in some sense this is 
a weakening of the result. 
In the other direction, the analogous result does not hold for
\eqref{Eq_sososo-2} in that we cannot replace the 
even-orthogonal Schur functions by orthogonal characters. 

By \eqref{Eq_so-odd-PS}, \eqref{Eq_ortho-PS} and \eqref{Eq_so-even-PS-wrong},
the above two identities result in the final corollary of this section.

\begin{corollary}\label{Cor_Dn-box-New}
For $r$ a positive integer, we have
\begin{subequations}
\begin{align}\label{Eq_ortho-Okada-spec-1}
\sum_{\la\subseteq (r^n)} &
\so_{2n,\la}(q^{1/2},q^{3/2},\dots,q^{n-1/2}) \\
&=q^{-rn^2/2} \frac{(q^{r/2+1/2};q)_n}{(q^{1/2};q)_n} 
\bigg( \frac{(-q^{r/2};q)_n}{(-1;q)_n}+\frac{(q^{r/2};q)_n}{(-1;q)_n}\bigg)
\prod_{i=1}^n \prod_{j=1}^{n-1} \frac{1-q^{i+j+r-1}}{1-q^{i+j-1}} \notag 
\intertext{and}
\label{Eq_ortho-Okada-spec-2} 
\sum_{\substack{\la\subseteq (r^n) \\[1pt] l(\la)=n}} &
\ortho_{2n,\la}(q^{1/2},q^{3/2},\dots,q^{n-1/2}) \\
&=2 q^{-rn^2/2} \frac{(q^{r/2};q)_n}{(q^{1/2};q)_n}
\cdot \frac{(-q^{r/2+1/2};q)_n}{(-1;q)_n} 
\prod_{i=1}^n \prod_{j=1}^{n-1} \frac{1-q^{i+j+r-1}}{1-q^{i+j-1}}.  \notag 
\end{align}
\end{subequations}
\end{corollary}

If we let $q\to 1$ in \eqref{Eq_ortho-Okada-spec-2}, we obtain a closed-form
expression for the number of even Sundaram tableaux of height exactly
$n$ and width at most $r$.
From \eqref{Eq_o-so} and $\so_{2n,\bar{\la}}(x)=\so_{2n,\la}(\bar{x})$,
it follows that
\begin{equation}\label{Eq_ortho-so-even}
\ortho_{2n,\la}(1^n)=u_{\la} \so_{2n,\la}(1^n).
\end{equation}
Hence we can combine \eqref{Eq_ortho-Okada-spec-1} and
\eqref{Eq_ortho-Okada-spec-2} to also obtain the enumeration
of such tableaux contained in $(r^n)$.

\begin{theorem}
The number of even Sundaram tableaux of height at most $n$ and width at most
$r$ is given by
\[
2^{2n-1}\frac{(\tfrac{1}{2}r+\frac{1}{2})_n+(\tfrac{1}{2}r)_n}{n!}
\prod_{i=1}^n \prod_{j=1}^{n-1}  \frac{i+j+r-1}{i+j},
\]
and the number of such tableaux of height exactly $n$ is
\[
2^{2n}\frac{(\tfrac{1}{2}r)_n}{n!}
\prod_{i=1}^n \prod_{j=1}^{n-1} \frac{i+j+r-1}{i+j}.
\]
\end{theorem}

For example, when $r=n=2$ there are $46$
even Sundaram tableaux, with the following break-down according to shape

\medskip

\begin{center}
\begin{tikzpicture}[scale=0.3]
\draw[fill] (0,-0.5) circle (0.08cm);
\draw (1.5,0) rectangle (2.5,-1);
\draw (4,0) rectangle (5,-2);
\draw (4,-1) -- (5,-1);
\draw (6.5,0) rectangle (8.5,-1);
\draw (7.5,0) -- (7.5,-1);
\draw (10,0) -- (12,0) -- (12,-1) -- (11,-1) -- (11,-2) -- (10,-2) -- cycle;
\draw (10,-1) -- (11,-1) -- (11,0);
\draw (13.5,0) rectangle (15.5,-2);
\draw (14.5,0) -- (14.5,-2);
\draw (13.5,-1) -- (15.5,-1);
\draw (0,-3) node {$1$};
\draw (2,-3) node {$4$};
\draw (4.5,-3) node {$6$};
\draw (7.5,-3) node {$9$};
\draw (11,-3) node {$16$};
\draw (14.5,-3) node {$10$};
\end{tikzpicture}
\end{center}
so that exactly $32$ of these have height $2$.

\subsection{Proof of Theorems~\ref{Thm_sososo} and \ref{Thm_sososo-2}}
\label{sec:OkadProof}

Our proofs closely follow Okada's Pfaffian-based approach,
which relies on two key results: the Ishikawa--Wakayama minor summation 
formula \cite[Theorem~2]{IW95} (see also \cite[Theorem~3]{Okada89}) 
and Okada's Pfaffian evaluation \cite[Theorem~4.4]{Okada98}.

Recall that the Pfaffian of a $2m\times 2m$ skew-symmetric matrix $M$ 
is defined as \[
\Pf(M)=\sum_{\pi} (-1)^{c(\pi)} \prod_{(i,j)\in\pi} M_{ij},
\]
where the sum is over perfect matchings $\pi$ (or $1$-factorisations)
of the complete graph on $2m$ vertices (labelled $1,2,\dots,2m$), and 
the product is over all edges $(i,j)$ in the matching, $1\le i<j\le 2m$.
The crossing number $c(\pi)$ of a perfect matching $\pi$ is the number 
of pairs of edges $(i,j)$ and $(k,l)$ of $\pi$ such that $i<k<j<l$.

\begin{theorem}[Minor summation formula]\label{Thm_IsWa}
Let $n$ and $r$ be positive integers such that $n$ is even and $n\leq r$, 
and let $M$ be an arbitrary $n\times r$ matrix. Then
\begin{equation}\label{Eq_minor}
\sum_{\substack{J\subseteq \{1,\dots,r\} \\[1pt] \abs{J}=n}}
\det_{\substack{1\leq i\leq n \\[1pt] j\in J}}\big(M_{ij}\big)=\Pf(B),
\end{equation}
where 
$B$ is the $n\times n$ skew-symmetric matrix
\begin{equation}\label{Eq_B}
B=MAM^t,
\end{equation}
with $A$ the $r\times r$ skew-symmetric matrix with entries $A_{ij}=1$ 
for $j>i$. 
\end{theorem}

Here it should be understood that $J$ is viewed as an ordered $n$-subset of
$\{1,\dots,r\}$, i.e., $J=\{j_1<j_2<\dots<j_n\}$.

\begin{theorem}[Okada's Pfaffian evaluation]\label{Thm_Okada}
Let $x=(x_1,\dots,x_n)$ where $n$ is even. Let 
$Q(x;a,b)$ be the $n\times n$ skew-symmetric matrix with entries
\begin{equation}\label{Eq_Q-def}
Q_{ij}(x;a,b)=\frac{q(x_i,x_j,a_i,a_j)q(x_i,x_j,b_i,b_j)} 
{(x_i-x_j)(1-x_ix_j)},
\end{equation}
where $q(\alpha,\beta,\gamma,\delta):=
(\alpha-\beta)(1-\gamma\delta)-(1-\alpha\beta)(\gamma-\delta)$,
and let $W(x;a)$ be the $n\times n$ matrix with entries
\begin{equation}\label{Eq_W}
W_{ij}(x;a)=a_i x_i^{n-j}-x_i^{j-1}.
\end{equation}
Then
\begin{equation}\label{Eq_QWW}
\Pf\big(Q(x;a,b)\big)=\frac{\det \big(W(x;a)\big) \det\big(W(x;b)\big)}
{\prod_{1\leq i<j\leq n}(x_i-x_j)(1-x_ix_j)}.
\end{equation}
\end{theorem}

Combining these two theorems we readily obtain the following result.

\begin{corollary}\label{Cor_key}
Let $n,r$ be positive integers such that $n$ is even, 
$\varepsilon\in\{\pm 1\}$, and $M=M(a,\varepsilon)$ is the $n\times r$ matrix 
with entries
\[
M_{ij}=x_i^{j-a}-\varepsilon x_i^{-j+a}.
\]
Then
\begin{multline*}
\sum_{\substack{J\subseteq \{1,\dots,r\} \\[1pt] \abs{J}=n}}
\det_{\substack{1\leq i\leq n \\[1pt] j\in J}}\big(M_{ij}\big)
=(-1)^{n/2} \det_{1\leq i,j\leq n} 
\Big(x_i^{r/2+n/2-j-a+1}-\varepsilon x_i^{-(r/2+n/2-a-j+1)}\Big) \\
\times
\frac{\det_{1\leq i,j\leq n}
\big(x_i^{r/2+n/2-j+1/2}-x_i^{-(r/2+n/2-j+1/2)}\big)}
{\det_{1\leq i,j\leq n} \big(x_i^{n-j+1/2}-x_i^{-(n-j+1/2)}\big)}.
\end{multline*}
\end{corollary}

\begin{proof}
A routine calculation using the summation of geometric series shows that 
for the above choice of $M$, the matrix $B$ in \eqref{Eq_B} is given by
\[
B_{ij}=\sum_{1\leq k<l\leq r}
\big(M_{ik}M_{jl}-M_{il}M_{jk}\big)
=\frac{(x_ix_j)^{a-r}}{(1-x_i)(1-x_j)}\,
Q_{ij}\big(x;x^r,\varepsilon x^{r-2a+1}\big),
\]
where $x^a$ is shorthand for $(x_1^a,\dots,x_n^a)$.
Since $\Pf(u_i u_j v_{ij})=\big(\prod_i u_i\big) \Pf(v_{ij})$, we obtain
\begin{align*}
\Pf(B)&=\bigg(\prod_{i=1}^n \frac{x_i^{a-r}}{1-x_i}\bigg)
\Pf\Big(Q\big(x;x^r,\varepsilon x^{r-2a+1}\big)\Big) \\
&=\frac{\det\big(W(x;x^r)\big)\det\big(W(x;\varepsilon x^{r-2a+1})\big)}
{\prod_{i=1}^n x_i^{r-a}(1-x_i)\prod_{1\leq i<j\leq n}(x_i-x_j)(1-x_ix_j)}.
\end{align*}
Use of \eqref{Eq_W}, the $\mathrm{B}_n$ Vandermonde determinant
\eqref{Eq_Bn-denom}, and the fact that $n$ is even and $\varepsilon^2=1$ 
completes the proof.
\end{proof}

\begin{proof}[Proof of Theorem~\ref{Thm_sososo}]
{From} (see, e.g., \cite[Lemma~5.3(2)]{Okada98})
\begin{equation}\label{Eq_reduce}
\lim_{x_n\to 0} x_n^r \so^{\sigma}_{2n+1,\la}(x_1,\dots,x_n)
=\begin{cases}
\so^{\sigma}_{2n-1,\mu}(x_1,\dots,x_{n-1}),
& \text{if $r=\la_1$}, \\ 0, & \text{if $r>\la_1$},
\end{cases}
\end{equation}
for $\mu:=(\la_2,\dots,\la_{n-1})$,
it follows that, if we multiply both sides of \eqref{Eq_sososo} 
by $x_n^r$ and let $x_n$ tend to zero, 
we obtain \eqref{Eq_sososo} with $n$ replaced by $n-1$.
Hence it suffices to prove the claim for even values of $n$.

\medskip

Let $S_r$ denote the left-hand side of \eqref{Eq_sososo}.
{From} \eqref{Eq_soplusso}, it follows that 
\[
S_r=\sum_{\la\subseteq (r^n)} \so^{\sigma}_{2n+1,\la}(x).
\]
By \eqref{Eq_so2n1-def} and \eqref{Eq_soplus-def}, this can also
be written as
\[
S_r=\frac{\sum_{\la\subseteq (r^n)} \det_{1\leq i,j\leq n}
\big(x_i^{\la_j+n-j+1/2}-\varepsilon x_i^{-(\la_j+n-j+1/2)}\big)}
{\det_{1\leq i,j\leq n}\big(x_i^{n-j+1/2}-\varepsilon x_i^{-(n-j+1/2)}\big)}.
\]
If we replace the sum over $\la$ by a sum over $k_1,\dots,k_n$ via
the substitution 
\[
\la_j=k_{n-j+1}-\rho_j-1/2, \quad\text{for $1\leq j\leq n$},
\]
and reverse the order of the columns in the determinant, this leads to
\[
S_r=(-1)^{\binom{n}{2}}
\sum_{1\leq k_1<k_2<\dots<k_n\leq r+n}\,
\frac{\det_{1\leq i,j\leq n}\big(x_i^{k_j-1/2}-\varepsilon 
x_i^{-(k_j-1/2)}\big)}
{\det_{1\leq i,j\leq n}\big(x_i^{n-j+1/2}-\varepsilon x_i^{-(n-j+1/2)}\big)}.
\]
Now assume that $n$ is even. We can then apply Corollary~\ref{Cor_key}
with $r\mapsto r+n$ and $a=1/2$ to find
\begin{multline*}
S_r=\frac{\det_{1\leq i,j\leq n} 
\big(x_i^{r/2+n-j+1/2}-\varepsilon x_i^{-(r/2+n-j+1/2)}\big)} 
{\det_{1\leq i,j\leq n}\big(x_i^{n-j+1/2}-\varepsilon x_i^{-(n-j+1/2)}\big)} \\
\times
\frac{\det_{1\leq i,j\leq n}
\big(x_i^{r/2+n-j+1/2}-x_i^{-(r/2+n-j+1/2)}\big)}
{\det_{1\leq i,j\leq n} \big(x_i^{n-j+1/2}-x_i^{-(n-j+1/2)}\big)}.
\end{multline*}
Finally, recalling \eqref{Eq_so2n1-def} and \eqref{Eq_soplus-def},
we obtain
\[
S_r=\so^{\sigma}_{2n+1,(s^n)}(x) \so_{2n+1,(s^n)}(x),
\]
with $s=\frac{1}{2}r$.
\end{proof}

\begin{proof}[Proof of Theorem~\ref{Thm_sososo-2}]
Equation \ref{Eq_reduce} once again holds when $\so^{\sigma}_{2n+1,\la}$ is 
replaced by $\so_{2n,\la}$ or $\ortho_{2n,\la}$, so that we may again take
$n$ to be even.

Let $S_r$ and $S'_r$ denote the left-hand sides of \eqref{Eq_sososo-2} 
and \eqref{Eq_ooso}, respectively.
Using \eqref{Eq_so2n} and \eqref{Eq_o2n} and making the substitutions
\[
\begin{cases}
S_r:\quad \la_j=k_{n-j+1}-n+j-1, \\
S'_r:\quad \la_j=k_{n-j+1}-n+j ,
\end{cases}
\qquad
\text{for $1\leq j\leq n$},
\]
we get
\[
S_r=(-1)^{\binom{n}{2}}
\sum_{\sigma\in\{\pm 1\}} \sum_{1\leq k_1<k_2<\dots<k_n\leq r+n}\,
\frac{\det_{1\le i,j\le n}\big(\sigma x_i^{k_j-1}+x_i^{-(k_j-1)}\big)}
{\det_{1\le i,j\le n}\big(x_i^{n-j}+x_i^{-(n-j)}\big)}
\]
and
\[
S'_r=2(-1)^{\binom{n}{2}}
\sum_{1\leq k_1<k_2<\dots<k_n\leq r+n-1}\,
\frac{\det_{1\le i,j\le n}\big(x_i^{k_j}+x_i^{-k_j}\big)}
{\det_{1\le i,j\le n}\big(x_i^{n-j}+x_i^{-(n-j)}\big)}.
\]
By Corollary~\ref{Cor_key} with $r\mapsto r+n$, $a=1$, $\varepsilon=-\sigma$
and $r\mapsto r+n-1$, $a=0$, $\varepsilon=-1$, respectively, this yields
\begin{multline*}
S_r=\frac{\sum_{\sigma\in\{\pm 1\}} \det_{1\leq i,j\leq n} 
\big(\sigma x_i^{r/2+n-j}+x_i^{-(r/2+n-j)}\big)}
{\det_{1\le i,j\le n}\big(x_i^{n-j}+x_i^{-(n-j)}\big)} \\
\times
\frac{\det_{1\leq i,j\leq n}
\big(x_i^{r/2+n/2-j+1/2}-x_i^{-(r/2+n/2-j+1/2)}\big)}
{\det_{1\leq i,j\leq n} \big(x_i^{n-j+1/2}-x_i^{-(n-j+1/2)}\big)}
\end{multline*}
and
\begin{multline*}
S'_r=2\,\frac{\det_{1\leq i,j\leq n} 
\big(x_i^{r/2+n-j+1/2}+x_i^{-(r/2+n-j+1/2)}\big)} 
{\det_{1\le i,j\le n}\big(x_i^{n-j}+x_i^{-(n-j)}\big)} \\
\times \frac{\det_{1\leq i,j\leq n}
\big(x_i^{r/2+n-j}-x_i^{-(r/2+n-j)}\big)}
{\det_{1\leq i,j\leq n} \big(x_i^{n-j+1/2}-x_i^{-(n-j+1/2)}\big)},
\end{multline*}
where we have used that $n$ is even.
{From} \eqref{Eq_so2n1-def} and \eqref{Eq_so2n}, we see that the 
expression for $S_r$ is exactly
\[
\so_{2n,(s^n)}(x) \so_{2n+1,(s^n)}(x),\qquad s:=\tfrac{1}{2}r,
\]
and that for $S'_r$
\[
\ortho_{2n,(s^n)}(x) \so_{2n+1,(t^n)}(x),\qquad
s:=\tfrac{1}{2}(r+1),~t:=\tfrac{1}{2}(r-1). \qedhere
\]
\end{proof}

\section{Discrete Macdonald--Mehta integrals for $\gamma=1/2$}
\label{sec:gamma=1/2}

We will slightly extend our earlier definition \eqref{Eq_Srdef} by
considering 
$\mathcal{S}_{r,n}(\alpha,\gamma,\delta)$ for $n$ a non-negative integer
or half-integer.
In the latter case, the sum over $k_1,\dots,k_r$ is assumed to range
over half-integers, so that in both cases the $k_i$ are summed over 
$\{-n,-n+1,\dots,n\}$.

\subsection{The evaluation of $\mathcal{S}_{r,n}(1,\frac{1}{2},0)$} 
Instead of computing this sum directly, we first consider a $q$-analogue.

\begin{proposition}[$\mathrm{A}_{r-1}$ summation]\label{Prop_A-sum}
Let $0<q<1$, $r$ a positive integer and $n$ an integer or half-integer
such that $n\geq (r-1)/2$. Then
\begin{multline}\label{Eq_An-sum}
\sum_{k_1,\dots,k_r=-n}^n \,
\prod_{1\leq i<j\leq r} \Abs{[k_i-k_j]_q} \,
\prod_{i=1}^r q^{(k_i+n-r+i)^2/2} \qbin{2n}{n+k_i} \\
=\frac{r!}{[r]_{q^{1/2}}!} 
\prod_{i=1}^r (-q^{1/2};q^{1/2})_i (-q^{i/2+1};q)_{2n-r} \\
\times \prod_{i=1}^r 
\frac{\Gamma_q(1+\frac{1}{2}i)}{\Gamma_q(\tfrac{3}{2})}
\cdot \frac{\Gamma_q(2n+1)\,\Gamma_q(2n-i+\frac{5}{2})}
{\Gamma_q(2n-i+2)\,\Gamma_q(2n-\frac{1}{2}i+2)}. 
\end{multline}
\end{proposition}
Taking the
$q\to 1$ limit, we arrive at (cf.\ \eqref{eq:1any0})
\begin{align} \label{eq:S110}
\mathcal{S}_{r,n}(1,\tfrac{1}{2},0)&=
\sum_{k_1,\dots,k_r=-n}^n \, \prod_{1\leq i<j\leq r} \abs{k_i-k_j}\,
\prod_{i=1}^r \binom{2n}{n+k_i} \\ \notag
&=2^{2rn-\binom{r}{2}}\, \prod_{i=1}^r
\frac{\Gamma(1+\frac{1}{2}i)}{\Gamma(\tfrac{3}{2})}\cdot
\frac{\Gamma(2n+1)\,\Gamma(2n-i+\frac{5}{2})}
{\Gamma(2n-i+2)\,\Gamma(2n-\frac{1}{2}i+2)}.
\end{align}
The evaluation of $S_{1,1}(n)$ in \cite[Equation~(5.6)]{BOOPAA} is the
special case $r=2$ of this identity.

\begin{proof}[Proof of Proposition~\ref{Prop_A-sum}]
Denote the sum on the left of \eqref{Eq_An-sum} by $f_{n,r}$.
Since
\begin{multline*}
q^{\sum_{i=1}^r (k_i+n-r+i)^2/2} 
\prod_{1\leq i<j\leq r} \Abs{1-q^{k_i-k_j}} \\
=q^{(n-1/2)\binom{r}{2}-2\binom{r}{3}+\sum_{i=1}^r (k_i+n-r+1)^2/2}
\prod_{1\leq i<j\leq r} \Abs{q^{k_i}-q^{k_j}},
\end{multline*}
the summand of $f_{n,r}$ is a symmetric function which 
vanishes unless all $k_i$ are pairwise distinct.
Anti-symmetrisation thus yields
\[
f_{n,r}=\frac{r!}{(1-q)^{\binom r2}}\sum_{n\geq k_1>\cdots>k_r\geq -n}\,
\prod_{1\leq i<j\leq r} \big(1-q^{k_i-k_j}\big)
\prod_{i=1}^r q^{(k_i+n-r+i)^2/2} \qbin{2n}{n+k_i}.
\]
We write this as a sum over partitions $\la\subseteq (r^{2n-r+1})$ via
\[
k_i=\lambda'_i-n+r-i,\qquad 1\leq i\leq r.
\]
Then
\[
f_{n,r}=\frac{r!}{(1-q)^{\binom r2}}
\sum_{\la\subseteq (r^{2n-r+1})}q^{n(\la)+\abs{\la}/2}
\prod_{1\leq i<j\leq r} \big(1-q^{\la_i'-\la'_j+j-i}\big)
\prod_{i=1}^r\qbin{2n}{\la_i'+r-i}.
\]
By Lemma~\ref{Lem_Schur_PSprime} and the fact that $s_{\la}$ is homogeneous 
of degree $\abs{\la}$, this can be written as a sum over
principally specialised Schur functions. Performing in addition the 
replacement $n\mapsto (n+r-1)/2$, we arrive at
\begin{multline*}
f_{(n+r-1)/2,r}=\frac{r!}{(1-q)^{\binom r2}}
\prod_{1\leq i<j\leq r} \big(1-q^{j-i}\big)
\prod_{i=1}^r\qbin{n+r-1}{r-i}\\
\times
\sum_{\la\subseteq (r^n)}
s_{\la}\big(q^{1/2},q^{3/2},\dots,q^{n-1/2}\big),
\end{multline*}
for $n$ a non-negative integer. (When $n=0$, the sum on the right should
be interpreted as $1$.)
The sum can be computed by \cite[p.~85]{Macdonald95}\footnote{This is 
equivalent to MacMahon's formula \cite{MacMahon98} for the generating
function of symmetric plane partitions that fit in a box of size 
$n\times n\times r$, proved by Andrews \cite{Andrews78} and 
Macdonald \cite{Macdonald95}.}
\[
\sum_{\la\subseteq
(r^n)}s_{\la}\big(q^{1/2},q^{3/2},\dots,q^{n-1/2}\big)
=\prod_{i=1}^n \frac{1-q^{i+(r-1)/2}}{1-q^{i-1/2}}
\prod_{1\leq i<j\leq n} \frac{1-q^{r+i+j-1}}{1-q^{i+j-1}}.
\]
Some elementary simplifications of the $q$-products and
the subsequent replacement $n\mapsto 2n-r+1$ result in
\[
f_{n,r}=\frac {r!} {(1-q)^{\binom r2}}\,
\frac{(q^{(r+1)/2};q)_{2n-r+1}}{(q^{1/2};q)_{2n-r+1}}
\prod_{i=1}^r \frac{(q;q)_{2n}(q;q)_{i-1}(q^i;q^2)_{2n-r+1}}{(q;q)_{2n-i+1}^2}.
\]
To transform this into the claimed product over $q$-gamma functions is 
somewhat delicate. First we use $(a^2;q^2)_n=(a;q)_n(-a;q)_n$ to write
\[
\frac{(q^{(r+1)/2};q)_{2n-r+1}}{(q^{1/2};q)_{2n-r+1}}
\prod_{i=1}^r \frac{(q^i;q^2)_{2n-r+1}}{(q;q)_{2n-i+1}}=
\prod_{i=1}^r \frac{(-q^{i/2};q)_{2n-r+1} 
(q^{(i+1)/2};q)_{2n-r+1}}{(q;q)_{2n-i+1}}.
\]
The first term in the numerator is wanted, but we further need to
transform the
other two terms as follows:
\[
\prod_{i=1}^r \frac{(q^{(i+1)/2};q)_{2n-r+1}}{(q;q)_{2n-i+1}}
=\prod_{i=1}^r \frac{(-q^{1/2};q^{1/2})_{i-1} (q^{3/2};q)_{2n-i+1}}
{(q;q)_{i-1}(q^{i/2+1};q)_{2n-i+1}}\cdot \frac{1-q^{1/2}}{1-q^{i/2}}.
\]
Putting all this together, we get
\[
f_{n,r}=\frac{r!}{(1-q)^{\binom r2}\,[r]_{q^{1/2}}!} \,
\prod_{i=1}^r (-q^{1/2};q^{1/2})_i (-q^{i/2+1};q)_{2n-r} \cdot
\frac{(q;q)_{2n}(q^{3/2};q)_{2n-i+1}}{(q;q)_{2n-i+1}(q^{i/2+1};q)_{2n-i+1}}.
\]
By the definition of the $q$-gamma function, the result now follows.
\end{proof}

\subsection{The evaluation of $\mathcal{S}_{r,n}(2,\frac{1}{2},1)$} 
Again we first consider a $q$-analogue.

\begin{proposition}[$\mathrm{B}_r$ summation]\label{Prop_Br-sum}
Let $0<q<1$, $r$ a positive integer and $n$ an integer or 
half-integer such that $n\geq r-1/2$. Then
\begin{multline}
\sum_{k_1,\dots,k_r=-n}^n \, \prod_{1\leq i<j\leq r} 
\Abs{[k_i-k_j]_q\,[k_i+k_j]_q} \, \prod_{i=1}^r\,
\Abs{[k_i]_q} \, 
q^{\binom{k_i-r+i}{2}-\binom{\ceiling{n}-n}{2}} \qbin{2n}{n+k_i} \\
=2^r r!  \prod_{i=1}^r (-q;q^{1/2})_{2n-2i}\,
\frac{\Gamma_q(i)}{\Gamma_q(\frac{3}{2})}\cdot
\frac{\Gamma_q(2n+1)\,\Gamma_q(\ceiling{n}-i+\frac{3}{2})}
{\Gamma_q(2n-i+2)\,\Gamma_q(\ceiling{n}-i+1)}. 
\end{multline}
\end{proposition}
Taking the $q\to 1$ limit 
and using $r!\prod_{i=1}^r\Gamma(i)=\prod_{i=1}^r \Gamma(i+1)$,
we obtain (cf.\ \eqref{eq:2anyanygp4}) 
\begin{align}
\mathcal{S}_{r,n}(2,\tfrac {1} {2},1)&:=\sum_{k_1,\dots,k_r=-n}^n\,
\prod_{1\leq i<j\leq r} \Abs{k_i^2-k_j^2} \, 
\prod_{i=1}^r \, \abs{k_i} \binom{2n}{n+k_i} \\
&\hphantom{:}=2^{(2n+1)r-r(r+1)} \prod_{i=1}^r
\frac{\Gamma(1+i)}{\Gamma(\frac {3} {2})}\cdot
\frac{\Gamma(2n+1)\,\Gamma(n-i+\frac {3} {2})}
{\Gamma(2n-i+2)\,\Gamma(n-i+1)}.
\notag
\end{align}
Equation~(5.12) in \cite{BOOPAA} is the
special case $r=2$ of this identity.

\begin{proof}
Once again, the sum will be denoted by $f_{n,r}$.
This time the summand is symmetric under signed permutations 
of the $k_i$. Exploiting this hyperoctahedral symmetry, we obtain
\begin{multline*}
f_{n,r} = \frac {2^r r!} {(1-q)^{r^2}} \sum_{n\geq k_1>\cdots>k_r>0}\,
\prod_{1\leq i<j\leq r} (1-q^{k_i-k_j})(1-q^{k_i+k_j}) \\ \times
\prod_{i=1}^r (1-q^{k_i}) \, 
q^{\binom{k_i-r+i}{2}-\binom{\ceiling{n}-n}{2}} \qbin{2n}{n+k_i}.
\end{multline*}
We now set
\begin{equation}\label{Eq_k-lambda}
k_i=n-i-\la'_{r-i+1}+1,\qquad 1\leq i\leq r,
\end{equation}
where $\la$ is a partition contained in $(r^{\ceiling{n}-r})$. 
If we then replace $n\mapsto \floor{n}+r$ and use the dual $\mathrm{C}_n$
specialisation formula \eqref{Eq_SpecCnb} in the integer-$n$ case or the
dual $\mathrm{B}_n$ specialisation formula \eqref{Eq_so-odd-PS-2} with 
$q^{1/2}\mapsto -q^{1/2}$ in the half-integer case, we get
\[
f_{n+r,r}=\frac {2^r r! \, q^{r\binom{n+1}{2}}} {(1-q)^{r^2}}\, 
\frac{(q;q)_{n+r}}{(q;q)_n}
\prod_{i=1}^r \frac{(q;q)_{2n+2r}}{(q;q)_{2n+2r-2i+2}} 
\sum_{\la\subseteq (r^n)} \symp_{2n,\la}(q,q^2,\dots,q^n)
\]
and
\begin{multline*}
f_{n+r-1/2,r} = \frac {2^r r! \, q^{rn^2/2}} {(1-q)^{r^2}} \,
\frac{(q^{1/2};q)_{n+r}}{(q^{1/2};q)_n}
\prod_{i=1}^r \frac{(q;q)_{2n+2r-1}}{(q;q)_{2n+2r-2i+1}} \\
\times \sum_{\la\subseteq (r^n)} (-1)^{\abs{\la}} 
\so_{2n+1,\la}\big({-}q^{1/2},-q^{3/2},\dots,-q^{n-1/2}\big),
\end{multline*}
where $n$ is a non-negative integer. (The two sums on the right
are again to be interpreted as $1$ when $n=0$.)
By Corollaries~\ref{Cor_Bn-box} and \ref{Cor_Cn-box}, we can carry
out the summations, resulting in
\[
f_{n+r,r}=\frac {2^r r!} {(1-q)^{r^2}} \, \prod_{i=1}^r
\frac{(q;q)_{2n+2r}(q;q)_{2n+i}(q;q)_{i-1}}{(q;q)_{2n+2r-2i+2}(q;q)_{n+r-i}^2}
\]
and
\[
f_{n+r-1/2,r} = \frac {2^r r!} {(1-q)^{r^2}} \, 
\frac{(q^{1/2};q)_{n+r}}{(q^{1/2};q)_n}
\prod_{i=1}^r \frac{(q;q)_{2n+2r-1}(q;q)_{2n+i-1}(q;q)_{i-1}}
{(q;q)_{2n+2r-2i+1}(q;q)_{n+r-i}^2},
 \]
respectively. The replacement $n\mapsto n-r$ or $n\mapsto n-r+1/2$
and some elementary manipulations lead to
\[
f_{n,r}=\frac {2^r r!} {(1-q)^{r^2}} \, \prod_{i=1}^r
\frac{(-q^{1/2};q^{1/2})_{2n-2i+1}
(q^{1/2};q)_{\ceiling{n}-i+1}(q;q)_{2n}(q;q)_{i-1}}
{(q;q)_{2n-i+1}(q;q)_{\ceiling{n}-i}}.
\]
The proof is completed by writing this in terms of $q$-gamma functions.
\end{proof}

\subsection{The evaluation of $\mathcal{S}_{r,n}(2,\frac{1}{2},0)$} 

We first restate Proposition~\ref{Prop_Bn}, now including the
half-integral case (cf.\ \eqref{eq:2anyanygp4}).
\begin{proposition}[$\mathrm{D}_r$ summation]
Let $r$ be a positive integer and $n$ an integer or half-integer such 
that $n\geq r-1$. Then
\begin{align}\label{Eq_Dr-sum}
\mathcal{S}_{r,n}(2,\tfrac{1}{2},0)&=
\sum_{k_1,\dots,k_r=-n}^n\,
\prod_{1\leq i<j\leq r} \Abs{k_i^2-k_j^2} \,
\prod_{i=1}^r\binom{2n}{n+k_i} \\
&=2^{2rn-r(r-1)}\, 
\frac{\Gamma(1+\frac{1}{2}r)}{\Gamma(\frac{3}{2})}\cdot
\frac{\Gamma(\floor{n}-\frac{1}{2}r+\frac{3}{2})}{\Gamma(\floor{n}+1)} 
\notag \\
&\qquad\qquad\quad\times 
\prod_{i=1}^{r-1} \frac{\Gamma(i+1)}{\Gamma(\frac{3}{2})} \cdot
\frac{\Gamma(2n+1)\,\Gamma(\floor{n}-i+\frac{3}{2})}
{\Gamma(2n-i+1)\,\Gamma(\floor{n}-i+1)}. \notag
\end{align}
\end{proposition}
As already pointed out in the introduction, the special cases
$r=2,3,4$ cover \cite[Theorem 1]{BO}, and
Theorem~4.1 and Conjecture~4.1 in \cite{BOOPAA}, respectively.

When $n$ is a half-integer, the identity \eqref{Eq_Dr-sum} admits a 
$q$-analogue:
\begin{multline}
\sum_{k_1,\dots,k_r=-n}^n \,
\prod_{1\leq i<j\leq r} \Abs{[k_i-k_j]_q\,[k_i+k_j]_q} \, 
\prod_{i=1}^r q^{\binom{k_i-r+i+1/2}{2}} \qbin{2n}{n+k_i} \\
=2^r\, [2]_q^n \, 
\frac{r!}{[r]_q!} \cdot \frac{1}{(-q;q)_{n-r}}
\prod_{i=1}^r (-q;q^{1/2})_{2n-2i}  
\frac{\Gamma_{q^2}(1+\frac{1}{2}r)}{\Gamma_q(\frac{3}{2})}
\prod_{i=1}^{r-1} \frac{\Gamma_q(i+1)}{\Gamma_q(\frac{3}{2})} \\
\times \frac{\Gamma_{q^2}(n-\frac{1}{2}r+1)}
{\Gamma_q(n+\frac{1}{2})}
\prod_{i=1}^{r-1} \frac{\Gamma_q(2n+1)\,\Gamma_q(n-i+1)}
{\Gamma_q(2n-i+1)\,\Gamma_q(n-i+\frac{1}{2})}.
\end{multline}

\begin{proof}
As usual, we denote the sum on the left by $f_{n,r}$. Due to the
hyperoctahedral symmetry of the summand, we have
\[
f_{n,r}=
r!\cdot 2^r \sum_{n\geq k_1>\cdots>k_r\geq 0}(1-\tfrac{1}{2}\delta_{k_r,0})
\prod_{1\leq i<j\leq r} (k_i^2-k_j^2) \prod_{i=1}^r \binom{2n}{n+k_i}.
\]
Since $n+k_r$ must be an integer, the effective lower
bound is $1/2$ when $n$ is a half-integer. In this case
$(1-\frac{1}{2}\delta_{k_r,0})=1$.

Again we make the variable change \eqref{Eq_k-lambda}.
Due to the different lower bound compared to the $\mathrm{B}_r$ summation 
in Proposition~\ref{Prop_Br-sum}, this means that we will now be summing over 
partitions $\la$ contained in $r^{\floor{n}-r+1}$.
We also note that in the integer-$n$ case 
$(1-\frac{1}{2}\delta_{k_r,0})$ transforms into
\begin{equation}\label{Eq_delta-factor}
(1-\tfrac{1}{2}\delta_{\la'_1,n-r+1})=
(1-\tfrac{1}{2}\delta_{l(\la),n-r+1}).
\end{equation}
Next we replace $n\mapsto \ceiling{n}+r-1$. Note that this turns
\eqref{Eq_delta-factor} into
\[
(1-\tfrac{1}{2}\delta_{l(\la),n})=u_{\la}^{-1},
\]
with $u_{\la}$ as in \eqref{Eq_o2n}.
In the integer-$n$ case, we can then use
\eqref{Eq_ortho-PS-dual} for $q=1$ combined with
\eqref{Eq_ortho-so-even} to find
\[
f_{n+r-1,r}=2^r r! \,  
\prod_{i=1}^r\frac{(2n+2r-2)!}{(2n+2r-2i)!} \, \sum_{\la\subseteq (r^n)}
\so_{2n,\la}(1^n).
\]
In the half-integer case, we can use the $q=1$ instance of
\eqref{Eq_so-odd-PS-1}. This results in
\begin{equation}\label{Eq_Dr-halfq1-a}
f_{n+r-1/2,r}=2^r r! \,
\prod_{i=1}^r \frac{(2n+2r-1)!}{(2n+2r-2i+1)!} \, 
\sum_{\la\subseteq (r^n)} \so_{2n+1,\la}(1^n).
\end{equation}
As before, the sums on the right are $1$ when $n=0$.
Evaluation of these sums for general $n$ by the $q=1$ cases of 
\eqref{Eq_ortho-Okada-spec-1} and \eqref{Eq_B-sum-b} 
(with $\varepsilon=1$), respectively, gives
\[
f_{n+r-1,r}=2^{2n+r-1} \,
\frac{(\frac{1}{2}r+\frac{1}{2})_n}{n!} 
\prod_{i=1}^r\frac{(2n+2r-2)!}{(2n+2r-2i)!}
\prod_{i=1}^{r-1}\frac{(2n+i-1)!\,(i+1)!}{(n+r-i)!\,(n+r-i-1)!}
\]
and
\begin{equation}\label{Eq_Dr-halfq1-b}
f_{n+r-1/2,r}(q)=2^r \, 
\frac{(\frac{1}{2}r+\frac{1}{2})_n}{(\frac{1}{2})_n} 
\prod_{i=1}^r \frac{(2n+2r-1)!\,(2n+i-1)!\,i!}{(2n+2r-2i+1)!\,(n+r-i)!^2}.
\end{equation}
Replacing $n\mapsto n-r+1$ or $n\mapsto n-r+1/2$ and then expressing
$f_{n,r}$ in terms of gamma functions, we arrive at the right-hand side of 
\eqref{Eq_Dr-sum}.

The proof of the $q$-case for half-integer $n$ proceeds along exactly
the same lines, with \eqref{Eq_Dr-halfq1-a} replaced by
\[
f_{n+r-1/2,r}(q)=\frac {2^r r!} {(1-q)^{r^2-r}} \, q^{r\binom{n+1}{2}} 
\prod_{i=1}^r \frac{(q;q)_{2n+2r-1}}{(q;q)_{2n+2r-2i+1}} \,
\sum_{\la\subseteq r^n} \so_{2n+1,\la}(q,q^2,\dots,q^n).
\]
and \eqref{Eq_Dr-halfq1-b} by
\[
f_{n+r-1/2,r}(q)=\frac {2^r r!} {(1-q)^{r^2-r}} \, 
\frac{(q^{r+1};q^2)_n}{(q;q^2)_n}
\prod_{i=1}^r \frac{(q;q)_{2n+2r-1}(q;q)_{2n+i-1}(q;q)_{i-1}}
{(q;q)_{2n+2r-2i+1}(q;q)_{n+r-i}^2}.  \qedhere
\]
\end{proof}

\subsection{The evaluation of $\mathcal{S}_{r,n}(2,\frac{1}{2},2)$} 

This is the $(\alpha,\gamma)=(2,1/2)$ case in Table~\ref{Table_ten}.
It has no interpretation in terms of finite reflection groups. 
It is also the only case that apparently 
does not admit a simple closed-form product formula for half-integer $n$.

\begin{proposition}
Let $0<q<1$, $r$ a positive integer and $n$ an integer such that $n\geq r$. 
Then
\begin{multline}
\sum_{k_1,\dots,k_r=-n}^n  \, \prod_{1\leq i<j\leq r} 
\Abs{[k_i-k_j]_q\,[k_i+k_j]_q} 
\prod_{i=1}^r \, [k_i]_q^2 \, q^{(k_i-r+i-1)^2/2} 
\qbin{2n}{n+k_i} \\
=2^r  \, \frac{r!}{[r]!_q}  \cdot 
\frac{(-1;q^{1/2})_{r+1} (-q^{r/2+1};q)_{n-r}}{(-1;q)_{n-r}} 
\prod_{i=1}^{r+1} \frac{(-q^{1/2};q^{1/2})_{i-1}(-q^{1/2};q^{1/2})_{2n-i-r}}
{(-q^{1/2};q^{1/2})_{2n-2i+2}} \\ 
\times
\prod_{i=1}^r
\frac{\Gamma_q^2(1+\tfrac{1}{2}i)}{\Gamma_q(\tfrac{1}{2})}\cdot
\frac{\Gamma_q(2n+1)\,\Gamma_q(n-i+\frac{3}{2})}
{\Gamma_q(n-i+1)\,\Gamma_q^2(n-\frac{1}{2}i+1)}.
\end{multline}
\end{proposition}

In the $q\to 1$ limit, this becomes (cf.\ \eqref{eq:2anyanygp4})
\begin{align}
\mathcal{S}_{r,n}(2,\tfrac{1}{2},2)&=
\sum_{k_1,\dots,k_r=-n}^n\,
\prod_{1\leq i<j\leq r} \Abs{k_i^2-k_j^2} \, 
\prod_{i=1}^r k_i^2 \binom{2n}{n+k_i} \\
&=2^r\, \prod_{i=1}^r
\frac{\Gamma^2(1+\tfrac{1}{2}i)}{\Gamma(\tfrac{1}{2})}\cdot
\frac{\Gamma(2n+1)\,\Gamma(n-i+\frac{3}{2})}
{\Gamma(n-i+1)\,\Gamma^2(n-\frac{1}{2}i+1)}.
\end{align}
Equation~(5.13) in \cite{BOOPAA} is the
special case $r=2$ of this identity.

\begin{proof}
If we denote the sum on the left by $f_{n,r}$ and define $k_{r+1}:=0$, then 
the summand of $f_{n,r}$ can be rewritten as
\[
\qbin{2n}{n}^{-1} \prod_{1\leq i<j\leq r+1} 
\Abs{[k_i-k_j]_q\,[k_i+k_j]_q} \,
\prod_{i=1}^{r+1} q^{(k_i-r+i-1)^2/2} \qbin{2n}{n+k_i}.
\]
Hence, after anti-symmetrisation and the variable change
\[
k_i=n-i-\la'_{r-i+2}+1,\qquad 1\leq i\leq r+1
\]
(so that $\la'_1:=n-r$), we obtain
\begin{multline*}
f_{n+r-1,r-1}=\frac{2^{r-1} (r-1)! \, q^{rn^2/2}} {(1-q)^{r^2+r}}
\, \qbin{2n+2r-2}{n+r-1}^{-1} 
\prod_{i=1}^r \frac{(q;q)_{2n+2r-2}}{(q;q)_{2n+2i-2}} \\ \times
\sum_{\substack{\la\subseteq (r^n) \\[1pt] l(\la)=n}}
\ortho_{2n,\la}(q^{1/2},q^{3/2},\dots,q^{n-1/2}),
\end{multline*}
where we have also used \eqref{Eq_ortho-PS-dual}.
Next we apply \eqref{Eq_ortho-Okada-spec-2} so that
\begin{multline*}
f_{n+r-1,r-1}=2^r (r-1)! \,
\frac{(-q^{1/2};q^{1/2})_{2n-1} (-q^{(r+1)/2};q)_n}{(-1;q)_n}\\
\times \frac{(q^{r/2};q)_n}{(q;q)_n}
\qbin{2n+2r-2}{n+r-1}^{-1} 
\prod_{i=1}^{r-1} \frac{(q;q)_{2n+2r-2}(q;q)_{2n+i-1}(q;q)_i}
{(q;q)_{2n+2r-2i-2}(q;q)_{n+r-i}(q;q)_{n+r-i-1}}.
\end{multline*}
The rest follows as in earlier cases.
\end{proof}

\section{Discrete Macdonald--Mehta integrals for $\gamma=1$ and $\alpha=2$}
\label{sec:gamma=1-2}

In this section, we present our results concerning evaluations of 
$S_{r,n}(\alpha,\gamma,\delta)$ for $\gamma=1$ and $\alpha=2$. 
In contrast to the previous section, where identities for classical group
characters played a key role, here our starting point is a transformation 
formula for elliptic hypergeometric series.
Along the lines of Section~\ref{sec:gamma=1/2}, in each
case we shall start with a $q$-analogue, from which the evaluations of 
$S_{r,n}(2,1,\delta)$ for $\delta=0,1,2,3$ follow by a straightforward 
$q\to 1$ limit.
An additional feature is that the identities in this section typically
contain an additional parameter.

{\allowdisplaybreaks
\medskip
We start with the $p=0$, $x=q$ special case of a transformation
formula originally conjectured by the third author 
\cite[Conjecture~6.1]{WarnAG} and proven independently by 
Rains~\cite[Theorem~4.9]{RainAA} and
by Coskun and Gustafson~\cite{CoGuAA}.

\begin{theorem} \label{thm:Rains}
Let $a,b,c,d,e,f$ be indeterminates, $m$ a non-negative integer,
and $r\ge 1$. Then 
\begin{multline} \label{eq:Rains}
\sum_{0\leq k_1<k_2<\dots<k_r\leq m}
q^{\sum_{i=1}^r(2i-1)k_i}\prod_{1\leq i<j\leq r}
(1-q^{k_i-k_j})^2\,(1-aq^{k_i+k_j})^2\\\times
\prod_{i=1}^r\frac{(1-aq^{2k_i})\,(a,b,c,d,e,f,
\la aq^{2-r+m}/ef,q^{-m};q)_{k_i}}
{(1-a)\,(q,aq/b,aq/c,aq/d,aq/e,aq/f,efq^{r-1-m}/\la,aq^{1+m};q)_{k_i}}\\
=\prod_{i=1}^r\frac{(b,c,d,ef/a;q)_{i-1}}
{(\la b/a,\la c/a,\la d/a,ef/\la;q)_{i-1}}
\kern6cm
\\\times
\prod_{i=1}^r\frac{(aq;q)_m\,(aq/ef;q)_{m-r+1}\,
(\la q/e,\la q/f;q)_{m-i+1}}
{(\la q;q)_m\,(\la q/ef;q)_{m-r+1}\,
(aq/e,aq/f;q)_{m-i+1}}\\\times
\sum_{0\leq k_1<k_2<\dots<k_r\leq m}
q^{\sum_{i=1}^r(2i-1)k_i}\prod_{1\leq i<j\leq r}
(1-q^{k_i-k_j})^2\,(1-\la q^{k_i+k_j})^2\\\times
\prod_{i=1}^r\frac{(1-\la q^{2k_i})\,(\la,\la b/a,\la c/a,\la d/a,e,f,
\la aq^{2-r+m}/ef,q^{-m};q)_{k_i}}
{(1-\la)\,(q,aq/b,aq/c,aq/d,\la q/e,\la q/f,efq^{r-1-m}/a,
\la q^{1+m};q)_{k_i}},
\end{multline}
where $\la=a^2q^{2-r}/bcd$.
\end{theorem}

In the above formula, we let $m\to\infty$ to obtain
\begin{multline} \label{eq:Rains2}
\sum_{0\leq k_1<k_2<\dots<k_r}
q^{\sum_{i=1}^r(2i-1)k_i}\prod_{1\leq i<j\leq r}
(1-q^{k_i-k_j})^2\,(1-aq^{k_i+k_j})^2\\\times
\prod_{i=1}^r \Big(\frac{a^2}{q^{2r-3}bcdef}\Big)^{k_i}
\frac{(1-aq^{2k_i})\,(a,b,c,d,e,f;q)_{k_i}}
{(1-a)\,(q,aq/b,aq/c,aq/d,aq/e,aq/f;q)_{k_i}}\\
=\prod_{i=1}^r\frac{(b,c,d,ef/a;q)_{i-1}}
{(\la b/a,\la c/a,\la d/a,ef/\la;q)_{i-1}}
\frac{(aq,aq/ef,\la q/e,\la q/f;q)_\infty}
{(\la q,\la q/ef,aq/e,aq/f;q)_\infty}\\\times
\sum_{0\leq k_1<k_2<\dots<k_r}
q^{\sum_{i=1}^r(2i-1)k_i}\prod_{1\leq i<j\leq r}
(1-q^{k_i-k_j})^2\,(1-\la q^{k_i+k_j})^2\\\times
\prod_{i=1}^r \Big(\frac{a}{q^{r-1}ef}\Big)^{k_i}
\frac{(1-\la q^{2k_i})\,(\la,\la b/a,\la c/a,\la d/a,e,f;q)_{k_i}}
{(1-\la)\,(q,aq/b,aq/c,aq/d,\la q/e,\la q/f;q)_{k_i}}.
\end{multline}
The two specialisations which are relevant for us are 
$b = aq/c$ and $b = aq^2/c$.
The case $b=aq/c$ (which has the effect of generating terms
$(\la d/a;q)_{k_i}=(q^{1-r};q)_{k_i}$ in the right-hand side sum
of \eqref{eq:Rains2}, in turn implying that the only choice for
the summation indices $k_i$ to produce a non-vanishing summand
is $k_i=i-1$ for $i=1,2,\dots,r$) gives
\begin{multline} \label{eq:Rains3}
\sum_{0\leq k_1<k_2<\dots<k_r}
q^{\sum_{i=1}^r(2i-1)k_i}\prod_{1\leq i<j\leq r}
(1-q^{k_i-k_j})^2\,(1-aq^{k_i+k_j})^2\\\times
\prod_{i=1}^r \Big(\frac {a} {q^{2r-2}def}\Big)^{k_i}
\frac{(1-aq^{2k_i})\,(a,d,e,f;q)_{k_i}}
{(1-a)\,(q,aq/d,aq/e,aq/f;q)_{k_i}}\\
=q^{-\binom r3}\Big(\frac {a} {ef}\Big)^{\binom r2}
\prod_{i=1}^r\frac{(q,d,e,f,ef/a,aq^{i-r}/d;q)_{i-1}\,(aq^{2-r}/d;q)_{2i-2}}
{(aq/d,aq^{2-r}/de,aq^{2-r}/df,defq^{r-1}/a;q)_{i-1}}
\kern3cm
\\\times
\prod_{i=1}^r\frac{(aq,aq/ef,aq^{2-r}/de,aq^{2-r}/df;q)_\infty}
{(aq^{2-r}/d,aq^{2-r}/def,aq/e,aq/f;q)_\infty}.
\end{multline}
The case where $b=aq^2/c$ (which has the effect of generating terms
$(\la d/a;q)_{k_i}=(q^{-r};q)_{k_i}$ in the right-hand side sum
of \eqref{eq:Rains2}, in turn implying that the only choices for
the summation indices $k_i$ to produce a non-vanishing summand
are $k_i=i-1+\chi(i>s)$ for some non-negative integer~$s$, 
and $i=1,2,\dots,r$; 
here, $\chi(\mathcal A)=1$ if $\mathcal A$ is true and $\chi(\mathcal
A)=0$ otherwise) gives
\begin{multline} \label{eq:Rains4}
\sum_{0\leq k_1<k_2<\dots<k_r}
q^{\sum_{i=1}^r(2i-1)k_i}\prod_{1\leq i<j\leq r}
(1-q^{k_i-k_j})^2\,(1-aq^{k_i+k_j})^2\\\times
\prod_{i=1}^r \Big(\frac {a} {q^{2r-1}def}\Big)^{k_i}
\frac{(1-aq^{2k_i})\,(1-cq^{k_i-1})\,(1-aq^{k_i+1}/c)\,(a,d,e,f;q)_{k_i}}
{(1-a)\,(1-c/q)\,(1-aq/c)\,(q,aq/d,aq/e,aq/f;q)_{k_i}}\\
=(-1)^r q^{2\binom {r+1}3}
\Big(\frac {a} {q^{r-1}ef}\Big)^{\binom {r+1}2}
\frac   {(aq^{2-r}/cd)_r\,
     (cq^{-r}/d)_r} {(1-a q/c)^r\,
         (1-c/q)^r\,(q;q)_r\,(aq^{1-r}/d)_r^2}
\kern2cm\\
\times
\prod_{i=1}^r\frac{(q,e,f,aq^{i-r}/d;q)_{i}\,(d,ef/a;q)_{i-1}\,
                   (aq^{1-r}/d;q)_{2i}}
{(aq/d,aq^{1-r}/de,aq^{1-r}/df;q)_{i}\,(defq^{r}/a;q)_{i-1}}
\kern3cm
\\\times
\prod_{i=1}^r\frac{(aq,aq/ef,aq^{1-r}/de,aq^{1-r}/df;q)_\infty}
{(aq^{1-r}/d,aq^{1-r}/def,aq/e,aq/f;q)_\infty}\\
\times
\sum_{s=0}^r
  \frac {\big(1-\frac{aq^{-r}}{d}\,q^{2s}\big) } 
        {\big(1-\frac{aq^{-r}}{d}\big)}
\frac {(aq^{-r}/d,c/q,aq/c,aq^{1-r}/de,aq^{1-r}/df,q^{-r};q)_s} 
   {(q,aq^{2-r}/cd,cq^{-r}/d,e,f,aq/d;q)_s}
\Big(\frac {q^{r}ef} {a}\Big)^s.
\end{multline}
This is a transformation formula between a multiple basic hypergeometric
series associated with the root system $\BC_r$ and a very-well-poised
basic hypergeometric $_8\phi_7$-series (see \cite{GR04} for terminology).
}

\subsection{The evaluation of $\mathcal{S}_{r,n}(2,1,0)$} 

\begin{proposition}[$\mathrm{D}_r$ summation]
\label{thm:S220q}
Let $q$ be a real number with $0<q<1$.
For all non-negative integers or half-integers $m$ and $n$ and a positive
integer~$r$, we have
\begin{multline} \label{eq:S220q}
\sum_{k_1,\dots,k_r=-n} ^n\;
\prod_{1\leq i<j\leq r} [k_j-k_i]_q^2\,[k_i+k_j]_q^2
\prod_{i=1}^r q^{k_i^2-(2i-\frac{3}{2})k_i} \frac {1+q^{k_i}} {1+q}
\qbin{2n}{n+k_i}_q \qbin{2m}{m+k_i}_q\\
=r!\Big(\frac {2} {[2]_q}\Big)^r\,
q^{-2\binom {r+1}3+\frac {1} {2}\binom r2}
\prod_{i=1}^r 
\frac{\Gamma_{q^{1/2}}(2i-1)\,\Gamma_q(2n+1)\,\Gamma_q(2m+1)} 
{\Gamma_q(m+n-i+2)\,\Gamma_q(m+n-i-r+3)}\\
\cdot
\frac {\Gamma_{q^{1/2}}(2m+2n-2i-2r+5)} 
{\Gamma_{q^{1/2}}(2n-2i+3)\,\Gamma_{q^{1/2}}(2m-2i+3)}.
\end{multline}
\end{proposition}

Taking the $q\to 1$ limit, dividing both sides of the result by
${\binom {2m}m}^r$, and finally taking the limit
$m\to\infty$, we arrive at (cf.\ \eqref{eq:2anyanygp4})
\begin{align} \label{eq:S220}
\mathcal{S}_{r,n}(2,1,0)&=\sum_{k_1,\dots,k_r=-n} ^n\;
\prod_{1\leq i<j\leq r} (k_i^2-k_j^2)^2
\prod_{i=1}^r \binom {2n}{n+k_i}\\ \notag
&=2^{2r(n-r+1)} \Gamma(r+1)
\prod_{i=1}^{r-1} \frac{\Gamma(2i+1)\,\Gamma(2n+1)}{\Gamma(2n-2i+1)}.
\end{align}
The evaluation of $W_2(n)$ provided after the proof of Theorem~3.2 in
\cite{BOOPAA} is the special case $r=2$ of this identity.

\begin{proof}
To begin with, we observe that the summand of the sum on the left-hand
side of \eqref{eq:S220q} is invariant under permutations of the
summation indices. Indeed, writing $S_1(k_1,k_2,\dots,k_r)$ for this
summand, for a permutation $\sigma$ of $\{1,2,\dots,r\}$ we have
\begin{equation} \label{eq:Sk} 
S_1(k_{\si(1)},k_{\si(2)},\dots,k_{\si(r)})
=q^{E_1(\sigma;k_1,k_2,\dots,k_r)} S_1(k_1,k_2,\dots,k_r),
\end{equation}
where
\begin{equation} \label{eq:Ek} 
E_1(\sigma;k_1,k_2,\dots,k_r)=
2\sum_{1\leq i<j\leq r}
\chi\big(\si(i)>\si(j)\big)\big(k_{\si(j)}-k_{\si(i)}\big)
-2\sum_{i=1} ^r (ik_{\si(i)}-ik_i).
\end{equation}
Here, as before, 
$\chi(\mathcal A)=1$ if $\mathcal A$ is true and $\chi(\mathcal{A})=0$ 
otherwise. Let $I_\si(i)$ denote the number of indices $j$ with
$1\leq i<j\leq r$ and $\si(i)>\si(j)$. Then, by elementary counting, we have
\begin{align*}
\sum_{1\leq i<j\leq r}
\chi\big(\si(i)>\si(j)\big)(k_{\si(j)}-k_{\si(i)})
&=\sum_{j=1}^r \big(j-\si(j)+I_\si(j)\big)k_{\si(j)}
-\sum_{i=1}^r I_\si(i)k_{\si(i)}\\
&=\sum_{i=1} ^r (i-\si(i))k_{\si(i)}.
\end{align*}
If this is substituted back in \eqref{eq:Ek}, then one obtains
$E_1(\sigma;k_1,k_2,\dots,k_r)=0$.
In combination with \eqref{eq:Sk}, this implies the claimed
invariance of summands under permutations of the summation indices.
As a consequence, we may restrict the range of summation on the
left-hand side of \eqref{eq:S220q} to $k_1<k_2<\dots<k_r$, and
in turn multiply this restricted sum by $r!$, thereby not changing
the value of the left-hand side of \eqref{eq:S220q}.

Now, in this (restricted) sum, we replace $k_i$ by $k_i-n$, and we
rewrite the arising multiple sum in terms of $q$-shifted factorials.
The result is
\begin{multline*}
\frac {r!\,q^{rn^2+(r^2-\frac {r}{2})n} 
(1+q^{-n})^r} {(1-q)^{2r^2-2r}\,(1+q)^r}
\qbin{2m}{m-n}_q^r\:
\sum_{0\leq k_1<\dots<k_r}\;
\prod_{1\leq i<j\leq r}
\big(1-q^{k_i-k_j}\big)^2\, \big(1-q^{k_i+k_j-2n}\big)^2\\
\cdot \prod_{i=1}^r \frac {(q^{-2n},-q^{1-n},q^{-m-n};q)_{k_i}} 
{(q,-q^{-n},q^{m-n+1};q)_{k_i}}\, q^{(2i-2r+m+n+\frac{1}{2})k_i},
\end{multline*}
where the summation indices $k_i$ now run over integers.
Thus, we see that we may apply \eqref{eq:Rains3}
with $a=q^{-2n}$, $d=q^{-m-n}$, $e=q^{-n}$, $f=q^{-n+1/2}$ 
to evaluate this sum.
We have to be a little careful though because of the appearance
of the ratio $(aq;q)_\infty/(aq/e;q)_\infty$ on the right-hand
side of \eqref{eq:Rains3}, which becomes the indeterminate expression
$0/0$ for the above choices of $a$ and $e$. To be precise, in
\eqref{eq:Rains3} we have to first choose $e=\sqrt a$, and
subsequently calculate the limit as $a$ tends to $q^{-2n}$.
Doing this, we obtain
\begin{align} \label{eq:lim}
\lim_{a\to q^{-2n}}\frac {(aq;q)_\infty} {(\sqrt aq;q)_\infty}
&=
\lim_{a\to q^{-2n}}\frac {(aq;q)_{2n-1}\,(1-aq^{2n})\,(aq^{2n+1};q)_\infty} 
{(\sqrt aq;q)_{n-1}\,(1-\sqrt aq^n)\,(\sqrt aq^{n+1};q)_\infty}\\
&=2\,\frac {(q^{1-2n};q)_{2n-1}} {(q^{1-n};q)_{n-1}}
=2\,(q^{1-2n};q)_n.
\notag
\end{align}
After considerable simplification and rewriting of the right-hand side
of \eqref{eq:Rains3} under the above specialisation, we obtain the
right-hand side of \eqref{eq:S220q}.
\end{proof}

\pagebreak
\subsection{The evaluation of $\mathcal{S}_{r,n}(2,1,1)$} 

\begin{proposition} \label{thm:S221q}
Let $q$ be a real number with $0<q<1$.
For all non-negative integers $m$ and $n$ and a positive
integer~$r$, we have
\begin{multline} \label{eq:S221q}
\sum_{k_1,\dots,k_r=-n}^n\;
\prod_{1\leq i<j\leq r}
[k_j-k_i]_q^2\,[k_i+k_j]_q^2
\prod_{i=1}^r q^{k_i^2-(2i-1)k_i} \Abs{[k_i]_{q^2}} 
\qbin{2n}{n+k_i}_q \qbin{2m}{m+k_i}_q\\
=r!\Big(\frac {2} {[2]_q}\Big)^r\,q^{-2\binom {r+1}3}
\prod_{i=1}^r \bigg( \frac{\Gamma_q^2(i)\, \Gamma_q(2n+1)}
{\Gamma_q(n-i+2)\,\Gamma_q(n-i+1)}\\ \times
\frac{\Gamma_q(2m+1)}{\Gamma_q(m-i+2)\,\Gamma_q(m-i+1)}
\cdot
\frac{\Gamma_q(m+n-i-r+2)}{\Gamma_q(m+n-i+2)}\bigg).
\end{multline}
\end{proposition}

Taking the $q\to 1$ limit, dividing both sides of the result by
${\binom{2m}{m}}^r$, and finally performing the limit
$m\to\infty$, we obtain (cf.~\eqref{eq:2anyanygp4})
\begin{align} \label{eq:S221}
\mathcal{S}_{r,n}(2,1,1)&=
\sum_{k_1,\dots,k_r=-n}^n\;
\prod_{1\leq i<j\leq r} (k_i^2-k_j^2)^2
\prod_{i=1}^r \abs{k_i} \binom{2n}{n+k_i}\\ \notag
&=\prod_{i=1}^r 
\frac{\Gamma(i)\,\Gamma(i+1)\,\Gamma(2n+1)}{\Gamma(n-i+2)\,\Gamma(n-i+1)}.
\end{align}

\begin{proof}
Here, the summand of the multiple sum on the left-hand side
of \eqref{eq:S221q} is invariant under both permutations
of the summation indices and under replacement of $k_i$
by $-k_i$, for some fixed~$i$. To show this, if $S_2(k_1,k_2,\dots,k_r)$
denotes the summand, then we have
\[
S_2(k_1,\dots,k_{i-1},-k_i,k_{i+1},\dots,k_r)
=q^{E_2(k_1,k_2,\dots,k_r)} S_2(k_1,k_2,\dots,k_r),
\]
where
\begin{align*}
E_2(k_1,k_2,\dots,k_r)&=2\sum_{1\leq i<j\leq r}(-k_j-k_i)
+2\sum_{1\leq i<j\leq r}(-k_j+k_i)+\sum_{i=1}^r \big(2(2i-1)k_i-2k_i\big)\\
&=4\sum_{j=1}^r\big({-}(j-1)k_j\big)+\sum_{i=1}^r (4i-4)k_i=0.
\end{align*}
This proves the invariance of $S_2(k_1,k_2,\dots,k_r)$ under
the replacement $k_i\to-k_i$.
As a consequence, we may restrict the range of summation on the
left-hand side of \eqref{eq:S221q} to $1\leq k_1<k_2<\dots<k_r$, and
in turn multiply this restricted sum by $2^rr!$, thereby not changing
the value of the left-hand side of \eqref{eq:S221q}.

In this (restricted) sum,
we replace $k_i$ by $k_i+1$, and we
rewrite the arising multiple sum in terms of $q$-shifted factorials.
The result is
\begin{multline*}
\frac {2^rr!} {q^{r^2-r}(1-q)^{2r^2-2r}}
\qbin{2n}{n+1}_q^r
\qbin{2m}{m+1}_q^r
\sum_{0\leq k_1<\dots<k_r}\;
\prod_{1\leq i<j\leq r}
\big(1-q^{k_i-k_j}\big)^2\,
\big(1-q^{k_i+k_j+2}\big)^2\\
\cdot
\prod_{i=1}^r 
\frac {(q^2,-q^2,q^{1-n},q^{1-m};q)_{k_i}} 
{(q,-q,q^{n+2},q^{m+2};q)_{k_i}}\,
q^{(2i-2r+m+n)k_i}.
\end{multline*}
Thus, we see that we may apply \eqref{eq:Rains3}
with $a=q^{2}$, $d=q^{1-n}$, $e=q^{1-m}$, $f=q$ 
to evaluate this sum.
After considerable simplification and rewriting, we obtain the
right-hand side of \eqref{eq:S221q}.
\end{proof}

\begin{proposition} \label{thm:S221uq}
Let $q$ be a real number with $0<q<1$.
For all positive half-integers $m$ and $n$ and a positive
integer~$r$, we have
\begin{multline} \label{eq:S221uq}
\sum_{k_1,\dots,k_r=-n}^n\;
\prod_{1\leq i<j\leq r}
[k_j-k_i]_q^2\,[k_i+k_j]_q^2
\prod_{i=1}^r q^{k_i^2-(2i-1)k_i} \Abs{[k_i]_{q^2}} 
\qbin{2n}{n+k_i}_q
\qbin{2m}{m+k_i}_q\\
=r!\Big(\frac {2} {[2]_q}\Big)^r\,q^{-\frac{1}{4}\binom {2r+1}{3}}
\prod_{i=1}^r \bigg(\frac{\Gamma_q^2(i)\,
\Gamma_q(2n+1)}{\Gamma_q^2(n-i+\frac{3}{2})}\\ \times
\frac{\Gamma_q(2m+1)}{\Gamma_q^2(m-i+\frac{3}{2})}\cdot
\frac{\Gamma_q(m+n-i-r+2)} {\Gamma_q(m+n-i+2)}\bigg).
\end{multline}
\end{proposition}

\begin{proof}
This can be proved in the same way as Proposition~\ref{thm:S221q}.
The only differences are that, here, the summation index $k_i$ is
replaced by $k_i+\frac {1} {2}$, $i=1,2,\dots,r$, 
and that the relevant specialisation
of \eqref{eq:Rains3} is $a=q$, $d=q^{1/2-n}$, $e=q^{1/2-m}$, $f=q$. 
\end{proof}

From now on, all proofs are similar to one of the proofs of
Propositions~\ref{thm:S220q}--\ref{thm:S221uq}, except for the proof
of Proposition~\ref{thm:S223uq}. For the remaining
theorems in this section (except for Proposition~\ref{thm:S223uq}), 
we therefore content
ourselves with specifying which choice of parameters in
\eqref{eq:Rains3} has to be used, without providing further details.

\subsection{The evaluation of $\mathcal{S}_{r,n}(2,1,2)$} 

\begin{proposition}[$\mathrm{B}_r$ summation] 
\label{thm:S222q}
Let $q$ be a real number with $0<q<1$.
For all non-negative integers or half-integers $m$ and $n$ and a positive
integer~$r$, we have
\begin{multline} \label{eq:S222q}
\sum_{k_1,\dots,k_r=-n}^n\; \prod_{1\leq i<j\leq r}
[k_j-k_i]_q^2\,[k_i+k_j]_q^2
\prod_{i=1}^r q^{k_i^2-(2i-\frac {1} {2})k_i} 
\Abs{[k_i]_{q^2}\,[k_i]_q}
\qbin{2n}{n+k_i}_q \qbin{2m}{m+k_i}_q\\
=r!\Big(\frac {2} {[2]_q}\Big)^r\,[2]_{q^{1/2}}^{-r}\,
q^{-2\binom {r+1}3-\frac {1} {2}\binom {r+1}2}
\prod_{i=1}^r \bigg(\frac{\Gamma_{q^{1/2}}(2i)\,
\Gamma_q(2n+1)\,\Gamma_q(2m+1)} 
{\Gamma_q(m+n-i+2)\,\Gamma_{q}(m+n-i-r+2)}\\
\times
\frac{\Gamma_{q^{1/2}}(2m+2n-2i-2r+3)} 
{\Gamma_{q^{1/2}}(2n-2i+2)\,\Gamma_{q^{1/2}}(2m-2i+2)}\bigg).
\end{multline}
\end{proposition}

Taking the
$q\to 1$ limit, dividing both sides of the result by
${\binom {2m}m}^r$, and finally performing the limit
$m\to\infty$, we obtain (cf.~\eqref{eq:2anyanygp4})
\begin{align} \label{eq:S222}
\mathcal{S}_{r,n}(2,1,2)&=\sum_{k_1,\dots,k_r=-n}^n\;
\prod_{1\leq i<j\leq r} (k_i^2-k_j^2)^2
\prod_{i=1}^r k_i^2 \binom{2n}{n+k_i}\\
&=2^{r(2n-2r-1)} 
\prod_{i=1}^r \frac{\Gamma(2i+1)\,\Gamma(2n+1)}{\Gamma(2n-2i+2)}.  \notag
\end{align}

\begin{proof}
The special case of \eqref{eq:Rains3} which is relevant here is
$a=q^{-2n}$, $d=q^{-m-n}$, $e=q^{-n+1}$, and $f=q^{-n+1/2}$.
\end{proof}

\subsection{The evaluation of $\mathcal{S}_{r,n}(2,1,3)$} 

\begin{proposition} \label{thm:S223q}
Let $q$ be a real number with $0<q<1$.
For all non-negative integers $m$ and $n$ and a positive
integer~$r$, we have
\begin{multline} \label{eq:S223q}
\sum_{k_1,\dots,k_r=-n}^n\; \prod_{1\leq i<j\leq r}
[k_j-k_i]_q^2\,[k_i+k_j]_q^2
\prod_{i=1}^r q^{k_i^2-2ik_i} \Abs{[k_i]_{q^2}\,[k_i]_{q}^2} 
\qbin{2n}{n+k_i}_q \qbin{2m}{m+k_i}_q\\
=r!\Big(\frac{2}{[2]_q}\Big)^r\,q^{-2\binom {r+1}3-\binom {r+1}2}
\prod_{i=1}^r \bigg(
\frac {\Gamma_q(2n+1)}{\Gamma_q^2(n-i+1)}\cdot
\frac {\Gamma_q(2m+1)}{\Gamma_q^2(m-i+1)}\\
\times
\frac {\Gamma_q(i)\,\Gamma_q(i+1)\,\Gamma_q(m+n-i-r+1)} 
{\Gamma_q(m+n-i+2)}\bigg).
\end{multline}
\end{proposition}

Taking the
$q\to 1$ limit, dividing both sides of the result by
${\binom {2m}m}^r$, and finally performing the limit
$m\to\infty$, we obtain (cf.~\eqref{eq:2anyanygp4})
\begin{align} \label{eq:S223}
\mathcal{S}_{r,n}(2,1,3)&=
\sum_{k_1,\dots,k_r=-n}^{n}\;
\prod_{1\leq i<j\leq r}
\!\!\!\! 
(k_i^2-k_j^2)^2
\prod_{i=1}^r \abs{k_i}^3 
\binom{2n}{n+k_i}\\
&=
\prod_{i=1}^r 
\frac {\Gamma^2(i+1)\,\Gamma(2n+1)} {\Gamma^2(n-i+1)}.
\notag
\end{align}

\begin{proof}
The special case of \eqref{eq:Rains3} which is relevant here is
$a=q^2$, $d=q^{1-n}$, $e=q^{1-m}$, and $f=q^2$.
\end{proof}

\begin{proposition} \label{thm:S223uq}
Let $q$ be a real number with $0<q<1$.
For all positive half-integers $m$ and $n$ and a positive
integer~$r$, we have
\begin{align}\label{eq:S223uq}
&\sum_{k_1,\dots,k_r=-n}^n\; \prod_{1\leq i<j\leq r}
[k_j-k_i]_q^2\,[k_i+k_j]_q^2
\prod_{i=1}^r q^{k_i^2-2ik_i} \Abs{[k_i]_{q^2}\,[k_i]_q^2} \,
\qbin{2n}{n+k_i}_q \qbin{2m}{m+k_i}_q\\
&\qquad = r!\Big(\frac {2} {[2]_q}\Big)^r\,
q^{-2\binom{r+1}{3}-\frac{r^2}{2}-\frac{r}{4}}
\prod_{i=1}^r \bigg(\frac{\Gamma_q(2n+1)}{\Gamma_q^2(n-i+\frac{3}{2})}\cdot
\frac{\Gamma_q(2m+1)}{\Gamma_q^2(m-i+\frac{3}{2})} \notag \\ 
&\qquad\qquad\qquad\qquad\qquad\qquad\qquad\quad\times
\frac {\Gamma_q(i)\,\Gamma_q(i+1)\,\Gamma_q(m+n-i-r+1)} 
{\Gamma_q(m+n-i+2)}\bigg) \notag \\
&\qquad\quad \times \sum_{s=0}^r \frac {(\sqrt q;q)_s^2} {(1-q)^{2s}}
\frac{[n-s-1/2]_q!\,[m-s-1/2]_q!} {[n-r-1/2_q!\,[m-r-1/2]_q!}\,
\qbin{r}{s}_q\, \qbin{m+n-r}{s}_q. \notag
\end{align}
\end{proposition}

\begin{proof}
We start in the same way as in the proof of Proposition~\ref{thm:S221q},
observing that the summand on the left-hand side of \eqref{eq:S223uq} 
is invariant under permutations of the summation indices $k_i$ and under 
replacement of $k_i$ by $-k_i$, for some fixed~$i$. 
This allows one to concentrate on the range
\[
\frac{1}{2}\leq k_1<k_2<\dots< k_r.
\]
The final result is then obtained by multiplying the sum over
this range by $2^rr!$.

Next we replace $k_i$ by $k_i+\tfrac{1}{2}$ for $i=1,2,\dots,r$,
and rewrite the resulting sum using $q$-shifted factorials, to obtain
\begin{multline*}
\frac{1}{(1-q)^{2r^2-2r}} \qbin{2n}{n+1/2}_q^r \qbin{2m}{m+1/2}_q^r\:
\sum_{0\leq k_1<\dots<k_r}\; \prod_{1\leq i<j\leq r}
(1-q^{k_i-k_j})^2\, (1-q^{k_i+k_j+1})^2\\
\times \prod_{i=1}^r \bigg( q^{(2i-1)k_i-i+1/4+k_i(m+n-2r)} \,
\frac {(1-q^{2k_i+1})\,(1-q^{k_i+1/2})^2\,(q,q^{1/2-n},q^{1/2-m};q)_{k_i}} 
{(1-q^2)\,(1-q)^2\,(q,q^{3/2+n},q^{3/2+m};q)_{k_i}}\bigg).
\end{multline*}
We may now transform the multiple sum on the right using \eqref{eq:Rains4} 
with $a=q$, $c=q^{3/2}$, $d=q$, $e=q^{1/2-n}$ and $f=q^{1/2-m}$.
Using the standard basic hypergeometric notation
\[
\qhyp{r}{s}\bigg[\genfrac{}{}{0pt}{}{a_1,\dots,a_r}{b_1,\dots,b_s};q,z\bigg] 
=\sum_{\ell=0}^{\infty}\frac{(a_1,\dots,a_r;q)_{\ell}}
{(q,b_1,\dots,b_s;q)_{\ell}}\,
\Big((-1)^{\ell} q^{\binom{\ell}{2}}\Big)^{s-r+1}z^{\ell},
\]
where $(a_1,\dots,a_k;q)_{\ell}=(a_1;q)_{\ell}\cdots (a_k;q)_{\ell}$,
we obtain that this sum equals
\[
F(m,n,r)\,
\qhyp{8}{7}\bigg[\genfrac{}{}{0pt}{}
{q^{-r},q^{1-r/2},-q^{1-r/2},q^{1/2-r+n},q^{1/2-r+m},q^{1/2},q^{1/2},q^{-r}}
{q^{-r/2},-q^{-r/2},q^{1/2-n},q^{1/2-m},q^{1/2-r},q^{1/2-r},q};
q,q^{r-m-n} \bigg],
\]
where $F(m,n,r)$ is an explicit product, suppressed here in 
order to focus on the essential part in the expression.
To the above $\qhyp{8}{7}$-series, we may apply Watson's transformation 
formula between a very-well-poised $\qhyp{8}{7}$-series and a balanced 
$\qhyp{4}{3}$-series (see \cite[Appendix (III.17)]{GR04})
\begin{multline*}
\qhyp{8}{7}\bigg[\genfrac{}{}{0pt}{}{a,q\sqrt{a},-q\sqrt{a},b,c,d,e,f}
{\sqrt{a},-\sqrt{a},aq/b,aq/c,aq/d,aq/e,aq/f};q,\frac{a^2q^2}{bcdef}\bigg]\\
=\frac{(aq,aq/de,aq/df,aq/ef;q)_{\infty}}{(aq/d,aq/e,aq/f,aq/def;q)_{\infty}} 
\,\qhyp{4}{3}\bigg[\genfrac{}{}{0pt}{}{aq/bc,d,e,f}{aq/b,aq/c,def/a};q,q\bigg],
\end{multline*}
provided the $\qhyp{8}{7}$-series converges and the $\qhyp{4}{3}$-series
terminates. It is then a routine but tedious task to convert the resulting
expression into the right-hand side of \eqref{eq:S223uq}.
\end{proof}

\section{Discrete Macdonald--Mehta integrals for $\gamma=1$ and $\alpha=1$}
\label{sec:gamma=1-1}

The purpose of this section is to present our evaluations
of $S_{r,n}(\alpha,\gamma,\delta)$ for $\gamma=\alpha=1$. 
In principle, it would seem that such evaluations could also
follow from the transformation formula in Theorem~\ref{thm:Rains},
by considering a limit case where $a\to0$. Indeed,
the case $\delta=0$, that is, the evaluation of the sum
$S_{r,n}(1,1,0)$, is covered
by \eqref{eq:Rains}, and it also produces a $q$-analogue containing
a further parameter. Alas,
all our attempts to come up with appropriate further specialisations
that would produce the sum $S_{r,n}(1,1,\delta)$ with $\delta=1$
on the left-hand side of \eqref{eq:Rains} failed. Hence, in order to
achieve the corresponding summation, we designed an ad hoc 
approach combining the evaluation of certain Vandermonde-
and Cauchy-like determinants with summation formulas from the
theory of hypergeometric series. As opposed to the case $\delta=0$,
for $\delta=1$ we were not able to find a $q$-analogue.

It is interesting to note that the limit case
$a\to0$ of \eqref{eq:Rains} has been worked out earlier in
\cite[Equation~(3.7)]{KrScAB}, where it was used for the
enumeration of standard Young tableaux of certain skew shapes. 
As is pointed out there, that limit case had explicitly
appeared even earlier in \cite{KratBM}, where two different proofs
had been given
(one using a specialisation of an identity for Schur functions,
the other using a specialisation of a $q$-integral evaluation due
to Evans), and where it had been applied in an again different
context, namely that of the enumeration of domino tilings.

\subsection{The evaluation of $\mathcal{S}_{r,n}(1,1,0)$} 

\begin{proposition}[$\mathrm{A}_{r-1}$ summation] 
\label{thm:S120q}
Let $q$ be a real number with $0<q<1$.
For all non-negative integers or half-integers $m$ and $n$ and a positive
integer~$r$, we have
\begin{multline}\label{eq:S120q}
\sum_{k_1,\dots,k_r=-n}^n\;
\prod_{1\leq i<j\leq r} [k_j-k_i]_q^2
\prod_{i=1}^r q^{k_i^2+(m+n-2i+2)k_i} 
\qbin{2n}{n+k_i}_q \qbin{2m}{m+k_i}_q\\
=r!\, q^{-rmn-\frac {1}{6}r(r-1)(2r-3m-3n-1)} \kern5cm\\
\times \prod_{i=1}^r \frac{\Gamma_q(i)\,\Gamma_q(2n+1)\,\Gamma_q(2m+1)\,
\Gamma_q(2m+2n-r-i+3)}
{\Gamma_q(2n-i+2)\,\Gamma_q(2m-i+2)\,\Gamma_q^2(m+n-i+2)}.
\end{multline}
\end{proposition}

Taking the
$q\to 1$ limit, dividing both sides of the result by
${\binom {2m}m}^r$, and finally performing the limit
$m\to\infty$, we obtain (cf.~\eqref{eq:1any0})
\begin{align} \label{eq:S120}
\mathcal{S}_{r,n}(1,1,0)&=\sum_{k_1,\dots,k_r=-n}^n\;
\prod_{1\leq i<j\leq r} (k_i-k_j)^2 \prod_{i=1}^r \binom{2n}{n+k_i}\\
&=2^{2rn-r(r-1)} \prod_{i=1}^r \frac{\Gamma(i+1)\,\Gamma(2n+1)}
{\Gamma(2n-i+2)}. \notag
\end{align}

\begin{proof}
The special case of \eqref{eq:Rains3} which is relevant here is
$d=aq^{n-m}$, $e=q^{-m-n}$, $f=q^{-2n}$, and finally $a\to0$.
\end{proof}

\subsection{The evaluation of $\mathcal{S}_{r,n}(1,1,1)$} 

\begin{proposition}\label{Prop_S121}
For all non-negative integers $m$ and $n$ and a positive
integer~$r$, we have
\begin{multline} \label{eq:S121}
\sum_{k_1,\dots,k_r=-n}^n\; \prod_{1\leq i<j\leq r} (k_i-k_j)^2
\prod_{i=1}^r \, \abs{k_i} \binom{2n}{n+k_i}\binom{2m}{m+k_i}\\
=r! \prod_{i=1}^{\ceiling{r/2}}
\frac{\Gamma^2(i)\,\Gamma(2n+1)\,\Gamma(2m+1)\,\Gamma(m+n-i-\ceiling{r/2}+2)} 
{\Gamma(n-i+2)\,\Gamma(n-i+1)\,\Gamma(m-i+2)\,\Gamma(m-i+1)\,\Gamma(m+n-i+2)}
\\
\times \prod_{i=1}^{\floor{r/2}} 
\frac{\Gamma(i)\,\Gamma(i+1)\,\Gamma(2n+1)\,\Gamma(2m+1)\,
\Gamma(m+n-i-\floor{r/2}+1)} 
{\Gamma^2(n-i+1)\,\Gamma^2(m-i+1)\,\Gamma(m+n-i+2)}.
\end{multline}
\end{proposition}

Dividing both sides of \eqref{eq:S121} by $\binom {2m}{m}^r$ and performing 
the limit $m\to\infty$, we obtain (cf.~\eqref{eq:121})
\begin{align} \label{eq:S121rn}
\mathcal{S}_{r,n}(1,1,1)&=\sum_{k_1,\dots,k_r=-n}^n\; 
\prod_{1\leq i<j\leq r} (k_i-k_j)^2 \prod_{i=1}^r \, \abs{k_i}
\binom{2n}{n+k_i}\\
&= r! \prod_{i=1}^{\ceiling{r/2}}
\frac{\Gamma^2(i)\,\Gamma(2n+1)}{\Gamma(n-i+2)\,\Gamma(n-i+1)}
\prod_{i=1}^{\floor{r/2}}
\frac{\Gamma(i)\,\Gamma(i+1)\,\Gamma(2n+1)}{\Gamma^2(n-i+1)}. \notag
\end{align}

\begin{proof}
We start by writing the Vandermonde products
(there are \emph{two} since the Vandermonde product is squared) 
in the summand as the following determinants:
\[
\det_{1\leq i,j\leq r} \Big(1\quad k_i\quad (n^2-k_i^2)\quad 
k_i(n^2-k_i^2)\quad (n^2-k_i^2) \big((n-1)^2-k_i^2\big)\quad \dots \Big),
\]
respectively
\[
\det_{1\leq i,j\leq r} \Big(1\quad k_i\quad (m^2-k_i^2)\quad 
k_i(m^2-k_i^2)\quad (m^2-k_i^2) \big((m-1)^2-k_i^2\big)\quad \dots \Big),
\]
This has to be read in such a way that the individual entries above give
the columns of the matrix. More precisely, we have
\[
\prod_{1\leq i<j\leq r} (k_i-k_j)=\pm\det M(N),
\]
where $M(N)=\big(M_{i,j}(N)\big)_{1\leq i,j\leq r}$ is the $r\times r$
matrix defined by
\[
M_{i,j}(N)=(-1)^{2\floor{(j-1)/2}}k_i^{\chi(j\text{ even})}\,
(-N-k_i)_{\floor{(j-1)/2}}\,
(-N+k_i)_{\floor{(j-1)/2}},
\]
Here, as before, $\chi(\mathcal{A})=1$ if $\mathcal{A}$ is true and 
$\chi(\mathcal{A})=0$ otherwise, and
the Pochhammer symbol $(\alpha)_m$ is defined
by $(\alpha)_m:=\alpha(\alpha+1)\cdots(\alpha+m-1)$ for 
$m\geq 1$, and $(\alpha)_0:=1$. The substitution in \eqref{eq:S121} that
we apply is
\[
\prod_{1\leq i<j\leq r} (k_i-k_j)^2 = \det M(n)\cdot\det M(m).
\]
This turns the left-hand side of \eqref{eq:S121} into
\begin{align}\label{eq:Summe}
\sum_{\si,\tau\in S_r}&\sgn\si\tau 
\prod_{i=1}^r \Bigg(
\sum_{k_i=-\infty}^{\infty}\bigg(
\abs{k_i}\,k_i^{\chi(\si(i)\text{ even})+\chi(\tau(i)\text{ even})}\\
& \qquad\qquad\qquad \times
\frac {(2n)!} {(n+k_i-\floor{(\si(i)-1)/2})!\,
(n-k_i-\floor{(\si(i)-1)/2})!} \notag \\
& \qquad\qquad\qquad \times
\frac {(2m)!} {(m+k_i-\floor{(\tau(i)-1)/2})!\,
(m-k_i-\floor{(\tau(i)-1)/2})!}\bigg)\Bigg). \notag
\end{align}
We must now evaluate the sum over $k_i$.
There are three cases to be considered, depending on whether
$\si(i)$ and $\tau(i)$ are even or odd. 
For convenience, in the following we
shall use the short notation $S=\floor{(\si(i)-1)/2}$
and $T=\floor{(\tau(i)-1)/2}$.

\smallskip
\textsc{Case 1}: \emph{$\si(i)$ and $\tau(i)$ are both odd.}
In this case we need to evaluate (writing $k$ instead of $k_i$)
\begin{multline}
\sum_{k=-\infty}^{\infty}
\abs{k}\, \frac{(2n)!}{(n+k-S)!\, (n-k-S)!}\cdot
\frac{(2m)!}{(m+k-T)!\, (m-k-T)!}\\
=2\sum_{k=1}^{\infty} k\,
\frac{(2n)!}{(n+k-S)!\,(n-k-S)!}\cdot
\frac{(2m)!}{(m+k-T)!\, (m-k-T)!}.
\label{eq:Sum1}
\end{multline}
We write this sum in terms of the standard hypergeometric notation
\[
\hyp{r}{s}\bigg[\genfrac{}{}{0pt}{}{a_1,\dots,a_r}{b_1,\dots,b_s};z\bigg] 
=\sum_{\ell=0}^{\infty}
\frac{(a_1)_{\ell}\cdots(a_r)_{\ell}} 
{\ell!\,(b_1)_{\ell}\cdots(b_s)_{\ell}}\,z^{\ell},
\]
to obtain the expression
\begin{multline*}
2\,\frac{(2n)!}{(n-S+1)!\, (n-S-1)!}\cdot
\frac{(2m)!}{(m-T+1)!\, (m-T-1)!}\\
\times
\hyp{3}{2}\bigg[\genfrac{}{}{0pt}{}{2,-n+S+1,-m+T+1}{n-S+2,m-T+2};1\bigg].
\end{multline*}
This hypergeometric series can be evaluated by (the terminating
version) of Dixon's summation (see \cite[Appendix~(III.9)]{SlatAC})
\[
\hyp{3}{2}\bigg[\genfrac{}{}{0pt}{}
{a,b,-N}{1+a-b,1+a+N};1\bigg]=
\frac{(1+a)_N\,(1+\frac{a}{2}-b)_N}{(1+\frac{a}{2})_N\,(1+a-b)_N},
\]
where $N$ is a non-negative integer.
Indeed, if we choose $a=2$, $b=-n+S+1$, and $N=m-T-1$ in this summation
formula, then our expression becomes
\begin{equation}\label{eq:Exp1} 
\frac{1}{(m+n-S-T)}\cdot\frac{(2n)!}{(n-S)!\,(n-S-1)!}\cdot
\frac{(2m)!}{(m-T)!\,(m-T-1)!}
\end{equation}
after some simplification.

\smallskip
\textsc{Case 2}: \emph{$\si(i)$ and $\tau(i)$ have different parity.}
In this case, we need to evaluate the sum
\[
\sum_{k=-\infty} ^\infty \abs{k}\,k\,
\frac{(2n)!}{(n+k-S)!\, (n-k-S)!}\cdot
\frac{(2m)!} {(m+k-T)!\, (m-k-T)!}.
\]
Since replacement of the summation index $k$ by $-k$ converts
this expression into its negative, the sum above vanishes.

\smallskip
\textsc{Case 3}: \emph{$\si(i)$ and $\tau(i)$ are both even.}
Now we must evaluate the sum
\begin{multline*}
\sum_{k=-\infty}^{\infty} |k|\,k^2\,
\frac{(2n)!}{(n+k-S)!\, (n-k-S)!}\cdot
\frac{(2m)!}{(m+k-T)!\, (m-k-T)!}\\
=2\sum_{k=1}^{\infty} k^3\,
\frac{(2n)!}{(n+k-S)!\, (n-k-S)!}\cdot
\frac{(2m)!}{(m+k-T)!\, (m-k-T)!}.
\end{multline*}
We write
\begin{equation}\label{eq:Gl} 
k^2=-\big((n-S)^2-k^2\big)+(n-S)^2=-(n+k-S)(n-k-S)+(n-S)^2
\end{equation}
and substitute this in the summand. Splitting the sum accordingly,
we obtain the expression
\begin{multline*}
-2\sum_{k=1}^{\infty} k\,
\frac{(2n)!}{(n+k-S-1)!\, (n-k-S-1)!}\cdot
\frac {(2m)!} {(m+k-T)!\, (m-k-T)!}\\
+2(n-S)^2 \sum_{k=1}^{\infty} k\,
\frac{(2n)!}{(n+k-S)!\, (n-k-S)!}\cdot
\frac{(2m)!}{(m+k-T)!\, (m-k-T)!}.
\end{multline*}
We have evaluated both sums already earlier. To be more specific,
the second sum is the sum on the right-hand side of \eqref{eq:Sum1},
and the first sum arises by replacing $S$ by $S+1$ there. 
The closed-form expression for \eqref{eq:Sum1} is presented in
\eqref{eq:Exp1}. Consequently, our expression above becomes
\begin{multline}\label{eq:Exp3}
-\frac{1}{(m+n-S-T-1)}\cdot\frac{(2n)!}{(n-S-1)!\,(n-S-2)!}\cdot
\frac{(2m)!} {(m-T)!\,(m-T-1)!}\\
+\frac{(n-S)^2}{(m+n-S-T)}\cdot\frac{(2n)!}{(n-S)!\,(n-S-1)!}\cdot
\frac{(2m)!}{(m-T)!\,(m-T-1)!}\\
=\frac{1}{(m+n-S-T-1)(m+n-S-T)}\cdot\frac{(2n)!}{(n-S-1)!^2}\cdot
\frac{(2m)!}{(m-T-1)!^2}.
\end{multline}

If we summarise our findings so far
(combine \eqref{eq:Summe}, \eqref{eq:Exp1} and \eqref{eq:Exp3}), 
then we have seen that the left-hand side of \eqref{eq:S121} equals
\begin{equation}\label{eq:sumA} 
\sum_{\si,\tau\in S_r}\sgn\si\tau 
\prod_{i=1}^r A_{\si(i),\tau(i)},
\end{equation}
where, using the shorthand notation $K=\floor{(k-1)/2}$ and
$L=\floor{(l-1)/2}$,
\[
A_{k,l}=\begin{cases} 
\frac{1}{(m+n-K-L)}\cdot\frac{(2n)!}{(n-K)!\,(n-K-1)!}\cdot
\frac{(2m)!}{(m-L)!\,(m-L-1)!},
&\text{if $k$ and $l$ are odd,}\\[2mm] 
\frac{1}{(m+n-K-L-1)(m+n-K-L)}\cdot\frac{(2n)!}{(n-K-1)!^2}\cdot
\frac{(2m)!}{(m-L-1)!^2},
&\text{if $k$ and $l$ are even,}\\[2mm]
0,&\text{otherwise.}
\end{cases}
\]
We may reorder the product in \eqref{eq:sumA},
$$
\sum_{\si,\tau\in S_r}\sgn\tau\si^{-1} 
\prod_{i=1}^r 
A_{i,\tau\si^{-1}(i)}.
$$
Writing $\rho=\tau\si^{-1}$, we may as well sum over all
$\si$ and $\rho$. Thereby we obtain
$$
\sum_{\si,\rho\in S_r}\sgn\rho
\prod_{i=1}^r 
A_{i,\rho^{-1}(i)}
=
r!\det_{1\leq i,j\leq r}(A_{i,j}).
$$
Thus, the remaining task is to evaluate the determinant of the
$A_{i,j}$'s.

If we recall the definition of $A_{i,j}$, then we see that
the matrix $(A_{i,j})_{1\leq i,j\leq r}$ has a checkerboard structure.
By reordering rows and columns, the matrix can be brought into a block
form, from which it follows that
\begin{equation} \label{eq:A.A} 
\det_{1\leq i,j\leq r}(A_{i,j})
=
\det_{1\leq i,j\leq \ceiling{r/2}}(A_{2i-1,2j-1})\cdot
\det_{1\leq i,j\leq \floor{r/2}}(A_{2i,2j}).
\end{equation}
Explicitly, the first determinant on the right-hand side of
\eqref{eq:A.A} is
\begin{multline*}
\det_{1\leq i,j\leq \ceiling{r/2}}\bigg(
\frac{1}{(m+n-i-j+2)}\cdot\frac{(2n)!}{(n-i+1)!\,(n-i)!}\cdot
\frac{(2m)!}{(m-j+1)!\,(m-j)!}\bigg) \\
=\prod_{i=1}^{\ceiling{r/2}}\bigg(\frac{(2n)!}{(n-i+1)!\,(n-i)!}\cdot
\frac{(2m)!}{(m-i+1)!\,(m-i)!}\bigg)
\det_{1\leq i,j\leq \ceiling{r/2}}\bigg(\frac{1}{m+n-i-j+2}\bigg).
\end{multline*}
Clearly, the last determinant is a special case of Cauchy's
double alternant (take $X_i=n-i+1$ and $Y_j=m-j+2$ in Eq.~(2.7)
of \cite{Krattenthaler99}). Substitution of the result leads to 
\begin{multline} \label{eq:detA1} 
\det_{1\leq i,j\leq \ceiling{r/2}}(A_{2i-1,2j-1}) \\
=\prod_{i=1}^{\ceiling{r/2}}\frac{(2n)!}{(n-i+1)!\,(n-i)!}\cdot
\frac{(2m)!}{(m-i+1)!\,(m-i)!}\cdot
\frac{(i-1)!^2\,(m+n-i-\ceiling{r/2}+1)!} {(m+n-i+1)!}
\end{multline}
after some manipulation.

On the other hand, the second determinant on the right-hand side of
\eqref{eq:A.A} is
\begin{align*}
&\det_{1\leq i,j\leq \floor{r/2}}\bigg(
\frac{1}{(m+n-i-j+1)(m+n-i-j+2)}\cdot\frac{(2n)!}{(n-i)!^2}\cdot
\frac{(2m)!}{(m-j)!^2}\bigg) \\
&\quad =
\prod_{i=1}^{\floor{r/2}}\bigg(\frac{(2n)!} {(n-i)!^2}\cdot
\frac{(2m)!}{(m-i)!^2} \bigg)
\det_{1\leq i,j\leq \floor{r/2}}\bigg(
\frac{1}{(m+n-i-j+1)(m+n-i-j+2)}\bigg)\\
&\quad = 
\prod_{i=1}^{\floor{r/2}}\bigg( \frac{(2n)!}{(n-i)!^2}\cdot
\frac{(2m)!}{(m-i)!^2}\cdot\frac{(m+n-i-\floor{r/2})!}{(m+n-i+1)!}\bigg)\\
&\kern2cm
\times
\det_{1\leq i,j\leq \floor{r/2}}\Big(
(m+n-i-j+3)_{j-1}\,(m+n-i-\floor{r/2}+1)_{\floor{r/2}-j} \Big).
\end{align*}
In order to evaluate this determinant, we have to put
$n=\floor{r/2}$, $X_i=m+n-i$, $A_j=-j+1$, and $B_j=-j+3$
in \cite[Lemma~3]{Krattenthaler99}.
Substitution of the result gives
\begin{equation}\label{eq:detA2} 
\det_{1\leq i,j\leq \floor{r/2}}(A_{2i,2j})
=\prod_{i=1}^{\floor{r/2}}\frac{(2n)!}{(n-i)!^2}\cdot
\frac{(2m)!} {(m-i)!^2}\cdot
\frac{(m+n-i-\floor{r/2})!\,(i-1)!\,i!}{(m+n-i+1)!}
\end{equation}
after some manipulation.

If we finally combine \eqref{eq:A.A}, \eqref{eq:detA1}, and
\eqref{eq:detA2}, then we obtain the right-hand side of \eqref{eq:S121}.
\end{proof}

\begin{proposition}
For all positive half-integers $m$ and $n$ and a positive integer~$r$, we have
\begin{align}\label{eq:S121u}
&\sum_{k_1,\dots,k_r=-n}^n\; \prod_{1\leq i<j\leq r} (k_i-k_j)^2
\prod_{i=1}^r \, \abs{k_i} \, \binom{2n}{n+k_i}\binom{2m}{m+k_i}\\
&= r! \prod _{i=1}^{\ceiling{r/2}}
\frac{\Gamma(2n+1)}{\Gamma^2(n-i+\frac{3}{2})}\cdot
\frac{\Gamma(2m+1)}{\Gamma^2(m-i+\frac{3}{2})}\cdot
\frac{\Gamma^2(i)\,\Gamma(m+n-i-\ceiling{\frac{r}{2}}+2)}{\Gamma(m+n-i+2)}
\notag \\
&\quad\times \prod_{i=1}^{\floor{r/2}}
\frac{\Gamma(2n+1)}{\Gamma^2(n-i+\frac{3}{2})}\cdot
\frac{\Gamma(2m+1)}{\Gamma^2(m-i+\frac{3}{2})}\cdot
\frac{\Gamma(i)\,\Gamma(i+1)\,\Gamma(m+n-i-\floor{\frac{r}{2}}+1)} 
{\Gamma(m+n-i+2)} \notag \\
& \quad \times \sum_{s=0}^{\floor{r/2}}
(-1)^{\floor{r/2}-s}2^{-4(\floor{r/2}-s)}
\frac{m!\,n!\,(\floor{\frac {r} {2}})!\,(m+n-s)!} 
{s!\,(m-s)!\,(n-s)!\, (m+n-\floor{\frac{r}{2}})!}\,
{\binom {2\floor{\frac{r}{2}}-2s}{\floor{\frac{r}{2}}-s}}^2. \notag 
\end{align}
\end{proposition}

\begin{proof}
The proof follows along the lines of the previous proof. In fact,
not much needs to be changed. Until we reach Case~1, everything
is identical. The sum to be evaluated in Case~1 is now
\[
2\sum_{k=1/2}^\infty k\,
\frac{(2n)!}{(n+k-S)!\, (n-k-S)!}\cdot
\frac{(2m)!}{(m+k-T)!\, (m-k-T)!},
\]
with the understanding that $k$ ranges over half-integers.
In hypergeometric terms, this sum equals
\begin{multline*}
\frac{(2n)!}{(n-S+\frac{1}{2})!\, (n-S-\frac{1}{2})!}\cdot
\frac {(2m)!}{(m-T+\frac{1}{2})!\, (m-T-\frac{1}{2})!} \\
\times
\hyp{4}{3}\bigg[\genfrac{}{}{0pt}{}{1,\frac{3}{2},-n+S+\frac{1}{2},
-m+T+\frac{1}{2}}{\frac{1}{2},n-S+\frac{3}{2},m-T+\frac{3}{2}};1\bigg].
\end{multline*}
This hypergeometric series can be summed by means of
the summation formula (see \cite[Appendix~(III.22)]{SlatAC})
\[
\hyp{4}{3}\bigg[\genfrac{}{}{0pt}{}{a,\frac {a} {2}+1,b,c}
{\frac{a}{2},1+a-b,1+a-c};1\bigg]
=\frac{\Gamma(1+a-b)\,\Gamma(1+a-c)\,\Gamma(\frac{1}{2}+\frac{a}{2})\,
\Gamma(\frac{1}{2}+\frac{a}{2}-b-c)}
{\Gamma(1+a)\,\Gamma(1+a-b-c)\,\Gamma(\frac{1}{2}+\frac{a}{2}-b)\,
\Gamma(\frac{1}{2}+\frac{a}{2}-c)}
\]
with $a=1$, $b=-n+S+\frac {1} {2}$, and $c=-m+T+\frac {1} {2}$. The
result is that our sum simplifies to
\[
\frac{1}{(m+n-S-T)}\cdot\frac{(2n)!}{(n-S-\frac {1} {2})!^2}\cdot
\frac{(2m)!}{(m-T-\frac{1}{2})!^2}.
\]

In Case~2 we obtain zero, as before. Finally, the result in Case~3 is
\begin{multline*}
-\frac{1}{(m+n-S-T-1)}\cdot\frac{(2n)!}{(n-S-\frac {3} {2})!^2}\cdot
\frac {(2m)!}{(m-T-\frac {1} {2})!^2}\\
+\frac{(n-S)^2}{(m+n-S-T)}\cdot\frac{(2n)!}{(n-S-\frac {1} {2})!^2}\cdot
\frac{(2m)!}{(m-T-\frac {1} {2})!^2}\\
=\frac {(2n)!}{(n-S-\frac {1} {2})!^2}\cdot
\frac {(2m)!}{(m-T-\frac {1} {2})!^2}\,
\bigg(\frac{(n-S)^2}{m+n-S-T}-\frac{(n-S-\frac{1}{2})^2}{m+n-S-T-1}\bigg).
\end{multline*}

Consequently, as in the previous proof, the left-hand side of
\eqref{eq:S121u} can be written as a product of two determinants
multiplied by $r!$. More precisely, it is equal to
\[
r!\cdot
\det_{1\leq i,j\leq \ceiling{r/2}}\big(B^{(1)}_{i,j}\big)\cdot
\det_{1\leq i,j\leq \floor{r/2}}\big(B^{(2)}_{i,j}\big),
\]
where
\[
B^{(1)}_{i,j}=
\frac{1}{(m+n-i-j+2)}\cdot\frac{(2n)!}{(n-i+\frac{1}{2})!^2}\cdot
\frac{(2m)!}{(m-j+\frac{1}{2})!^2},
\]
and
\[
B^{(2)}_{i,j}=\frac {(2n)!}{(n-i+\frac {1} {2})!^2}\cdot
\frac{(2m)!}{(m-j+\frac {1} {2})!^2}\,
\bigg(\frac {(n-i+1)^2}{m+n-i-j+2}-
\frac{(n-i+\frac{1}{2})^2}{m+n-i-j+1}\bigg).
\]
The first determinant
is again evaluated by applying Cauchy's double alternant,
\[
\det_{1\leq i,j\leq \ceiling{r/2}}\big(B^{(1)}_{i,j}\big)
=\prod_{i=1}^{\ceiling{r/2}}\frac{(2n)!}{(n-i+\frac{1}{2})!^2}\cdot
\frac{(2m)!}{(m-i+\frac{1}{2})!^2}\cdot
\frac{(i-1)!^2\,(m+n-i-\ceiling{r/2}+1)!}{(m+n-i+1)!}.
\]
In order to evaluate the second determinant, we observe that, after
having factored out the terms which depend only on the row index~$i$
or only on the column index~$j$, each entry is a sum of two terms.
We use linearity of the
determinant in the rows to decompose it into a sum of simpler
determinants. In principle, 
this leads to $2^{\floor{r/2}}$ terms. However, one readily sees that
in most of these two successive rows are linearly dependent, and hence these
terms vanish. More precisely, we have
\begin{multline*}
\det_{1\leq i,j\leq \ceiling{r/2}}\big(B^{(2)}_{i,j}\big) =
\prod_{i=1}^{\floor{r/2}}\frac{(2n)!}{(n-i+\frac{1}{2})!^2}\cdot
\frac{(2m)!}{(m-i+\frac{1}{2})!^2}\\
\times
\sum_{s=0}^{\floor{r/2}}(-1)^{\floor{r/2}-s}
\det_{1\leq i,j\leq \floor{r/2}}
\bigg(\frac{(n-i+1-\chi(i>s)\frac{1}{2})^2} {m+n-i-\chi(i>s)-j+2}\bigg).
\end{multline*}
The last determinant can be evaluated by appealing to
Cauchy's double alternant another time, and the result is
\begin{multline*}
\det_{1\leq i,j\leq \floor{r/2}}\big(B^{(2)}_{i,j}\big)
=\prod _{i=1}^{\floor{r/2}}\frac{(2n)!}{(n-i+\frac{1}{2})!^2}\cdot
\frac{(2m)!}{(m-i+\frac{1}{2})!^2}\cdot
\frac {(i-1)!\,i!\,(m+n-i-\floor{\frac{r}{2}})!}{(m+n-i+1)!}
\\
\times
\sum_{s=0}^{\floor{r/2}} (-1)^{\floor{r/2}-s}
\frac{(n-s+1)_s^2\,(n-\floor{\frac{r}{2}}+\frac{1}{2})_{\floor{r/2}-s}^2\,
(m+n-s)!} 
{s!\,(\floor{\frac{r}{2}}-s)!\,(m+n-s-\floor{\frac {r} {2}})!}.
\end{multline*}
In order to make the symmetry in $m$ and $n$ of the final result
immediately obvious, we convert the last sum into a different
form. This is done by first writing it in hypergeometric notation,
\begin{multline*}
\sum_{s=0}^R (-1)^{R-s}
\frac{(n-s+1)_s^2\,(n-R+\frac {1} {2})_{R-s}^2\, (m+n-s)!} 
{s!\,(R-s)!\,(m+n-s-R)!}\\
=(-1)^R \frac{(n-R+\frac{1}{2})_R^2\,(m+n)!}{R!\,(m+n-R)!}
\hyp{4}{3}\bigg[\genfrac{}{}{0pt}{}{-n, -n, -m-n+R,-R}
{-n+\frac {1} {2},-n+\frac {1} {2},-m-n};1\bigg]
\end{multline*}
(here, $R$ is short for $\floor{r/2}$), apply one of Whipple's
balanced $_4F_3$-transformation formulas 
(see \cite[Equation~(4.3.5.1)]{SlatAC}),
\begin{multline*}
\hyp{4}{3}\bigg[\genfrac{}{}{0pt}{}{a,b,c,-N}{e,f,1+a+b+c-e-f-N};1\bigg] =
\frac{(-a+e)_N\,(-a+f)_N}{(e)_N\,(f)_N} \\
\times \hyp{4}{3}\bigg[\genfrac{}{}{0pt}{}{-N,a,1+a+c-e-f-N,1+a+b-e-f-N}
{1+a+b+c-e-f-N,1+a-e-N,1+a-f-N};1\bigg]   
\end{multline*}
where $N$ is a non-negative integer, to obtain
\begin{multline*}
\sum_{s=0}^R (-1)^{R-s}
\frac{(n-s+1)_s^2\,(n-R+\frac {1} {2})_{R-s}^2\,
(m+n-s)!}{s!\,(R-s)!\,(m+n-s-R)!}\\
=\sum_{s=0}^R (-1)^{R-s}2^{-4(R-s)}
\frac {\Gamma(m+1)\,\Gamma(n+1)\,R!\,(m+n-s)!} 
{s!\,\Gamma(m-s+1)!\,\Gamma(n-s+1)!\,(m+n-R)!}\,
{\binom {2R-2s} {R-s}}^2.
\end{multline*}
Combining everything, we arrive at the right-hand side of
\eqref{eq:S121u}.
\end{proof}

\section{Discussion}\label{sec:disc}
We conclude our paper with a discussion of some open problems, 
additional results and future work.

\subsection{Arbitrary values of $\gamma$}

We have only proved discrete analogues of the Mac\-don\-ald--Mehta integral 
\eqref{Eq_Meh-Mac} for $\gamma=1/2$ and $1$, values which in type 
$\mathrm{A}$ correspond to the \emph{Gau\ss ian orthogonal} and 
\emph{Gau\ss ian unitary} random matrix ensembles GOE and GUE, see
e.g., \cite{FW08}. For more general
integer or half-integer values of $\gamma$, the sum \eqref{Eq_Srdef}
is not expressible in terms of a simple ratio of gamma functions.
One of the reasons for this is that we have insisted on the 
simplest-possible discrete analogue of the $G$-Vandermonde product
$\abs{\Delta(x^{\alpha})}^{2\gamma}$ as $\abs{\Delta(k^\alpha)}^{2\gamma}$.
To obtain formulas for more general choices of $\gamma$, more complicated 
analogues are required. For example, the $\mathrm{A}_{r-1}$ identity 
\eqref{eq:S120}, pertaining to $\gamma=1$, may be generalised to
all non-negative integer values of $\gamma$ as\footnote{To prove
this, we can take \cite[Theorem 4.1]{W05} with $a=q^{-n-m}$, $b\to\infty$,
$c=q^{m-n+1}$ and $k=\gamma$, where it is assumed without loss of generality
that $m\geq n$. Symmetrising the summand using Lemma 3.1 of that same paper
we obtain a generalisation of \eqref{eq:S120q} in which 
$[k_j-k_i]_q^2$ is replaced by $(q^{k_j-k_i},q^{1-\gamma+k_j-k_i};q)_{\gamma}$ 
and $q^{k_i^2+(m+n-2i+2)k_i}$ by $q^{k_i^2+(m+n-2(i-1)\gamma)k_i}$.
The rest follows as in the proof of \eqref{eq:S120}.}
\begin{multline*}
\sum_{k_1,\dots,k_r=-n}^n\;
\prod_{1\leq i<j\leq r} \Abs{(k_i-k_j)_{\gamma}\,(k_j-k_i)_{\gamma}} 
\prod_{i=1}^r \binom{2n}{n+k_i}\\
=2^{2rn-\gamma r(r-1)} \prod_{i=1}^r 
\frac{\Gamma(1+i\gamma)}{\Gamma(1+\gamma)}\cdot
\frac{\Gamma(2n+1)}{\Gamma(2n-(i-1)\gamma+1)}.
\end{multline*}
For $\gamma=2$ this choice of Vandermonde-type product is in agreement with
the \emph{discrete symplectic ensemble} considered by Borodin and Strahov 
\cite{BS09}. We did not, however, succeed in finding analogous generalisations
for the other summations presented in this paper.

\subsection{More general reflection groups}\label{sec:exc}

Another notable omission has been the treatment of reflection groups
other than $\mathrm{A}_{r-1}$, $\mathrm{B}_r$ and 
$\mathrm{D}_r$. So far we have not found nice closed-form discrete analogues 
of Macdonald's integral \eqref{Eq_Macdonald} for any of the exceptional 
reflection groups or for the remaining infinite series,
made up of the dihedral groups $I_2(m)$, $m\geq 3$ 
(the automorphism groups of the regular $m$-gons).
It is difficult to conclude with certainty that no nice 
discrete analogues actually exist for any of these missing cases.
In writing down the polynomials $P_G(x)$ for
$\mathrm{A}_{r-1}$, $\mathrm{B}_r$ and $\mathrm{D}_r$ in \eqref{Eq_Poly-ABD},
we implicitly used the fact that Macdonald's integral \emph{does not} depend
on the choice of $P_G(x)$. The actual form of $P_G(x)$ \emph{does}
depend on the choice of normals $a_i$ in \eqref{Eq_Ppoly}, and hence on the
choice of reflecting hyperplanes $H_1,\dots,H_m$ generating $G$.
For a given $G$, the set of hyperplanes, and hence the 
set of normals, is fixed up to a global rotation $R$ of $\R^r$.
If $a_i'=R(a_i)$ for $i=1,\dots,m$, then
\[
\Int_{\R^r} \BigAbs{\prod_{i=1}^m (a'_i\cdot y_i)}^{2\gamma} \dup \varphi(y)
\stackrel{y=R(x)}{=}
\Int_{\R^r} \BigAbs{\prod_{i=1}^m (a_i\cdot x_i)}^{2\gamma} \dup \varphi(x),
\]
since the measure $\varphi(x)$ is rotationally invariant.
At the discrete level, however, rotational invariance is lost, and hence the
choice of $P_G(x)$ crucially affects the definition of a discrete analogue.
Since there are infinitely many inequivalent choices of $P_G(x)$, there are
infinitely many discrete analogues one may wish to try.

\subsection{Expressing the discrete Macdonald--Mehta integrals uniformly}

Another loose end concerns the question as to whether the six integral
evaluations of Table~\ref{Table_ten} corresponding to
$\mathrm{A}_{r-1}$, $\mathrm{B}_r$ and $\mathrm{D}_r$ can be expressed
in a single expression using only data coming from the 
underlying reflection group.
Obviously, each case contains the factor
\[
\prod_{i=1}^r \frac{\Gamma(1+d_i\gamma)}{\Gamma(1+\gamma)}
\]
(with $\{d_i\}=\{1,\dots,r\}$ for $\mathrm{A}_{r-1}$,
$\{2,4,\dots,2r\}$ for $\mathrm{B}_r$ and
$\{2,4,\dots,2r-2,r\}$ for $G=\mathrm{D}_r$),
since the discrete evaluations reproduce the Macdonald--Mehta integral
in the limit. We have however not been able to write the $n$-dependent factors
in a uniform manner.

\subsection{Missing $q$-analogues of $\mathcal S_{r,n}(\alpha,\gamma,\delta)$}
\label{sec:q}
We have obtained $q$-analogues for all evaluations listed in 
Table~\ref{Table_ten} except for $(\alpha,\gamma,\delta)$ given by 
$(2,\frac{1}{2},0)$ and $(1,1,1)$.
We can easily write down a $q$-analogue for the first of these
two cases (given in \eqref{Eq_Dr-sum}). Instead of 
$\sum_{\la\subseteq (r^n)} \so_{2n,\la}(1^n)$ we have to consider
\[
\sum_{\la\subseteq (r^n)}
\so_{2n,\la}\big(q^{1/2},q^{3/2},\dots,q^{n-1/2}\big).
\]
Closed-form expressions for the summand as well as the actual sum are
available in \eqref{Eq_so-even-PS-wrong} and \eqref{Eq_ortho-Okada-spec-1}.
However, neither of these completely factor.
A more natural $q$-analogue might result from summing
\[
\sum_{\la\subseteq (r^n)}
\ortho_{2n,\la}\big(q^{1/2},q^{3/2},\dots,q^{n-1/2}\big)
\]
(cf.~\eqref{Eq_ortho-PS-dual} for a fully factored expression for 
the summand). Unfortunately, we do not know a simple formula for the 
above character sum.

The problem of finding a $q$-analogue of our evaluation of 
$\mathcal{S}_{r,n}(1,1,1)$ (given in \eqref{eq:121} as well as 
\eqref{eq:S121rn}) lies with identities such as \eqref{eq:Gl}
used in the proof of Proposition~\ref{Prop_S121}.
It seems highly non-trivial to come up with an appropriate $q$-analogue 
of \eqref{eq:Gl} such that in the next step of our calculations
a $q$-analogue of Dixon's summation may be applied.
In any case, the form of the evaluation of $\mathcal{S}_{r,n}(1,1,1)$, 
with its inherent distinction between even and odd~$r$ values,
is an indication that this particular case is an outlier.

\subsection{Alternating sums}\label{sec:alt}

As a variation on the main theme of the paper, we also considered the
\emph{alternating} sums
\begin{equation}\label{Eq_Shat}
\Shat_{r,n}(\alpha,\gamma,\delta):=\sum_{k_1,\dots,k_r=-n}^n
\abs{\Delta(k^{\alpha})}^{2\gamma}\,
\prod_{i=1}^r (-1)^{k_i}\,\abs{k_i}^{\delta}\, \binom{2n}{n+k_i}.
\end{equation}
This differs from $\mathcal{S}_{r,n}(\alpha,\gamma,\delta)$ only in the
sign $\prod_{i=1}^r (-1)^{k_i}$, but importantly, does not have a 
continuous analogue. The sum \eqref{Eq_Shat} admits a
closed-form evaluation for all ten choices of $\alpha,\beta$ and $\gamma$
considered in Table~\ref{Table_ten}. (In some cases these evaluations
are simply~$0$.) 
Since in each case a suitable adaptation of the arguments leading to the 
evaluation of $\mathcal S_{r,n}(\alpha,\gamma,\delta)$ suffices,\footnote{For 
example, in the proofs in Section~\ref{sec:gamma=1/2} we have to insert 
$(-1)^{\abs\lambda}$ in the summands of the appropriate character sums, 
while in the derivations in Sections~\ref{sec:gamma=1-2}
and \ref{sec:gamma=1-1} one typically has to specialise one
of the indeterminates $d,e,f$ in \eqref{eq:Rains3} to $-\sqrt{aq}$.}
we refrain from presenting the corresponding identities and proofs here.
We remark that it is often possible to prove alternating versions 
of most of our parametric extensions and $q$-analogues as well.
As a typical example, we here just state one such result.

\begin{proposition}\label{Prop:S222qp}
Let $q$ be a real number with $0<q<1$.
For all non-negative integers $n$, $m$, and $p$, and a positive
integer~$r$, we have
\begin{multline*}
\sum_{k_1,\dots,k_r=-n}^n\; \prod_{1\leq i<j\leq r} 
[k_j-k_i]_q^2\,[k_i+k_j]_q^2 \\ \times
\prod_{i=1}^r (-1)^{k_i}q^{\frac{3}{2}k_i^2-(2i-\frac{1}{2})k_i} 
\Abs{[k_i]_{q^2}\,[k_i]_q}\,
\qbin{2n}{n+k_i}_q \qbin{2m}{m+k_i}_q \qbin{2p}{p+k_i}_q \\
=(-1)^{\binom {r+1}2}r!\Big(\frac{2}{[2]_q}\Big)^r\,q^{-\binom{r+1}{3}}
\prod_{i=1}^r \bigg(\frac{\Gamma_q(2n+1)\,\Gamma_q(2m+1)\,\Gamma_q(2p+1)} 
{\Gamma_q(n-i+1)\,\Gamma_q(m-i+1)\,\Gamma_q(p-i+1)} \\
\times \frac{\Gamma_q(i)\,\Gamma_q(n+m+p-i-r+2)} 
{\Gamma_q(n+m-i+2)\,\Gamma_q(m+p-i+2)\,\Gamma_q(p+n-i+2)}\bigg).
\end{multline*}
\end{proposition}

\begin{proof}
This follows by specialising $a=q^{-2n}$, $d=q^{-m-n}$, $e=q^{-p-n}$ and 
$f=q^{-n+1}$ in \eqref{eq:Rains3}.
\end{proof}

Sending $p$ to $\infty$ in Proposition~\ref{Prop:S222qp}, we obtain
\begin{multline*}
\sum_{k_1,\dots,k_r=-n}^n\; \prod_{1\leq i<j\leq r}
[k_j-k_i]_q^2\,[k_i+k_j]_q^2\\ \times
\prod_{i=1}^r (-1)^{k_i}q^{\frac{3}{2}k_i^2-(2i-\frac{1}{2})k_i} 
\Abs{[k_i]_{q^2}\,[k_i]_q}\, \qbin{2n}{n+k_i}_q \qbin{2m}{m+k_i}_q\\
=(-1)^{\binom {r+1}{2}}r!\Big(\frac{2}{[2]_q}\Big)^r\,q^{-\binom{r+1}{3}}
\prod_{i=1}^r \frac{\Gamma_q(2n+1)\,\Gamma_q(2m+1)\,\Gamma_q(i)} 
{\Gamma_q(n-i+1)\,\Gamma_q(m-i+1)\,\Gamma_q(n+m-i+2)}.
\end{multline*}
Upon letting $q\to 1$, dividing both sides by $\binom{2m}{m}^r$,
and finally also letting $m$ tend to $\infty$, we arrive at
\[
\Shat_{r,n}(2,1,2)=
\begin{cases}
(-1)^{\binom{r+1}{2}}r!\big((2r)!\big)^r &\text{ if $n=r$},\\[1mm]
0 &\text{ otherwise}.
\end{cases}
\]

\subsection{Additional character identities}\label{sec:char}

In Section~\ref{sec:gamma=1/2} we evaluated the sum 
$\mathcal{S}_{r,n}(\alpha,\tfrac{1}{2},\delta)$ 
using identities for classical group characters. 
Our evaluations of $\mathcal{S}_{r,n}(\alpha,1,\delta)$ in
Sections~\ref{sec:gamma=1-2} and \ref{sec:gamma=1-1} were entirely
different, relying on a transformation formula for elliptic hypergeometric 
series. It is nevertheless natural to wonder whether there are also
character identities hidden behind the $\gamma=1$ formulas.
The answer to this question is, at least partially, affirmative.
If one specialises all variables $x_i$ to $1$ in the identities
given in \cite[Theorem~2.2]{Okada98}, then one obtains \eqref{eq:S220q},
\eqref{eq:S222q} and
\eqref{eq:S120q} in the integer-$n$ case, \emph{all for $q=1$}.
We discovered this fact in a rather roundabout way as follows.
Helmut Prodinger suggested to the first author that non-intersecting
lattice paths may have a role to play in proving some of the discrete
Macdonald--Mehta integrals, an idea we initially discarded. 
Subsequently we realised that the combinatorics of non-intersecting 
lattice paths can indeed be used to prove the evaluations
of $\mathcal{S}_{r,n}(\alpha,1,\delta)$ for 
$(\alpha,\delta)\in\{(1,0),(2,0),(2,2)\}$.
However, we did not see how to use this approach to also prove 
corresponding $q$-analogues. Clearly, to obtain these one would have to
introduce appropriate $q$-weights for the non-intersecting lattice paths. 
By introducing weights, we however discovered numerous identities 
for classical group characters, which Soichi Okada quickly identified as
\cite[Theorem~2.2]{Okada98}. While we still do not see how to
specialise these identities appropriately to produce $q$-analogues, 
using our combinatorial machinery we did find one identity for
Proctor's odd symplectic characters \cite{ProcAF} missed by Okada.
The full details of this part of the story of discrete analogues of 
Macdonald--Mehta integrals will be presented in \cite{BrKWAB}.

\end{document}